\documentclass[14pt, reqno]{amsart}
\usepackage{tikz}
\usepackage{caption}
\usepackage{amsmath}
\usepackage{amssymb}
\usepackage{xcolor} 
\usetikzlibrary{decorations.pathreplacing, fit}
\usepackage[top=1.2in, bottom=1.2in, left=1.2in, right=1.2in]{geometry}
\newtheorem{theorem}{Theorem}[section]
\newtheorem{lemma}[theorem]{Lemma}

\theoremstyle{definition}

\theoremstyle{remark}

\usepackage{hyperref}
 \newtheorem{thm}{Theorem}[section]

 \theoremstyle{definition}
 
 \newtheorem{rem}[thm]{Remark}
 \numberwithin{equation}{section}

 \numberwithin{equation}{section}
\usepackage{amsmath,amssymb}
\usepackage{mathrsfs}
\def\leq{\leqslant}
\def\geq{\geqslant}
\allowdisplaybreaks[3]

\begin{document}
\title[Global existence of nonlinear Klein-Gordon equations in 3D]{Global Existence of classical solutions to 3D nonlinear Klein-Gordon equations with  low-regularity initial data}

\author{Wei Xu}
\address{School of Mathematics and Information Sciences, Nanchang Hangkong University, Nanchang, China}
\email{70709@nchu.edu.cn(W. Xu)}

\author{Yi Zhou}
\address{School of Mathematics Science, Fudan University, Shanghai 200433, China}
\email{yizhou@fudan.edu.cn(Y. Zhou)}




\begin{abstract}
This paper studies global existence for the Cauchy problem of nonlinear Klein–Gordon equations in three space dimensions, strictly within the regularity regime of classical local existence. For small initial data in $H^{s+1}\times H^s$ with $s>\frac52$, we prove global classical solutions via higher-order and lower-order energy estimates. The proof relies on two key ingredients. The first is Lemma \ref{product_KG}, due to Georgiev–Popivanov, which reduces quadratic nonlinearities to cubic terms plus ghost-energy-controllable terms, securing lower-order estimates. The second is a sharp Klainerman–Sobolev-type inequality without the scaling operator (Lemma \ref{sob_lemma}), established herein, which yields $(1+t)^{-1}$ decay for derivatives up to second order; integration by parts then controls derivative-loss terms in the higher-order estimates.
\end{abstract}

\maketitle

\tableofcontents

\section{\textbf{Introduction}}\label{introduction}
\subsection{Background and motivation}

Points in $\mathbb R^{+}\times\mathbb R^3$ are denoted by $(x^0,x^1,x^2,x^3)=(t,x)$. We write partial derivatives as $\partial_{\alpha}=\frac{\partial}{\partial x^{\alpha}}$, $\alpha=0,1,2,3$, with the abbreviations $\partial=(\partial_0,\partial_1,\partial_2,\partial_3)=(\partial_t,\nabla)$. Let $r=|x|$ and $\partial_r=\frac{x}r\cdot \nabla$.
For the spatial rotations, we set $\Omega=x\wedge \nabla$ and $\widetilde{\Omega}_{ij}=x^i\partial_j-x^j\partial_i$ for $i,j=1,2,3$. The Lorentz operators are defined by $L_i=t\partial_i+x^i\partial_t$ for $i=1,2,3$ and let the scaling operator $S=t\partial_t+r\partial_r$.

In this paper, we are concerned with the following Cauchy problem of nonlinear Klein-Gordon equations when the spatial dimension $d$ is 3
\begin{equation}\label{NLKG_FULL}
\begin{cases}
\Box u+u = F(u,\partial u,\partial^2u), (t,x)\in[0,+\infty)\times\mathbb R^d,\\
u(0,x)=\varphi(x),\ \partial_tu(0,x)=\psi(x),\ x\in\mathbb R^d.
\end{cases}
\end{equation}
The smooth nonlinearity $F(u,\partial u,\partial^2u)$ is at least quadratic and is linear in $\partial^2u$. Since nonlinear terms of order three and higher do not influence the results, they are neglected. Precisely,  
\[
      F(u,\partial u, \partial^2u)=\left(C^{\alpha\beta\gamma}\partial_{\gamma}u+C^{\alpha\beta}u\right)\partial_{\alpha}\partial_{\beta}u+Du^2+D^{\alpha}u\partial_{\alpha}u+D^{\alpha\beta}\partial_{\alpha}u\partial_{\beta}u,
\]
where $C^{\alpha\beta\gamma},C^{\alpha\beta},D,D^{\alpha\beta},D^{\alpha}$ are arbitrary real constants. Here and throughout, $\Box =\partial_t^2-\Delta$ denotes the usual d'Alembertian operator and $\Delta=\sum_{i=1}^3\partial_i^2$ denotes the laplacian operator. Roman indices $i,j,k$ are restricted to \{1,2,3\}, while Greek indices $\alpha,\beta,\gamma$ range over \{0,1,2,3\}, with 0 corresponding to the time variable. The Einstein summation convention is adopted for repeated indices.

A straightforward calculation then yields
\begin{align*}
    \Box u+u=\frac{1}{1-C^{00\gamma}\partial_{\gamma}u-C^{00}u}\Big((&C^{00\gamma}\partial_{\gamma}u+C^{00}u)(\Delta u-u)+(C^{ i\alpha\gamma }\partial_{\gamma}u+C^{i\alpha}u)\partial_{i}\partial_{\alpha} u
    \\&+(C^{0i\gamma}\partial_{\gamma}u+C^{0i }u)\partial_{i}\partial_t u+Du^2+D^{\alpha}u\partial_{\alpha}u+D^{\alpha\beta}\partial_{\alpha}u\partial_{\beta}u\Big). 
\end{align*}
By Taylor's expansion and neglecting cubic nonlinear terms which do not affect our subsequent discussion, we may write
\begin{align*}
F(u,\partial u,\partial^2u)=\left(\widetilde{C}^{\alpha\beta\gamma}\partial_{\gamma}u+\widetilde{C}^{\alpha\beta}u\right)\partial_{\alpha}\partial_{\beta}u+\widetilde{D}u^2+\widetilde{D}^{\alpha}u\partial_{\alpha}u+\widetilde{D}^{\alpha\beta}\partial_{\alpha}u\partial_{\beta}u,
\end{align*}
where $\widetilde{C}^{00\gamma}=\widetilde{C}^{00}=0$. Without loss of generality, we therefore assume
\begin{align}\label{NLG1}
\Box u+u=\left(C^{\alpha i\gamma}\partial_{\gamma}u+C^{\alpha i}u\right)\partial_{\alpha}\partial_{i}u+H(u),
\end{align}
with the semilinear term
\[
H(u)=Du^2+D^{\alpha}u\partial_{\alpha}u+D^{\alpha\beta}\partial_{\alpha}u\partial_{\beta}u
\]
and the constants $C^{i\alpha\gamma},C^{i\alpha}$ satisfying the symmetry conditions
\[
C^{i\alpha\gamma}=C^{\alpha i\gamma},\ C^{i\alpha}=C^{\alpha i}.
\]

Let us define the set of vector fields $Z=(\partial,\Omega,L):=(Z_1,\ldots,Z_{10})$, where $L=(L_1,L_2,L_3)$. Then the commutator $[\partial,Z]$ is zero or in the span of $\partial$, while the commutator of any two elements of $Z$ again belongs to the span of $Z$. Moreover, every element of $Z$ commutes with the d'Alembertian $\Box$. By $Z^a$, $a=(a_1,\cdots,a_{10})$, we may define an ordered product of $|a|=\sum_{i=1}^{10}a_i$ vector fields $Z_{1}^{a_1}\cdots Z_{10}^{a_{10}}$. For any positive integer $k$, any constant $p$ with $1\leq p\leq+\infty$, and any function $f$, let
\[
\|Z^{k}f\|_{L^p}:=\sum_{|a|= k}\|Z^{a}f\|_{L^p},\quad \|Z^{\leq k}f\|_{L^p}:=\sum_{|a|\leq k}\|Z^{a}f\|_{L^p}.
\]

We now recall some basic results concerning the nonlinear Klein-Gordon equation \eqref{NLKG_FULL}.

The Cauchy problem for nonlinear Klein-Gordon equations is a central topic in the theory of nonlinear hyperbolic partial differential equations, with profound applications in quantum field theory, fluid dynamics and geometric analysis.  In the three‑dimensional case (i.e., four‑dimensional spacetime), the global existence of classical solutions under small initial data has been extensively studied over the past four decades. The seminal papers \cite{Klainerman1985, Shatah1985} introduced, respectively, the hyperboloidal (vector field) method and the normal form method to prove global existence of small-amplitude solutions for nonlinear Klein–Gordon equations in three space dimensions. For the small initial data in the suitably weighted Sobolev spaces on the product space $\mathbb R^2\times \mathbb T$, the global existence of solutions to the quasilinear Klein-Gordon equations with quadratic nonlinearities in three space dimensions was established in \cite{LiTaoYin2023}. For further references in this setting $d=3$, see \cite{Delort1998,IonescuPausader2014,LiTaoYin2024,SimonTaflin,HouYin2026}. In particular, \cite{IonescuPausader2014} proved small-data global existence and scattering for quasilinear systems of Klein–Gordon equations with different speeds, which also yielded a robust global stability result for the Euler–Maxwell equations, as an application. The values of $s$ reported in the literature are as follows: $s=13$ in \cite{Klainerman1985}, $s=4+2a_0$ (with the positive constant $a_0$ sufficiently large) in \cite{Shatah1985}, $s=\frac{29}{2}$ in \cite{HouYin2026}, and $s=10^4$ in \cite{IonescuPausader2014}. In contrast, our result gives $s=\frac{5}{2}+\delta_0$, where $\delta_0$ is an arbitrary positive constant. Compared with the existing results, ours is significantly better.

In H\"ormander's analysis \cite{Hormander}, the solutions to inhomogeneous Klein–Gordon equations are shown to exhibit $t^{-d/2}$ time decay, which yields the global existence for $d\geq3$, the almost global lifespan bounds for $d=2$ and the lifespan lower bounds $T_{\varepsilon}\geq C/(\varepsilon^2)$ for $d = 1$. H\"ormander further conjectured that the optimal lifespan bounds should be much stronger, namely global existence in $d = 2$ and almost global existence in $d = 1$. Ozawa, Tsutaya and Tsutsumi \cite{OzawaTsutayaTsutsumi1996}, were the first to prove H\"ormander's conjecture in the case $d=2$ for quadratic semilinear nonlinearities $\mathcal{O}(u^2+(\partial_tu)^2+|\nabla u|^2)$, building upon earlier partial results of \cite{Georgiev,Kosecki}. Notably, Georgiev–Popivanov \cite{Georgiev} established the global existence for semilinear Klein-Gordon equations with quadratic nonlinear terms proportional to $(\partial_tu)^2-(\partial_1u)^2-(\partial_2u)^2+u^2$ in $d=2$. Their proof employed a technique similar to that in Lemma \ref{product_KG}, which is one of the key ingredients in our proof. In a subsequent work \cite{OzawaTsutayaTsutsumi1997}, Ozawa, Tsutaya and Tsutsumi extended the analysis to the quasilinear setting, establishing both a global existence theory and a classical scattering result. For additional references in two space dimensions, see \cite{Delort1998,SimonTaflin,HouYin2026}. The case $d=1$ has been thoroughly investigated in  \cite{Delort1997,Delort2001,Stingo2018,HouTaoYin2025}, where long-time existence results for small initial data are established, sometimes under a null condition. 

For classical solutions of nonlinear Klein–Gordon equations, the required Sobolev regularity index \(s\) plays a critical role. 
The classical theory demands sufficiently high regularity to control the nonlinear terms in energy estimates. In the general setting of small initial data in \(H^{s+1}\times H^s\), local existence of classical solutions is guaranteed when \(s>\frac52\),  for symmetric hyperbolic systems in $d=3$. The threshold \(\frac52\) arises from the embedding requirements. For \(s>\frac52\), Lemma \ref{sob_lemma} yields that the $L^{\infty}$-norms of derivatives up to order 2 decay like $(1+t)^{-1}$  for all $t$.

\subsection{Main results}
This paper establishes the global existence of classical solutions to the Cauchy problem of nonlinear Klein--Gordon equations \eqref{NLG1} in three space dimensions, under the assumption that the initial data are small and belong to \(H^{s+1}\times H^s\) with \(s>\frac52\). We define
\begin{align*}
&u_j=P_ju,\ u_{j,k}=R_kP_ju,
\end{align*}
where $P_j$ denotes the Littlewood–Paley projection on $\mathbb R^3$ and $R_k$ denotes the corresponding projection on the sphere $\mathbb S^2$. The definitions of $P_j$ and $R_k$ are given in Section \ref{prelim sec}. For convenience, the symbols used throughout this paper are summarized in the following Table.
\begin{center}
    \begin{tabular}{| c | c |}
    \hline
    \textbf{Definition of Symbols}   \\[0.2ex] \hline
    $M_{low}(u(t))=M_{low1}(u(t))+M_{low2}(u(t))+M_{low3}(u(t))+M_{low4}(u(t))$ \\[0.2ex] \hline
    $M_{high}(u(t))=M_{high1}(u(t))+M_{high2}(u(t))+M_{high3}(u(t))+M_{high4}(u(t))+M_{high5}(u(t))$\\[0.2ex] \hline
    $M_{low1}(u(t))=\sum_{j=0}^{+\infty}(2^j)^{5+56\delta}\|\partial^{\leq1}u_j\|_{L^{2}(\mathbb R^3)}^2$\\[0.2ex] \hline
    $M_{low2}(u(t))=\sum_{j=0}^{+\infty}\sum_{k=j+20}^{+\infty}\Big((2^j)^{3+42\delta}(2^{k})^{2+2\delta}+(2^j)^{5+42\delta}(2^{k})^{10\delta}\Big)\|\partial^{\leq1}u_{j,k}\|_{L^{2}(\mathbb R^3)}^2$\\[0.2ex] \hline
    $M_{low3}(u(t))=\sum_{j=0}^{+\infty}(2^j)^{3+46\delta}\|\partial^{\leq1}(Zu)_j\|_{L^{2}(\mathbb R^3)}^2$\\[0.2ex] \hline
    $M_{low4}(u(t))=\sum_{j=0}^{+\infty}\sum_{k=j+20}^{+\infty}(2^j)^{3+30\delta}(2^{k})^{12\delta}\|\partial^{\leq1}(Zu)_{j,k}\|_{L^{2}(\mathbb R^3)}^2$\\[0.2ex] \hline
    $M_{high1}(u(t))=\sum_{j=0}^{+\infty}(2^j)^{5+64\delta}\|\partial^{\leq1}u_j\|_{L^{2}(\mathbb R^3)}^2$\\[0.2ex] \hline
    $M_{high2}(u(t))=\sum_{j=0}^{+\infty}\sum_{k=j+20}^{+\infty}(2^j)^{3+50\delta}(2^k)^{2+2\delta}\|\partial^{\leq1}u_{j,k}\|_{L^{2}(\mathbb R^3)}^2$\\[0.2ex] \hline
    $M_{high3}(u(t))=\sum_{j=0}^{+\infty}(2^j)^{3+54\delta}\|\partial^{\leq1}(Zu)_{j}\|_{L^{2}(\mathbb R^3)}^2$\\[0.2ex] \hline
    $M_{high4}(u(t))=\sum_{j=0}^{+\infty}\sum_{k=j+20}^{+\infty}(2^j)^{1+40\delta}(2^k)^{2+2\delta}\|\partial^{\leq1}(Zu)_{j,k}\|_{L^{2}(\mathbb R^3)}^2$\\[0.2ex] \hline
    $M_{high5}(u(t))=\sum_{j=0}^{+\infty}(2^j)^{1+42\delta}\|\partial^{\leq1}(Z^2u)_j\|_{L^{2}(\mathbb R^3)}^2$\\[0.2ex] \hline
    \end{tabular}
    \captionof{table}{List of symbols used}\label{Table1}
\end{center}

We now state precisely the main result of this paper.
\begin{theorem}\label{main_theorem}
Let $\delta$ be any constant in $(0,0.001]$, and let $M_0$ be any positive constant. 
Assume that the initial data, compactly supported in the ball of radius $M_0$, satisfy 
\[
M_{high}(u(0))\leq \widetilde{C}\varepsilon^2,
\]
for some fixed positive constant $\widetilde{C}$, where $M_{high}(u(t))$ is defined in table \ref{Table1}. If $\varepsilon$ in $(0,\delta^6]$ is sufficiently small, depending on $\delta$ and $M_0$, then system \eqref{NLKG_FULL} admits a unique global solution on $[0,+\infty)\times\mathbb R^3$.
\end{theorem}
\begin{rem}
First, the solution $u(t,\cdot)$ of system \eqref{NLKG_FULL} also owns compact support in the ball of radius $t+M_0$ for any positive time $t$. Second, we establish global existence strictly within the regularity regime required for classical local existence. 
\end{rem}

\begin{rem}
This result can be applied to construct global perturbed solutions to the irrotational Euler–Poisson system in three space dimensions, as in \cite{GermainMasmoudiPausader2013,Guo1998}; for the two-dimensional counterpart, see \cite{LiWu2014,IonescuPausader2013}. Moreover, motivated by \cite{Zheng}, our approach can also be applied to the irrotational gravity water wave equations in three space dimensions.
\end{rem}

We now briefly outline the proof of Theorem \ref{main_theorem}. Throughout this paper, we impose the following bootstrap assumption:
\begin{align}\label{bootstrap1}
M_{low}(u(t)) \leq 2C_0\varepsilon^2,
\end{align}
where $C_0$ is a positive constant such that inequalities \eqref{C_0_1},\eqref{C_0_2},\eqref{C_0_3},\eqref{C_0_4},\eqref{C_0_5} hold. Morally speaking, the key idea is based on a bootstrap argument and the main steps are as follows:
\begin{itemize}
  \item \textbf{step1}: Assuming \eqref{bootstrap1}, we establish the higher-order energy estimate 
  \begin{align}\label{bootstrap2}
  M_{high}(u(t)) \leq &C\varepsilon^2(1+t)^{\delta^2}.
  \end{align}
  The proof is given in Section \ref{high_sec}.
  \item \textbf{step2}: We show that the lower-order energy estimate \eqref{bootstrap1} remains valid with the constant $2C_0$ replaced by $C_0$ in Section \ref{sec_low}.
\end{itemize}
By the continuity argument, 
\[
M_{low}(u(t)) \leq C_0\varepsilon^2.
\]
Hence we can extend the local solution of system \eqref{NLKG_FULL} to a global one by the local existence theory.

Two key ingredients underpin the argument.
\begin{itemize}
\item The first is Lemma \ref{product_KG}, due to Georgiev–Popivanov \cite{Georgiev}, which allows us to reduce the quadratic nonlinear terms to a sum of cubic terms and favorable terms that are controllable via ghost energy estimates in the course of the lower-order energy estimates.
\item The second is Lemma \ref{sob_lemma}, which provides a sharp version of the Klainerman–Sobolev-type inequality without the scaling operator. This lemma, established via the Littlewood–Paley decompositions on $\mathbb R^3$ and the sphere $\mathbb S^2$, implies that the $L^{\infty}$-norm of $\partial^{\leq2}u$ decays at the rate $(1+t)^{-1}$. Combined with integration by parts, this decay estimate yields favorable bounds for the derivative-loss terms in the higher-order energy  estimates.
\end{itemize}

Substantial effort was devoted to defining the exponents of the weight for the low‑order and high‑order energies.

(1) Multiplying quadratic nonlinear terms by a multiplier produces cubic nonlinear terms, whose integral over the whole space $\mathbb R^3$ is considered. To apply Gronwall's inequality for deriving higher-order energy estimates $M_{high}(u(t))$ as in Section \ref{high_sec}, this integral must be bounded by 
\[
C(1+t)^{-1}(M_{low}(u(t)))^{\frac12}M_{high}(u(t)),
\]
which requires that the weights from the lower-order energy norms dominate those from the higher-order energy norms:
\begin{itemize}
\item  For the terms involving $\left(C^{i\alpha\gamma}\partial_{\gamma}Zu+C^{i\alpha}Zu\right)\partial_{\alpha}\partial_iu$ in estimating $M_{high3}(u(t))$, we choose the weight $(2^j)^{54\delta}$ (from $M_{high3}(u(t))$) to be smaller than $(2^j)^{56\delta}$ (from $M_{low1}(u(t))$), thereby ensuring that the high‑frequency component of  $\partial_{\alpha}\partial_iu$ is acceptable as in \eqref{difficulty1}.
\item For the terms involving $\left(C^{i\alpha\gamma}\partial_{\gamma}Zu+C^{i\alpha}Zu\right)\partial_{\alpha}\partial_i Zu$ in estimating $M_{high5}(u(t))$, we take the weight $(2^j)^{42\delta}$ in $M_{high5}(u(t))$ smaller than $(2^j)^{46\delta}$ in $M_{low3}(u(t))$, so that the high‑frequency part of  $\partial_{\alpha}\partial_iZu$ is acceptable as in \eqref{difficulty2}.
\item In the estimation of $M_{high2}(u(t))$, the following index conditions are required: for $IV_{2,1,1,1}^{\alpha=0}$, the exponent $56\delta$ in $M_{low1}(u(t))$ must be bigger than $50\delta$ in $M_{high2}(u(t))$; meanwhile, in case 1 for $IV_{2,1,2}^{\alpha=0}$, the combined exponent $42\delta+10\delta$ (from $(2^j)^{42\delta}(2^k)^{10\delta}$ in $M_{low2}(u(t))$) is required to be no less than $50\delta+2\delta$ (from $(2^j)^{50\delta}(2^k)^{2\delta}$ in $M_{high2}(u(t))$).
\item  In estimating $M_{high4}(u(t))$, for the terms $V_{3,1}$ and $V_{3,2},V_{2,1,1,1}^{\alpha=0}$, the factors $42\delta$ from $(2^j)^{42\delta}$ in $M_{low2}(u(t))$ and $46\delta$ from $(2^j)^{46\delta}$ in $M_{low3}(u(t))$ must both be larger than $40\delta$ from $(2^j)^{40\delta}$  in $M_{high4}(u(t))$, respectively. Moreover, the combined exponent $30\delta+12\delta$ from $(2^j)^{30\delta}(2^k)^{12\delta}$ in $M_{low4}(u(t))$ must
    be no less than $40\delta+2\delta$ from $(2^j)^{40\delta}(2^k)^{2\delta}$ in $M_{high4}(u(t))$ to ensure that the terms $V_{2,1,2}^{\alpha=0}$ are admissible. 
\end{itemize}

(2) Taking $M_{low1}(u(t))$ as an example, to achieve additional time decay, we invoke Lemma \ref{product_KG}, which introduces extra second‑order derivatives. Moreover, for the estimate of $(N_{1,j}(u(t)))^{1/\delta}$, where $N_{1,j}(u(t))$ is essentially equal to $\|\partial^{\leq1}u_j\|_{L^2}^2$, we multiply both sides of the ghost energy estimate \eqref{low1_per} by the factor
\[
\frac{1}{\delta}(N_{1,j}(u(t)))^{\frac{1}{\delta}-1},
\]
which combining the higher-order energy estimate $M_{high1}(u(t))$ produces 
\[
\left((2^{j})^{-5-64\delta}\right)^{\frac1{\delta}-1}=(2^j)^{-\frac{5}{\delta}-64+5+64\delta}.
\] 
To ensure the acceptability of the $M_{low1}(u(t))$ estimate as in \eqref{low2}, the exponent $64\delta$ from $(2^j)^{64\delta}$ in $M_{high1}(u(t))$ must exceed $7\delta$ plus the exponent $56\delta$ from $(2^j)^{56\delta}$ in $M_{low1}(u(t))$.
This procedure ultimately requires that the higher‑order energy norms dominate the corresponding lower‑order ones, as detailed in Section \ref{sec_low}: For the estimate of $M_{low3}(u(t))$ in \eqref{add11} to be acceptable, the exponent $54\delta$ from $(2^j)^{54\delta}$ in $M_{high3}(u(t))$ must exceed $7\delta+46\delta$, where $46\delta$ is from $(2^j)^{46\delta}$ in $M_{low3}(u(t))$; Likewise, for the estimate of $M_{low2}(u(t))$ in \eqref{cc3} to be valid, the exponent $50\delta$ from $(2^j)^{50\delta}$ in $M_{high2}(u(t))$ must be greater than $5.5\delta+42\delta+2\delta$, with $42\delta$ and $2\delta$ from $(2^j)^{3+42\delta}(2^k)^{2+2\delta}$ in $M_{low2}(u(t))$.

(3) In view of the foregoing discussion and with the choice of $64\delta$ as the exponent of the weight in $M_{high1}(u(t))$, the weight exponents for the remaining energies should be set as follows:
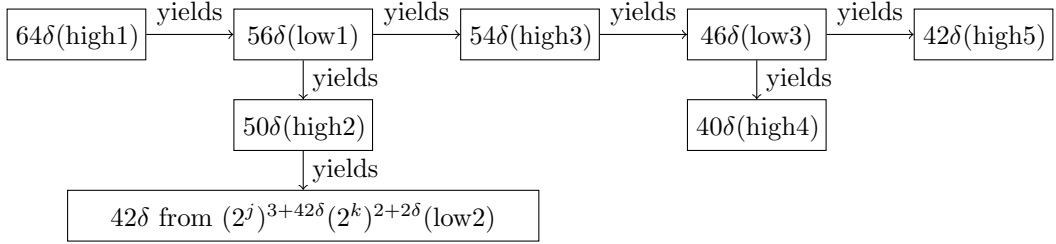
\begin{center}
\begin{tikzpicture}[    box/.style = {draw, rectangle, text width = 1.6cm, text height = 0.35cm, align = center},
 box2/.style = {draw, rectangle, text width = 6cm, text height = 0.35cm, align = center},
    arrow/.style = {->, >=Stealth, thick},
    brace/.style = {decorate, decoration={brace, amplitude=5pt}, thick}
]

\node[box] (decay) at (-6,0) {$64\delta$(high1) };
\node[box] (middle1) at (-3,0) {$56\delta$(low1)};
\node[box] (middle2) at (0,0) {$54\delta$(high3)};
\node[box] (middle3) at (3,0) {$46\delta$(low3)};
\node[box] (middle4) at (6,0) {$42\delta$(high5)};

\node[box] (new) at (-3,-1.2) {$50\delta$(high2)};
\node[box] (decay3) at (3,-1.2) {$40\delta$(high4)};

\node[box2] (growth2) at (-3,-2.4) {$42\delta$ from $(2^j)^{3+42\delta}(2^k)^{2+2\delta}$(low2) };

\draw[->] (decay.east) -- (middle1.west) node[midway, above] {\textcolor{black}{yields}};
\draw[->] (middle1.east) -- (middle2.west) node[midway, above] {\textcolor{black}{yields}};
\draw[->] (middle2.east) -- (middle3.west) node[midway, above] {\textcolor{black}{yields}};
\draw[->] (middle3.east) -- (middle4.west) node[midway, above] {\textcolor{black}{yields}};

\draw[->] (middle1.south) -- (new.north) node[midway, right] {\textcolor{black}{yields}};
\draw[->] (new.south) -- (growth2.north) node[midway, right] {\textcolor{black}{yields}};
\draw[->] (middle3.south) -- (decay3.north) node[midway, right] {\textcolor{black}{yields}};

\end{tikzpicture}
\end{center}
Lower-order energy norms: $M_{low4}(u(t))$ and the component containing the weight $(2^j)^{5+42\delta}(2^k)^{10\delta}$ in $M_{low2}(u(t))$, are derivable from the remaining energy norms. Their presence is indispensable for the proof. 

The major challenges encountered in the proof are the following. 

(1) Terms $[P_j,C^{i\alpha\gamma}\partial_{\gamma}u+C^{i\alpha}u]\partial_i\partial_{\alpha}u$ are introduced in the estimation of $M_{high1}(u(t))$ to address the derivative-loss terms. To treat the high‑frequency component of $\partial_i\partial_{\alpha}u$ in these terms, we appeal to Lemma \ref{para-product2}. An analogous treatment is valid for $M_{high2}(u(t)),M_{high3}(u(t)),M_{high4}(u(t)),M_{high5}(u(t))$. Lemma \ref{diff1} proves crucial for the estimates of $M_{high2}(u(t))$ and $M_{high4}(u(t))$, as it provides the necessary control over derivative‑loss terms, when the Littlewood-Paley operator on the sphere $\mathbb S^2$ is employed. 

(2)  For the terms involving $\left(C^{i\alpha\gamma}\partial_{\gamma}Z^2u+C^{i\alpha}Z^2u\right)\partial_{\alpha}\partial_iu$ in estimating $M_{high5}(u(t))$, we establish a radial-decay estimate on the product (see Lemma \ref{product_radial}) to estimate the high-frequency part of $\partial_{\alpha}\partial_iu$, by shifting $|\nabla|^{\frac12}$ from $\partial_{\alpha}\partial_iu$ to $C^{i\alpha\gamma}\partial_{\gamma}Z^2u+C^{i\alpha}Z^2u$ as in \eqref{difficulty3}.

(3) To obtain the $(1+t)^{-1-\delta_1}$ time decay rate (with $\delta_1>0$) in estimating $M_{low3}(u(t))$, we invoke the identity $(1+t)\partial_if=L_if-x^i\partial_tf+\partial_if$ and Lemma \ref{product_KG}. These yield the terms  $VII_{1,4,1,j_1,j_2}$. Since the summation of highest derivative order in terms $VII_{1,4,1,j_1,j_2}$ is $6$ and Sobolev embedding requires more than $3/2$ derivatives, the inequality $6+\frac32>7$, where the threshold $7$ plus certain multiples of $\delta$ originates from the definition of $M_{high}(u(t))$, implies that these terms become problematic when summed over $j_1,j_2$ in the regime where $2^j$ is the lowest frequency as in \eqref{combin2}. On the other hand, if we only aim for the $(1+t)^{-1+\delta^2}$ decay for terms $VII_{1,4,1,j_1,j_2}$, we may obtain a $(2^{j_1})^{-1-12\delta}$ decay in the frequency variable $2^{j_1}$ as in \eqref{combin1}. Therefore, we apply interpolation to balance the time and frequency decays, as shown in \eqref{cc1}.

(4) For terms $VII_{1,4,3}$ arising in the estimate of $M_{low3}(u(t))$, note that they may carry no spatial derivatives in $w_3,w_4$, and $\partial^{a}(Zu)_{j_1}$ with $|a|\leq2$ cannot produce time decay but can produce radial decay. The key point is to use \eqref{add10} to convert time decay into radial decay, which in turn generates the terms $VII_{1,4,3,3}$ and $VII_{1,4,3,j_1,j_2}$. The treatment of $VII_{1,4,3,j_1,j_2}$ is analogous to that of $VII_{1,4,1,j_1,j_2}$. Moreover, we employ integration by parts to transfer the operator $L$ from 
\[
P_j\left((\partial_t \partial^{\leq1}u_{j_1})(\partial_t \partial^{\leq1}(Zu)_{j_2})-(\nabla \partial^{\leq1}u_{j_1})\cdot (\nabla \partial^{\leq1}(Zu)_{j_2})+\partial^{\leq1}u_{j_1}\partial^{\leq1}(Zu)_{j_2}\right)
\]
to $(\partial_t+\partial_r)(Zu)_j$, which ensures that $VII_{1,4,3,3}$ is admissible.

(5) In the estimation of the lower-order energies $M_{low2}(u(t))$ and $M_{low4}(u(t))$, we employ the inequality $\|R_kf\|_{L^2}\leq C\|f\|_{L^2}$ to eliminate the Littlewood-Paley operator on the sphere, so that the ghost-energy estimates for $M_{low2}(u(t))$ and $M_{low4}(u(t))$ coincide with those for $M_{low1}(u(t))$ and $M_{low3}(u(t))$, respectively. However, with this reduction, for the part of $M_{low2}(u(t))$ involving the weight $(2^j)^{3+42\delta}(2^k)^{2+2\delta}$, a direct application of $M_{high2}(u(t))$ would lead to insufficient decay in the weight $2^k$. To overcome this, we derive 
\[
\|\partial^{\leq1}u_{j,k}\|_{L^{2}}^2\sim (2^j)^{-3-50\delta}(2^k)^{-2-2\delta},
\]
from $M_{high2}(u(t))$ and
\[
\|\partial^{\leq1}u_{j,k}\|_{L^{2}}^2\sim (2^j)^{-1-40\delta}(2^k)^{-4-2\delta},
\]
from $M_{high4}(u(t))$ (Here we employed the estimate $\|\Omega R_kf\|_{L^2}\sim 2^k\|R_kf\|_{L^2}$, for $k\geq1$), and then combine them via interpolation to obtain 
\[
\|\partial^{\leq1}u_{j,k}\|_{L^{2}}^2\sim (2^k)^{-2.1-2\delta}(2^j)^{-2.9-49.5\delta},
\]
thereby rendering the energy estimate for this component admissible. For the part of $M_{low2}(u(t))$ with weight $(2^j)^{5+42\delta}(2^k)^{10\delta}$, an analogous argument using $M_{high1}(u(t))$ and $M_{high2}(u(t))$, together with interpolation, yields the desired estimate. Finally, for $M_{low4}(u(t))$, the same interpolation strategy using $M_{high3}(u(t))$ and $M_{high4}(u(t))$ applies.

The remainder of this paper is organized as follows. In Section \ref{prelim sec}, we present several useful estimates, including the radial-decay and time-decay Sobolev embedding theorems. In Section \ref{sec3}, by a continuity argument, we show that the higher-order energy estimates are bounded by $C\varepsilon^2(1+t)^{\delta^2}$, while the lower-order energy estimates are uniformly bounded by $C\varepsilon^2$, thereby completing the proof of Theorem \ref{main_theorem}. Finally, in Appendix \ref{appendix}, we collect some useful lemmas concerning the Littlewood–Paley decompositions on $\mathbb R^3$ and on the sphere $\mathbb S^2$, which are employed in Section \ref{sec3}.

\section{\textbf{Preliminaries}}\label{prelim sec}
In this paper, we use $C$ to denote a generic positive constant which may vary from line to line (unless otherwise stated).
We now present some useful formulae that will be used in the proof of Theorem \ref{main_theorem}. 
The first is the following Hardy inequality, proved in \cite[Theorem 2.57]{Bahouri}.
\begin{lemma}\label{Hardy}
For any $s$ in $[0,\frac32)$ and any function $f(x)$ with the right-hand side finite, the following Hardy inequality holds
\[
\int_{\mathbb R^3}\frac{|f(x)|^2}{|x|^{2s}}dx\leq C\|f\|_{\dot{H}^{s}(\mathbb R^3)}^2.
\]
Here the homogeneous Sobolev norm is defined by
\begin{align*}
\|v\|_{\dot{H}^s(\mathbb R^3)}=\||\xi|^sF[v]\|_{L^2(\mathbb{R}_{\xi}^3)},
\end{align*}
where $F[v](\xi)=\int_{\mathbb R^3}e^{-i\langle x,\xi\rangle}v(x)dx$.
\end{lemma}

The following lemma on the product of solutions to the Klein–Gordon equations is key to proving the main theorem.
\begin{lemma}\label{product_KG}
For any scalar functions $w_1$ and $w_2$, the following identity holds:
\begin{align}\label{KG_pro3}
-3\partial(w_1 w_2)=&\partial[\Box(w_1 w_2)+w_1 w_2]-\partial\left(w_1(\Box  w_2+ w_2)\right)-\partial\left((\Box w_1+w_1) w_2\right)
\nonumber\\&-2\partial(\partial_tw_1\partial_tw_2-\nabla w_1\cdot\nabla w_2+w_1 w_2),
\end{align}
along with the point-wise estimate
\begin{align}\label{KG_pro4}
&|-2\partial(\partial_tw_1\partial_tw_2-\nabla w_1\cdot\nabla w_2+w_1 w_2)|
\nonumber\\ \leq &\frac{C}{\langle t-r\rangle}\left(|\partial w_1\partial Z^{\leq 1} w_2|+|\partial Z^{\leq 1} w_1\partial w_2|+|Zw_1||w_2|+|w_1||Zw_2|\right)
\nonumber\\&+\frac{C}{\langle t-r\rangle}\left(|Sw_2(\Box w_1+w_1)|+|Sw_1(\Box w_2+w_2)|\right),
\end{align}
where $\langle t-r\rangle=\sqrt{1+|t-r|^2}$.
\end{lemma}
\begin{proof}
For the proof, we refer to \cite{Georgiev}; for completeness, we include the details here. A direct computation yields
\[
\partial_{\alpha}^2(w_1w_2)=w_1\partial_{\alpha}^2w_2+w_2\partial_{\alpha}^2w_1+2\partial_{\alpha}w_1\partial_{\alpha}w_2,\ \alpha=0,1,2,3.
\]
Note that here, repeated index does not denote summation. The identity $\Box(w_1w_2)=\partial_t^2(w_1w_2)-\sum_{i=1}^3\partial_i^2(w_1w_2)$ gives identity \eqref{KG_pro3}. We now focus on the last term on the right-hand side of \eqref{KG_pro3}.
\begin{itemize}
\item When $\partial$ is taken to be $\partial_t$, the identity
\begin{align}\label{cc5}
x^iL_i=t(S-t\partial_t)+r^2\partial_t=tS-(t^2-r^2)\partial_t,
\end{align}
yields
\[
(t-r)\partial_t=\frac{tS-x^iL_i}{t+r},
\]
which, together with $\langle t-r\rangle=\sqrt{1+|t-r|^2}$, implies the estimate 
\begin{align}
\label{xw4}&\langle t-r\rangle|2\partial_t(\partial_tw_1\partial_tw_2-\partial_iw_1\partial_iw_2+w_1 w_2)|
\nonumber\\ \leq& C|\partial_tw_1S\partial_tw_2-\partial_iw_1S\partial_iw_2+w_1S w_2+S\partial_tw_1\partial_tw_2-S\partial_iw_1\partial_iw_2+Sw_1 w_2|
\nonumber\\&+C\left(|\partial w_1Z\partial w_2|+|Z\partial w_1\partial w_2|+|w_1||Zw_2|+|Zw_1||w_2|\right).
\end{align}
\item When $\partial$  is taken to be $\partial_k$, we need to consider the following terms
\[
|2\partial_k(\partial_tw_1\partial_tw_2-\partial_iw_1\partial_iw_2+w_1w_2)|.
\]
Noticing that
\begin{align}\label{lemm2_add}
x^j\widetilde{\Omega}_{kj}=x^k(S-t\partial_t)-r^2\partial_k=x^kS-tL_k+(t^2-r^2)\partial_k,
\end{align}
they satisfy the same bound as in inequality \eqref{xw4}.
\end{itemize} 
Direct calculation gives
\begin{align*}
&\partial_tw_1S\partial_tw_2-\partial_iw_1S\partial_iw_2+(Sw_1)w_2
\nonumber\\=&[(t\partial_tw_1)\partial_t^2w_2+(\partial_tw_1)x^j\partial_t\partial_jw_2]-[(\partial_iw_1)t\partial_i\partial_tw_2+(\partial_iw_1)x^j\partial_i\partial_jw_2]+(Sw_1)w_2
\nonumber\\=&[(t\partial_tw_1)\partial_t^2w_2+(\partial_tw_1)(L_j-t\partial_j)\partial_jw_2]
\nonumber\\&-[(\partial_iw_1)(L_i-x^i\partial_t)\partial_tw_2-(\partial_iw_1)(\widetilde{\Omega}_{ij}-x^i\partial_j)\partial_jw_2]+(Sw_1)w_2,
\end{align*}
which implies that 
\begin{align}
\label{kg_pro1}&|\partial_tw_1S\partial_tw_2-\partial_iw_1S\partial_iw_2+(Sw_1)w_2|\leq C|Sw_1(\Box w_2+w_2)|+C|\partial Z^{\leq 1}w_2||\partial w_1|,
\nonumber\\&|S\partial_tw_1\partial_tw_2-S\partial_iw_1\partial_iw_2+w_1(Sw_2)|
\leq C|Sw_2(\Box w_1+w_1)|+C|\partial Z^{\leq 1}w_1||\partial w_2|.
\end{align}
This implies that \eqref{KG_pro4} holds and completes the proof.
\end{proof}

The following Hardy-type inequality will be used in the process of $L^{\infty}$ estimates in Section \ref{sec_infi}.
\begin{lemma}\label{Hardy2}
For any $R\geq0$ and $0\leq s<1/2$, we have
\begin{align}
\left\|\frac{v(x)}{||x|-R|^s}\right\|_{L^2(\mathbb R^n)}\leq C\|v\|_{\dot{H}^s(\mathbb R^n)},
\end{align}
where $C$ is a positive constant independent of $R$ and $v$.
\end{lemma}
\begin{proof}
This lemma is proved in \cite{DAncona}; see Theorem 4.4 and Lemma 4.3 (1).
\end{proof}

We now introduce the Littlewood–Paley projections on $\mathbb R^3$ and the sphere $\mathbb S^2$.

Let $\overline{\chi}(x)$ be a smooth radial non-increasing function satisfying
\begin{align*}
\overline{\chi}(x)=1,\ when\ |x|\leq 1;\quad \overline{\chi}(x)=0,\ when\ |x|\geq 2.
\end{align*}
Define $\chi(x):=\overline{\chi}(x)-\overline{\chi}(2x)$. Set $\chi_0(x):=\overline{\chi}(x)$ and $\chi_j(x):=\chi\left(x/(2^j)\right)$ for $j\geq1$.
It follows that $supp\{\chi_j(x)\}\subseteq\{2^{j-1}\leq |x|\leq 2^{j+1}\}$ for $j\geq1$ and $\sum_{j=0}^{+\infty}\chi_j(x)\equiv1$.

The standard Littlewood-Paley operator $P_j$ is defined by
\[
\widehat{P_jf}(\xi):=\chi_j(\xi)\widehat{f}(\xi),\quad j\geq 0,
\]
where $\widehat{f}$ denotes the Fourier transform of $f$. Then $\sum_{j=0}^{+\infty}P_jf=f$ holds for any function $f$. For convenience, we set $P_jf=0$ if $j\leq -1$. Let
\[
P_{\leq j}:=\sum_{k=0}^jP_k,\quad \widehat{P_{\leq j}f}(\xi)=\overline{\chi}\left(\frac{\xi}{2^j}\right)\widehat{f}(\xi).
\]

Next, to introduce the Littlewood-Paley operator on the sphere, we recall the spherical Laplacian
\[
\Delta_{S^2}=\Omega\cdot \Omega.
\]
Using the polar coordinates
\begin{align}\label{polar}
x^1=r\cos\theta,\ x^2=r\sin\theta\cos\phi,\ x^3=r\sin\theta\sin\phi,
\end{align}
with $0\leq\theta\leq \pi$ and $0\leq\phi\leq2\pi$, we have
\begin{align}\label{Sph_Lap_def2}
\Delta_{S^2}=\frac1{\sin\theta}\frac{\partial}{\partial\theta}\left(\sin\theta\frac{\partial}{\partial\theta}\right)+\frac1{\sin^2\theta}\frac{\partial^2}{\partial\phi^2}.
\end{align}
Consider the eigenvalue problem $-\Delta_{S^2}\varphi=\mu\varphi$. Then the eigenvalues $\mu$ are given by $\mu_l=l(l+1), l=0,1,2,\cdots$, and for each $l$, there are $2l+1$ eigenfunctions $\varphi$, the spherical harmonics, denoted by $\varphi_{l,m}(\theta,\phi)$, $m=-l,-l+1,\cdots,0,1,\cdots,l$, which satisfy
\begin{align*}
\langle \varphi_{l,m},\overline{\varphi_{l^{'},m^{'}}}\rangle:=\int_{S^2}\varphi_{l,m}\overline{\varphi_{l^{'},m^{'}}}d_{S^2}=\delta_{ll^{'}}\delta_{mm^{'}},
\end{align*}
where $\overline{g}$ denotes the complex conjugate of $g$ and $\delta_{ll^{'}}$ is the Kronecker delta.
Here $d_{S^2}=\sin\theta d\theta d\phi$ denotes the surface measure on the sphere and we normalize $\varphi_{0,0}=\frac{1}{\sqrt{4\pi}}$.

The spherical harmonics $\varphi_{l,m}$ form a complete orthogonal basis on $L^2(S^2)$. Consequently, any function $f(\theta,\phi)$ defined on the sphere admits the expansion
\begin{align*}
f=\frac{1}{4\pi}\int_{S^2}f(\theta,\phi)d_{S^2}+\sum_{l=1}^{+\infty}\sum_{m=-l}^l\langle f,\overline{\varphi_{l,m}}\rangle\varphi_{l,m}, 
\end{align*}
where $\langle f,\overline{\varphi_{l,m}}\rangle=\int_{S^2}f\overline{\varphi_{l,m}}d_{S^2}$. We define the Littlewood-Paley decomposition on the sphere by
\begin{align}\label{Rk_def}
&R_0f:=\frac{1}{4\pi}\int_{S^2}f(\theta,\phi)d_{S^2},
\nonumber\\&R_kf:=\sum_{l=2^{k-1}}^{2^{k}-1}\sum_{m=-l}^l\langle f,\overline{\varphi_{l,m}}\rangle\varphi_{l,m},\ k=1,2,\cdots.
\end{align}
Then $f=\sum_{k=0}^{+\infty}R_kf$. Let $R_{\leq k}f:=\sum_{j=0}^kR_jf$ for any function $f$.

We now establish a sharp radial-decay  $L^{\infty}$ estimates, frequently used in Section \ref{sec3}
\begin{lemma}\label{radial_decay}
For any function $f$  with the right-hand side finite, we have
\[
\langle r\rangle|f|\leq C\sum_{j=0}^{+\infty}(2^j)^{3/2}\|P_jf\|_{L^2(\mathbb R^3)}+C\sum_{j=0}^{+\infty}\sum_{k=j+20}^{+\infty}(2^{j})^{1/2}2^k\|R_kP_jf\|_{L^2(\mathbb R^3)}.
\]
\end{lemma}
\begin{proof}
For any $k\geq1$, by using lemma \ref{ALP1},
\[
   r^2|R_kP_jf|^2\leq Cr^2\int_{S^2}|\Omega R_kP_jf|d_{S^2}\leq C\int_{\mathbb R^3}|\Omega R_kP_jf||\nabla\Omega R_kP_jf|dx,
\]
which combining Lemma \ref{equal_omiga}, Lemma \ref{ALP_commu} and H\"older inequality implies that
\[
   r^2|R_kP_jf|^2\leq C2^{j}\|\Omega R_kP_jf\|^2_{L^2}\leq C2^j(2^k)^2\|R_kP_jf\|^2_{L^2(\mathbb R^3)}.
\]
Similarly, 
\[
r^2|R_0P_jf|^2\leq C2^{j}\|P_jf\|_{L^2(\mathbb R^3)}^2.
\]
Thus
\[
r|f|\leq r\sum_{k=0}^{+\infty}\sum_{j=0}^{+\infty}|R_kP_jf|\leq C\left(\sum_{k=1}^{+\infty}\sum_{j=0}^{+\infty}(2^{j})^{\frac12}2^k\|R_kP_jf\|_{L^2(\mathbb R^3)}+\sum_{j=0}^{+\infty}(2^{j})^{\frac12}\|P_jf\|_{L^2(\mathbb R^3)}\right).
\]
We divide it into two cases: $k\geq j+20$ and $k<j+20$. In the case $r>1$, the estimate 
\[
\sum_{j=0}^{+\infty}\sum_{k=1}^{j+19}(2^{j})^{\frac12}2^k\|R_kP_jf\|_{L^2(\mathbb R^3)}\leq \sum_{j=0}^{+\infty}\sum_{k=1}^{j+19}(2^{j})^{\frac12}2^k\|P_jf\|_{L^2(\mathbb R^3)}\leq C\sum_{j=0}^{+\infty}(2^j)^{3/2}\|P_jf\|_{L^2(\mathbb R^3)},
\]
concludes the proof. When $r\leq1$, Bernstein's inequality shows that this lemma holds.
\end{proof}

We also need the $L^2$ estimates of the product below.
\begin{lemma}\label{product_radial}
For any $j\geq0$ and any functions $f,g$ with the right-hand side finite, we have
\begin{align*}
&\|r(P_jf)R_0g\|_{L^2}^2\leq C2^j\|P_jf\|^2_{L^2}\|g\|_{L^2}^2;
\\&\|r(P_jf)R_kg\|_{L^2}^2\leq C2^j(2^k)^2\|P_jf\|^2_{L^2}\|R_kg\|_{L^2}^2,\ for\ k\geq1.
\end{align*}
\end{lemma}
\begin{proof}
A straight calculation shows that
\begin{align*}
\|r(P_jf)R_0g\|_{L^2}^2 \leq &C\int_0^{+\infty}r^4\int_{S^2}|P_jf|^2d_{S^2}\int_{S^2}|g|^2d_{S^2}dr
 \leq C\int|P_jf||\nabla P_jf|dx\int|g|^2dx.
\end{align*}
For any $k\geq1$, using lemma \ref{ALP1} gives
\begin{align*}
\|r(P_jf)R_kg\|_{L^2}^2\leq &C\int_0^{+\infty}r^4\int_{S^2}|P_jf|^2d_{S^2}\int_{S^2}|\Omega R_kg|^2d_{S^2}dr
\leq C\int|P_jf||\nabla P_jf|dx\int|\Omega R_kg|^2dx.
\end{align*}
This completes the proof.
\end{proof}

Finally, we present another key ingredient in the proof of Theorem \ref{main_theorem}: the time-decay Sobolev embedding lemmas, namely the $L^{\infty}-L^2$ estimates.
\begin{lemma}\label{sob1}
For any scalar function $f$ and $k\geq1$, we get 
\begin{align}\label{sob2}
&(1+t)|R_kP_jf|\leq C(2^j)^{\frac32}\|R_kP_jf\|_{L^2(\mathbb R^3)}+C2^{\frac{j}2}\sum_{k_1=k-1}^{k+1}\sum_{j_1=j-1}^{j+1}\|P_{j_1}R_{k_1}Z^{\leq1}f\|_{L^2(\mathbb R^3)}, j\geq0,
\\&\label{sob3}(1+t)|R_0P_jf|\leq C2^{\frac j2}\sum_{j_1=j-1}^{j+1}\|P_{j_1}Z^{\leq1}f\|_{L^2(\mathbb R^3)}, \ for \ j\geq 1,
\\&\label{sob4}(1+t)|R_0P_0f|\leq C\sum_{j\leq 2}\|P_jZ^{\leq1}f\|_{H^1(\mathbb R^3)}+C\|\partial_tP_0f\|_{\dot{H}^{-1}(\mathbb R^3)}.
\end{align}
\end{lemma}

\begin{proof}
By using lemma \ref{ALP1}, for $k\geq1$, we have
\begin{align*}
t^2\|R_kP_jf\|^2_{L^{\infty}(S^2)}&\leq Ct^2\int_0^{+\infty}\int_{S^2}|\Omega R_kP_jf||\partial_r\Omega R_kP_jf|d_{S^2}dr
\\&\leq C\int_0^{+\infty}\left\|R_k\left(\frac tr\Omega P_jf\right)\right\|_{L^2(S^2)}\left\| R_k\left(\frac 1r\partial_r(t\Omega P_jf)\right)\right\|_{L^2(S^2)}r^2dr.
\end{align*}
A straight calculation shows that 
\begin{align*}
   &t\widetilde{\Omega}_{ij}=x^i(t\partial_j+x^j\partial_t)-x^j(t\partial_i+x^i\partial_t)=x^iL_j-x^jL_i,
   \\&\partial_r(t\widetilde{\Omega}_{ij}f)=x^i\partial_rL_jf-x^j\partial_rL_if+\frac{x^i}rL_jf-\frac{x^j}rL_if,
\end{align*}
which implies
\begin{align}
    \label{xw2}& \left\|R_k\left(\frac tr\Omega P_jf\right)\right\|_{L^2(S^2)}\leq C\sum_{k_1=k-1}^{k+1}\left\|R_{k_1}L P_jf\right\|_{L^2(S^2)},
    \\ \nonumber&\left\| R_k\left(\frac 1r\partial_r(t\Omega P_jf)\right)\right\|_{L^2(S^2)}\leq C\sum_{k_1=k-1}^{k+1}\left\|R_{k_1}\partial_ rL P_jf\right\|_{L^2(S^2)}+C\frac1r\sum_{k_1=k-1}^{k+1}\left\|R_{k_1}L P_jf\right\|_{L^2(S^2)}.
\end{align}
where we used Lemma \ref{xw1}. Thus by Hardy inequality (Lemma \ref{Hardy}), we have
\begin{align*}
t^2\|R_kP_jf\|^2_{L^{\infty}(S^2)}\leq C\sum_{k_1=k-1}^{k+1}\left\|R_{k_1}L P_jf\right\|_{L^2(\mathbb R^3)}\sum_{k_1=k-1}^{k+1}\left\|\nabla R_{k_1}L P_jf\right\|_{L^2(\mathbb R^3)}.
\end{align*}
Using $[L_m,P_j]f=\sum_{l=j-1}^{j+1}\widetilde{P}_{m,j}\partial_tP_lf$ (Lemma \ref{LP_commu}) and Lemma \ref{ALP_commu}, for $k\geq1$, we arrive at
\begin{align*}
t^2\|R_kP_jf\|^2_{L^{\infty}(S^2)}\leq &C\sum_{k_1=k-1}^{k+1}\left(\left\|R_{k_1} P_jLf\right\|_{L^2(\mathbb R^3)}+\sum_{m=1}^3\sum_{l=j-1}^{j+1}\left\|R_{k_1}F(\widetilde{P}_{m,j}\partial_tP_lf)\right\|_{L^2(\mathbb R^3)}\right)
\\&\times\sum_{k_1=k-1}^{k+1}\left(2^j\left\|R_{k_1} P_jLf\right\|_{L^2(\mathbb R^3)}+\sum_{m=1}^3\sum_{l=j-1}^{j+1}\left\||\xi|R_{k_1}F(\widetilde{P}_{m,j}\partial_tP_lf)\right\|_{L^2(\mathbb R^3)}\right).
\end{align*}
Combing Lemma \ref{xw1}, then inequality \eqref{sob2} holds when $t\geq1$. When $t\leq1$, employing Bernstein's inequality gives inequality \eqref{sob2}.

Now let's consider $R_0P_jf$. A straight calculation shows that
\begin{align*}
t^2(R_0P_jf)^2=t^2\left(\frac{1}{4\pi}\int_{S^2}P_jfd_{S^2}\right)^2\leq Ct^2\int_{S^2}|P_jf|^2d_{S^2}\leq Ct^2\int_0^{+\infty}\int_{S^2}|\partial_rP_jf||P_jf|d_{S^2}dr.
\end{align*}
Let $w=(w^1,w^2,w^3)$, where $w^i$ is the solution of $\Delta w^i=\partial_iP_jf$, for $i=1,2,3$. Then  $P_jf=\nabla\cdot w$. By using $t\partial_i=L_i-x^i\partial_t$, we have
\begin{align*}
t^2(R_0P_jf)^2
&\leq Ct^2\int_0^{+\infty}\int_{S^2}|\nabla P_jf||\nabla w|d_{S^2}dr
\\&\leq C\left(\int |LP_jf||L w|\frac{1}{r^2}dx+\int |LP_jf||\partial_tw|\frac{1}rdx+\int |\partial_tP_jf||Lw|\frac1r dx+\int |\partial_tP_jf||\partial_tw|dx\right).
\end{align*}
Since 
\[
\Delta L_kw^i=L_k \Delta w^i+2\partial_t\partial_kw^i=L_k \partial_iP_jf+2\partial_t\partial_kw^i=\partial_iL_k P_jf-\delta_{ki}\partial_tP_jf+2\partial_t\partial_kw^i,
\]
we have
\begin{align}\label{xw3}
\|\nabla Lw\|_{L^2(\mathbb R^3)}\leq C\|LP_jf\|_{L^2(\mathbb R^3)}+C\|\partial_tP_jf\|_{\dot{H}^{-1}(\mathbb R^3)}.
\end{align}
By using Hardy inequality (Lemma \ref{Hardy}) and inequality \eqref{xw3}, for $j\geq1$, we have
\[
 \int |LP_jf||L w|\frac{1}{r^2}dx\leq C\|\nabla L P_jf\|_{L^2(\mathbb R^3)}\|\nabla Lw\|_{L^2(\mathbb R^3)}\leq C2^j\|L P_jf\|^2_{L^2(\mathbb R^3)}+C\|\partial_tP_jf\|^2_{L^2(\mathbb R^3)},
\]
which implies that for $j\geq1$,
\begin{align*}
t^2(R_0P_jf)^2\leq C2^j\|L P_jf\|^2_{L^2(\mathbb R^3)}+C\|\partial_tP_jf\|^2_{L^2(\mathbb R^3)}.
\end{align*}
For the case $j=0$, we have
\[
t^2(R_0P_0f)^2\leq C\|L P_0f\|^2_{H^1(\mathbb R^3)}+C\|\partial_tP_0f\|^2_{\dot{H}^{-1}(\mathbb R^3)}+C\|\partial_tP_0f\|^2_{L^2(\mathbb R^3)}.
\]
Thus inequalities \eqref{sob3} and \eqref{sob4} hold when $t\geq1$. For the case $t<1$, by Hardy inequality, a straight calculation shows that
\begin{align*}
    |R_0P_jf|^2&\leq C\int_{S^2}|P_jf|^2d_{S^2} \leq C\int_0^{+\infty}\int_{S^2}|\nabla P_jf||P_jf|d_{S^2}dr \leq C\|\nabla^2 P_jf\|_{L^2}\|\nabla P_jf\|_{L^2}.
\end{align*}
This concludes the proof.
\end{proof}

\begin{lemma}\label{sob_lemma}
For any small constants $\varepsilon_1,\varepsilon_2$ in $(0,1)$ and any scalar function $f$ with the right-hand side finite, the following estimate holds
\begin{align*}
(1+t)^2\|f(t,\cdot)\|^2_{L^{\infty}(\mathbb R^3)}&\leq C\|\partial_tP_0f\|^2_{\dot{H}^{-1}(\mathbb R^3)}+C\frac{1}{(1-2^{-\varepsilon_1})^3}\sum_{j=0}^{+\infty}(2^{j})^{1+\varepsilon_1}\|P_jZ^{\leq1}f\|^2_{L^2(\mathbb R^3)}
\\&\quad+C\frac{1}{(1-2^{-\varepsilon_1})(1-2^{-\varepsilon_2})}\sum_{j=0}^{+\infty}\sum_{k=j+20}^{+\infty}(2^{j})^{1+\varepsilon_1}(2^{k})^{\varepsilon_2}\|P_jR_kZ^{\leq1}f\|^2_{L^2(\mathbb R^3)}.
\end{align*}
\end{lemma}
\begin{proof}
Using the decomposition $f=\sum_{k=0}^{+\infty}\sum_{j=0}^{+\infty}R_kP_jf$ and Lemma \ref{sob1} gives that
\begin{align*}
(1+t)|f(t,\cdot)|&\leq C\|\partial_tP_0f\|_{\dot{H}^{-1}(\mathbb R^3)}+C\sum_{j=0}^{+\infty}2^{\frac j2}\|P_jZ^{\leq1}f\|_{L^2(\mathbb R^3)}
\\&\quad+C\sum_{j=0}^{+\infty}\sum_{k=1}^{+\infty}2^{\frac{j}2}\sum_{k_1=k-1}^{k+1}\sum_{j_1=j-1}^{j+1}\|P_{j_1}R_{k_1}Z^{\leq1}f\|_{L^2(\mathbb R^3)},
\end{align*}
where the last term on the right-hand side of the above inequality can be bounded by
\begin{align*}
&\quad C\sum_{j=0}^{+\infty}\sum_{k=0}^{j+19}2^{\frac{j}2}\|P_{j}R_{k}Z^{\leq1}f\|_{L^2(\mathbb R^3)}+C\sum_{j=0}^{+\infty}\sum_{k=j+20}^{+\infty}2^{\frac{j}2}\|P_{j}R_{k}Z^{\leq1}f\|_{L^2(\mathbb R^3)}.
\end{align*}
Using Cauchy-Schwarz inequality and the estimates
\begin{align*}
\sum_{j=0}^{+\infty}(2^j)^{-\varepsilon_1}=\frac{1}{1-2^{-\varepsilon_1}},\quad \sum_{k=0}^{+\infty}(2^k)^{-\varepsilon_2}=\frac{1}{1-2^{-\varepsilon_2}},\ \sum_{j=0}^{+\infty}(j+20)^2(2^{-j})^{\varepsilon_1}\leq C\frac{1}{(1-2^{-\varepsilon_1})^3},
\end{align*}
we conclude the proof.
\end{proof}

\section{\textbf{Proof of Theorem \ref{main_theorem}}}\label{sec3}
In this section, we give the proof of theorem \ref{main_theorem}. Using the bootstrap assumption \eqref{bootstrap1}, Lemma \ref{sob1} and Lemma \ref{sob_lemma} yields the following point-wise $L^{\infty}$ estimates 
\begin{align}\label{small1}
&(1+t)\|\partial^{\leq2} u\|_{L^{\infty}}\leq \frac{C\varepsilon}{\delta^{3/2}},
\\&\label{small2}(1+t)\|\partial^{\leq2} u_j\|_{L^{\infty}}\leq \frac{C\varepsilon}{\delta}(2^{j})^{-\delta}.
\end{align}
Detailed proofs of inequalities \eqref{small1} and \eqref{small2} are provided in Section \ref{sec_infi}.

\subsection{Higher-order energy estimates}\label{high_sec}
We now estimate $M_{high}(u(t))$.
\subsubsection{\textbf{Estimates for $M_{high1}(u(t))$}}
Applying Lemma \ref{LP_commu} yields $[\Box,P_j]u=0$. Hence
\begin{align}\label{xw5}
\Box u_j+u_j=[P_j,C^{i\alpha\gamma}\partial_{\gamma}u+C^{i\alpha}u]\partial_i\partial_{\alpha}u+(C^{i\alpha\gamma}\partial_{\gamma}u+C^{i\alpha}u)\partial_i\partial_{\alpha}u_j+P_j(H(u)).
\end{align}
Multiplying both sides of equality \eqref{xw5} by $\partial_tu_j$, integrating over $\mathbb R^3$ and integration by parts gives
\begin{align*}
&\quad\frac{d}{dt}\Big(\frac12\|\partial^{\leq1}u_j\|^2_{L^2(\mathbb R^3)}+\frac12\int \left(C^{i\alpha\gamma}\partial_{\gamma}u+C^{i\alpha}u\right)\partial_i u_j\partial_{\alpha} u_jdx\Big)
\\&\leq -\int \left(C^{i\alpha\gamma}\partial_{i}\partial_{\gamma}u+C^{i\alpha}\partial_{i}u\right)\partial_{\alpha}u_j\partial_tu_jdx+\frac12\int \left(C^{i\alpha\gamma}\partial_{t}\partial_{\gamma}u+C^{i\alpha}\partial_{t}u\right)\partial_{i}u_j\partial_{\alpha}u_jdx
\nonumber\\&\qquad+\int\left([P_j,C^{i\alpha\gamma}\partial_{\gamma}u+C^{i\alpha} u]\partial_{i}\partial_{\alpha}u\right)\partial_tu_jdx+\int P_j(H(u))\partial_tu_jdx,
\end{align*}
It follows from inequality \eqref{small1} and Lemma \ref{para-product2} that
\begin{align}\label{aa2}
M_{high1}(u(T))\leq& CM_{high1}(u(0))+\frac{C\varepsilon}{\delta^3}\int_0^T(1+t)^{-1}M_{high1}(u(t))dt+I_1+I_2+I_3+I_4,
\end{align}
where 
\begin{align*}
&I_1=C\int_0^T\sum_{j=0}^{+\infty}(2^j)^{5+64\delta}\left|\int\left(\sum_{l=0}^{+\infty}\sum_{m=0}^{l+2} P_j\left(\left(P_{m}\partial_i \partial u\right)P_l\partial^{\leq1}u\right)\right)\partial_tu_jdx\right|dt,
\\&I_2=C\int_0^T\sum_{j=0}^{+\infty}(2^j)^{5+64\delta}\left|\int\left(\sum_{l=0}^{+\infty}\left(P_{\leq l+2}P_j\partial_i \partial u\right)P_l\partial^{\leq1}u\right)\partial_tu_jdx\right|dt,
\\&I_3=C\int_0^T\sum_{j=0}^{+\infty}(2^j)^{4+64\delta}\left|\int\left(\sum_{l=j-2}^{j+2}L(P_{\leq l-3}\partial^{\leq1}u,P_l\partial_i \partial u)\right)\partial_tu_jdx\right|dt,
\\&I_4=C\sum_{j=0}^{+\infty}(2^j)^{5+64\delta}\int_0^T\int\left|P_j(H(u))\partial_tu_j\right|dxdt.
\end{align*} 

We begin with the term $I_1$. Identity \eqref{van3} yields that a necessary condition for the non-vanishing of $\sum_{m=0}^{l+2}P_j\left(\left(P_{m}\partial_i \partial u\right)P_l\partial^{\leq1} u\right)$ is $j\leq l+4$. Observing the decomposition $\sum_{m=0}^{l+2}=\sum_{m=l-3}^{l+2}+\sum_{m=0}^{l-4}$, we find that the second sum is nonzero only if $l-2\leq j\leq l+2$. Hence we deduce that
\begin{align*}
I_1&\leq C\int_0^T\sum_{j=0}^{+\infty}\sum_{l=j-4}^{+\infty}\sum_{m=l-3}^{l+2}(2^{l})^{\frac 52+32\delta}\left(\|P_m\partial\nabla u\|_{L^{\infty}}\|P_l\partial^{\leq1} u\|_{L^2}\right)\left((2^{j})^{\frac 52+32\delta}\|\partial_tu_j\|_{L^2}\right)dt
\\&\quad+C\int_0^T\sum_{j=0}^{+\infty}\sum_{l=j-2}^{j+2}(2^{l})^{\frac 52+32\delta}\left(\left\|\sum_{m=0}^{l-4} P_m\partial\nabla u\right\|_{L^{\infty}} \|P_l\partial^{\leq1} u\|_{_{L^2}}\right)\left((2^{j})^{\frac 52+32\delta}\|\partial_tu_j\|_{L^2}\right)dt.
\end{align*}
Applying \eqref{small2} yields that$(1+t)\|P_m\partial\nabla u\|_{L^{\infty}(\mathbb R^3)}\leq \frac{C\varepsilon}{\delta}(2^m)^{-\delta}$. For any $\epsilon$ in $(0,1)$, we get
\[
      \sum_{l=j-4}^{+\infty}(2^l)^{-\epsilon}\leq \frac{C}{1-2^{-\epsilon}}(2^j)^{-\epsilon}\leq \frac{C2^{\epsilon}}{2^{\epsilon}-1}(2^j)^{-\epsilon}\leq \frac{C}{\epsilon}(2^j)^{-\epsilon}.
\]
Consequently, we obtain
\begin{align}\label{I1_Bound}
I_1\leq \frac{C\varepsilon}{\delta^3} \int_0^T (1+t)^{-1}M_{high1}(u(t))dt.
\end{align}

The term $I_2$ can be handled analogously to $I_1$. Identity \eqref{van1} shows that the non-vanishing of $P_{\leq l+2}P_j\partial_ig$ requires $l\geq j-4$. Since 
\begin{align*}
\sum_{l=j-4}^{+\infty}\|(P_{\leq l+2}P_j\partial_i\partial u)P_l\partial^{\leq1} u\|_{L^2}&\leq \frac{C\varepsilon}{\delta(1+t)}(2^j)^{-\delta}\left(\sum_{l=j-4}^{+\infty}(2^{-l})^{5+64\delta}\right)^{\frac12}\left(\sum_{l=j-4}^{+\infty}(2^{l})^{5+64\delta}\|P_l\partial^{\leq1}u\|^2_{L^2}\right)^{\frac12}
\\&\leq \frac{C\varepsilon}{\delta}(1+t)^{-1}(2^{-j})^{\frac52+33\delta}\left(M_{high1}(u(t))\right)^{1/2},
\end{align*}
then terms $I_2$ have the same bound as those for terms $I_1$ in \eqref{I1_Bound}.

Now let's deal with terms $I_3$. By H\"older inequality, it suffices to estimate
\begin{align*}
&\quad(2^j)^{\frac 32+32\delta}\sum_{l=j-2}^{j+2}\left\|\int_0^1\int_{\mathbb R^3}\check{\chi}(2^jy)(2^j)^3(2^jy)\cdot\left(\nabla P_{\leq l-3}\partial^{\leq1} u\right)(x-sy)\left(P_l\nabla\partial u\right)(x-y)dyds\right\|_{L_x^2}
\\&\leq C(2^j)^{\frac 32+32\delta}\sum_{l=j-2}^{j+2}\left\|\nabla P_{\leq l-3}\partial^{\leq1} u\right\|_{L^{\infty}}\left \|P_l\nabla\partial  u\right\|_{L^2}\int_0^1\int_{\mathbb R^3}\left|\check{\chi}(2^jy)(2^j)^3(2^jy)\right|dyds
\\&\leq C\sum_{l=j-2}^{j+2}(2^l)^{\frac 52+32\delta}\left\|\nabla P_{\leq l-3}\partial^{\leq1} u\right\|_{L^{\infty}}\left\|P_l\partial u\right\|_{L^2}.
\end{align*}
Using \eqref{small1} shows that terms $I_3$ have the same bound as terms $I_1$ given in \eqref{I1_Bound}.

Finally, we estimate the remaining term $I_4$. Using $P_j(H(u))=\sum_{l,m}P_j(\partial^{\leq1}u_l\partial^{\leq1} u_m)$, we split the summation into three cases according to which index among $j,l,m$ attains the minimum: $j=\min\{j,l,m\}$, $l=\min\{j,l,m\}$ or $m=\min\{j,l,m\}$. Lemma \ref{LP2} then yields the desired estimate.

From the above discussion, it follows that
\begin{align}\label{high1}
M_{high1}(u(T))\leq &CM_{high1}(u(0))+\int \frac{C\varepsilon}{\delta^3}(1+t)^{-1}M_{high}(u(t))dt.
\end{align}

\subsubsection{\textbf{Estimates for $M_{high3}(u(t))$}}

Applying $Z$ to equation \eqref{NLG1} gives
\begin{align}\label{aa1}
\Box (Zu)_j+(Zu)_j
=&P_j\left((C^{i\alpha\gamma}\partial_{\gamma}u+C^{i\alpha}u)\partial_{\alpha}\partial_iZu\right)+P_j\left((C^{i\alpha\gamma}\partial_{\gamma}Zu+C^{i\alpha}Zu)\partial_{\alpha}\partial_iu\right)
\nonumber\\&+P_j\left((\overline{C}^{\alpha\beta\gamma}\partial_{\gamma}u+\overline{C}^{\alpha\beta}u)\partial_{\alpha}\partial_{\beta}u\right)+P_j(Z(H(u))),
\end{align}
with constants $\overline{C}^{\alpha\beta\gamma}$ and $\overline{C}^{\alpha\beta}$. Multiplying \eqref{aa1} by $\partial_t(Zu)_j$ and integrating over $[0,T]\times\mathbb R^3$, analogously to \eqref{aa2}, we obtain,
\begin{align}\label{aa3}
M_{high3}(u(T))\leq &CM_{high3}(u(0))+\frac{C\varepsilon}{\delta^3}\int(1+t)^{-1}M_{high3}(u(t))dt+II_1+II_2+II_3+II_4,
\end{align}
where
\begin{align*}
&II_1=C\sum_{j=0}^{+\infty}(2^j)^{3+54\delta}\int_0^T\int\left([P_j,C^{i\alpha\gamma}\partial_{\gamma}u+C^{i\alpha}u]\partial_{i}\partial_{\alpha}Zu\right)\partial_t(Zu)_jdxdt,
\\&II_2=C\sum_{j=0}^{+\infty}(2^j)^{3+54\delta}\int_0^T\int\left|P_j\left(\partial^{\leq1} Zu\partial\nabla u\right)\partial_t(Zu)_j\right|dxdt,
\\&II_3=C\sum_{j=0}^{+\infty}(2^j)^{3+54\delta}\int_0^T\int\left|P_j\left(\partial^{\leq1} u\partial^2u\right)\partial_t(Zu)_j\right|dxdt,
\\&II_4=C\sum_{j=0}^{+\infty}(2^j)^{3+54\delta}\int_0^T\int\left|P_j\left(Z(\partial^{\leq1}u\partial^{\leq1} u)\right)\partial_t(Zu)_j\right|dxdt.
\end{align*}
We first consider terms $II_1$. Applying Lemma \ref{para-product2} gives
\[
II_1\leq II_{1,1}+II_{1,2}+II_{1,3},
\]
where
\begin{align*}
&II_{1,1}=C\int_0^T\sum_{j=0}^{+\infty}(2^j)^{3+54\delta}\left|\int\left(\sum_{l=0}^{+\infty}\sum_{m=0}^{l+2} P_j\left(\left(P_{m}\partial_i \partial Zu\right)P_l\partial^{\leq1}u\right)\right)\partial_t(Zu)_jdx\right|dt,
\\&II_{1,2}=C\int_0^T\sum_{j=0}^{+\infty}(2^j)^{3+54\delta}\left|\int\left(\sum_{l=0}^{+\infty}\left(P_{\leq l+2}P_j\partial_i \partial Zu\right)P_l\partial^{\leq1}u\right)\partial_t(Zu)_jdx\right|dt,
\\&II_{1,3}=C\int_0^T\sum_{j=0}^{+\infty}(2^j)^{2+54\delta}\left|\int\left(\sum_{l=j-2}^{j+2}L(P_{\leq l-3}\partial^{\leq1}u,P_l\partial_i \partial Zu)\right)\partial_t(Zu)_jdx\right|dt.
\end{align*}

Using the identities $\partial_t^2Zu=\Delta Zu-Zu+Z(\Box u+u)$, 
\begin{align}\label{space_div}
(1+t)\partial_i f=L_if-x^i\partial_tf+\partial_if,
\end{align}
Bernstein's inequality, Lemma \ref{LP_commu} and Lemma \ref{radial_decay}, we have
\begin{align*}
   (1+t)\|P_m\partial\nabla Zu\|_{L^{\infty}}
   &\leq C\|LP_m\partial Zu\|_{L^{\infty}}+\|\langle r\rangle P_m\partial^2 Zu\|_{L^{\infty}}
   \\&\leq \frac{C}{\delta}(2^{m})^{1-20\delta}\left(M_{high5}(u(t))+M_{high4}(u(t))+M_{high3}(u(t))\right)^{1/2}.
\end{align*}
By using a similar fashion of discussion to $I_1$ and the above inequality, we get
\begin{align*}
&\quad II_{1,1}
\\&\leq C\int_0^T\sum_{j=0}^{+\infty}\sum_{l=j-4}^{+\infty}\sum_{m=l-3}^{l+2}(2^{l})^{\frac 32+27\delta}\left(\|P_m\partial\nabla Zu\|_{L^{\infty}}\|P_l\partial^{\leq1} u\|_{L^2}\right)\left((2^{j})^{\frac 32+27\delta}\|\partial_t(Zu)_j\|_{L^2}\right)dt
\\&\quad+C\int_0^T\sum_{j=0}^{+\infty}\sum_{l=j-2}^{j+2}(2^{l})^{\frac 32+27\delta}\left(\left\|\sum_{m=0}^{l-4} P_m\partial\nabla Zu\right\|_{L^{\infty}} \|P_l\partial^{\leq1} u\|_{_{L^2}}\right)\left((2^{j})^{\frac 32+27\delta}\|\partial_t(Zu)_j\|_{L^2}\right)dt
\\&\leq \frac{C}{\delta^3}\int_0^T(1+t)^{-1}M_{high}(u(t))\left(M_{low1}(u(t))\right)^{\frac12}dt.
\end{align*}

Terms $II_{1,2}$ is similar to $II_{1,1}$ and we obtain
\begin{align*}
II_{1,2}&\leq C\int_0^T\left(\sum_{j=0}^{+\infty}\left((2^j)^{\frac32+27\delta}\sum_{l=j-4}^{+\infty}\|P_{\leq l+2}P_j\partial\nabla Zu\|_{L^2}\|P_l\partial^{\leq1} u\|_{L^{\infty}}\right)^2\right)^{\frac12}\left(M_{high3}(u(t))\right)^{\frac12}dt
\\&\leq \frac{C\varepsilon}{\delta^2}\int_0^T(1+t)^{-1}M_{high3}(u(t))dt.
\end{align*}

Now we arrive at terms $II_{1,3}$. As in the case of $I_3$,
\begin{align*}
II_{1,3}&\leq C\int_0^T\sum_{j=0}^{+\infty}(2^j)^{\frac 12+27\delta}\sum_{l=j-2}^{j+2}\left\|P_{\leq l-3}\partial^{\leq1}\nabla u \right\|_{L^{\infty}}\left\|P_l\partial\nabla  Z u\right\|_{L^2}(2^j)^{\frac32+27\delta}\|\partial_t(Zu)_j\|_{L^2}dt
\\&\leq \frac{C\varepsilon}{\delta^3}\int_0^T(1+t)^{-1}M_{high3}(u(t))dt.
\end{align*}

Now let's deal with terms $II_2$ as follows
\[
     II_2=C\sum_{j=0}^{+\infty}(2^j)^{3+54\delta}\int_0^T\int\left|\sum_{l,m=0}^{+\infty}P_j\left((\partial^{\leq1} Zu)_l(\partial\nabla u)_m\right)\partial_t(Zu)_j\right|dxdt.
\]
We consider the following three cases: \textbf{Case 1}: $l=\min\{j,l,m\}$; \textbf{Case 2}: $m=\min\{j,l,m\}$; \textbf{Case 3}: $j=\min\{j,l,m\}$. 

For \textbf{Case 3}, using Lemma \ref{LP2} and inequality \eqref{small2} shows that 
\begin{align*}
  II_2\leq&C\int_0^T\left(M_{high3}(u(t))\right)^{\frac12}\left(\sum_{j=0}^{+\infty}\left(\sum_{l=j}^{+\infty}\sum_{m=l-2}^{l+2}(2^l)^{\frac32+27\delta}\left\|\partial^{\leq1}(Zu)_l\right\|_{L^2}\left\|\partial\nabla u_m\right\|_{L^{\infty}}\right)^{2}\right)^{\frac12}dt
\\ \leq &\frac{C\varepsilon}{\delta^3}\int_0^T(1+t)^{-1}M_{high3}(u(t))dt.
\end{align*}

For \textbf{case 2}, using  Lemma \ref{LP2} and inequality \eqref{small1}, we get
\begin{align*}
II_2\leq\frac{C\varepsilon}{\delta^2}\int(1+t)^{-1}M_{high3}(u(t))dt.
\end{align*}

For \textbf{case 1}, using  Lemma \ref{LP2} and the identity \eqref{space_div} yields that
\begin{align}\label{difficulty1}
&II_2
\nonumber\\ \leq& C\sum_{j=0}^{+\infty}\sum_{m=j-2}^{j+2}\int_0^T(1+t)^{-1}\|P_{\leq j}\partial^{\leq1} Zu\|_{L^{\infty}}\left((2^{m})^{\frac32+27\delta}\|L(\partial u)_m\|_{L^2}\right)\left((2^j)^{\frac32+27\delta}\|\partial_t(Zu)_j\|_{L^2}\right)dt
\nonumber\\&+ C\sum_{j=0}^{+\infty}\sum_{m=j-2}^{j+2}\int_0^T(1+t)^{-1}\|\langle x\rangle P_{\leq j}\partial^{\leq1} Zu\|_{L^{\infty}}\left((2^{m})^{\frac32+27\delta}\|\partial^2u_m\|_{L^2}\right)\left((2^j)^{\frac32+27\delta}\|\partial_t(Zu)_j\|_{L^2}\right)dt
\nonumber\\ \leq &\frac{C\varepsilon}{\delta^2}\int(1+t)^{-1}\left(M_{high4}(u(t))+M_{high3}(u(t))\right)dt,
\end{align}
where we used Bernstein's inequality and lemma \ref{radial_decay} to get the following estimates
\begin{align*}
&\|P_{\leq j}\partial^{\leq1} Zu\|_{L^{\infty}}\leq \frac{C}{\delta}(M_{low3}(u(t)))^{\frac12},
\\&\|\langle x\rangle P_{\leq j}\partial^{\leq1} Zu\|_{L^{\infty}}\leq \frac{C}{\delta^2}(M_{high3}(u(t))+M_{high4}(u(t)))^{\frac12}.
\end{align*}

The estimates for terms $II_3$ follow analogously to that for $II_2$ and are even simpler. For terms $II_4$, we need to consider 
\[
\sum_{l,m=0}^{+\infty}P_j\left(P_l(\partial^{\leq1}u)P_m(\partial^{\leq1}Z^{\leq1} u)\right).
\]
The cases $j\leq20$ or $l=\min\{l,m,j\}$ or $j=\min\{l,m,j\}$ can be easily treated. When $j\geq21$ and $m=\min\{l,m,j\}$, using  $P_l(\partial^{\leq1}u)=\nabla\cdot w$, where $w^i$ be the solution of
$\Delta w^i=\partial_iP_l(\partial^{\leq1}u)$, applying identity \eqref{space_div} to $\nabla\cdot w$ and employing Bernstein's inequality, Lemma \ref{radial_decay} shows that this case is fine. Hence
\begin{align*}
II_3+II_4\leq \frac{C\varepsilon}{\delta^3}\int(1+t)^{-1}M_{high}(u(t))dt.
\end{align*}

Collecting the preceding estimates, we conclude that
\begin{align}\label{high2}
M_{high3}(u(T))\leq &CM_{high3}(u(0))+\frac{C\varepsilon}{\delta^3}\int (1+t)^{-1}M_{high}(u(t))dt.
\end{align}

\subsubsection{\textbf{Estimates for $M_{high5}(u(t))$}}In this paper, $\mathcal{O}(1)$ means a generic constant which may vary from line to line and let
\begin{align}\label{sum_syb}
\sum_{j_1\thicksim j_2\backsim j}=\sum_{j_1=j}^{+\infty}\sum_{j_2=j_1-2}^{j_1+2}+\sum_{j_1=0}^{j}\sum_{j_2=j-2}^{j+2}+\sum_{j_2=0}^{j}\sum_{j_1=j-2}^{j+2}.
\end{align}
Applying $Z^2$ to equation \eqref{NLG1}, we have
\begin{align}
\label{NLG3}&\Box Z^2u+Z^2u=\left(C^{i\alpha\gamma}\partial_{\gamma}u+C^{i\alpha}u\right)\partial_{\alpha}\partial_iZ^2u+\mathfrak{R},
\end{align}
where
\begin{align*}
\mathfrak{R}=\sum_{\substack{|a|+|b|\leq 2\\|b|\leq 1}}\left(C_{ab}^{\alpha\beta\gamma}\partial_{\gamma}Z^au+C_{ab}^{\alpha\beta}Z^au\right)\partial_{\alpha}\partial_{\beta}Z^bu+Z^2(H(u)).
\end{align*}
Here $C_{ab}^{\alpha\beta\gamma},C_{ab}^{\alpha\beta}$ are constants. Thus
\begin{align}\label{NLG_LP1}
&\Box (Z^2u)_j+(Z^2u)_j
\nonumber\\=&\left(C^{i\alpha\gamma}\partial_{\gamma}u+C^{i\alpha}u\right)\partial_{i}\partial_{\alpha}(Z^2u)_j+[P_j,C^{i\alpha\gamma}\partial_{\gamma}u+C^{i\alpha}u]\partial_{i}\partial_{\alpha}Z^2u+P_j(\mathfrak{R}).
\end{align}
The energy estimates for \eqref{NLG_LP1} are obtained by multiplying the equation by $\partial_t(Z^2u)_j$, integrating over $\mathbb R^3$, integrating by parts, and employing \eqref{small1}; this gives
\begin{align}
M_{high5}(u(T))\leq CM_{high5}(u(0))+\frac{C\varepsilon}{\delta^3}\int_0^T(1+t)^{-1}M_{high5}(u(t))dt+III_1+III_2,
\end{align}
where
\begin{align*}
&III_1=\sum_{j=0}^{+\infty}(2^j)^{1+42\delta}\int_0^T\int\left([P_j,C^{i\alpha\gamma}\partial_{\gamma}u+C^{i\alpha}u]\partial_{i}\partial_{\alpha}Z^2u\right)\partial_t(Z^2u)_jdxdt,
\\&III_2=\sum_{j=0}^{+\infty}(2^j)^{1+42\delta}\int_0^T\int P_j(\mathfrak{R})\partial_t(Z^{2}u)_jdxdt.
\end{align*}

Now it's time to deal with terms $III_1$. By using lemma \ref{para-product2}, 
\[
III_{1}\leq III_{1,1}+III_{1,2}+III_{1,3}, 
\]
where
\begin{align*}
&III_{1,1}=\int_0^T\sum_{j=0}^{+\infty}(2^j)^{1+42\delta}\left|\int\left(\sum_{l=0}^{+\infty} P_j\left(\left(P_{\leq l+2}\partial_i \partial Z^2u\right)P_l\partial^{\leq1}u\right)\right)\partial_t(Z^2u)_jdx\right|dt,
\\&III_{1,2}=\int_0^T\sum_{j=0}^{+\infty}(2^j)^{1+42\delta}\left|\int\left(\sum_{l=0}^{+\infty}\left(P_{\leq l+2}P_j\partial_i \partial Z^2u\right)P_l\partial^{\leq1}u\right)\partial_t(Z^2u)_jdx\right|dt,
\\&III_{1,3}=\int_0^T\sum_{j=0}^{+\infty}(2^j)^{42\delta}\left|\int\left(\sum_{l=j-2}^{j+2}L(P_{\leq l-3}\partial^{\leq1}u,P_l\partial_i \partial Z^2u)\right)\partial_t(Z^2u)_jdx\right|dt.
\end{align*} 

The estimate $\left(\sum_{m=0}^{l-4}\Big((2^m)^{\frac12-21\delta}\Big)^2\right)^{1/2}\leq C(2^l)^{\frac12-21\delta}$ implies
\[
\left\|P_{\leq l-4}\partial\nabla Z^2u\right\|_{L^2}\leq C(2^l)^{\frac12-21\delta}\Big(M_{high5}(u(t))\Big)^{1/2},
\]
which combining a similar fashion of discussion to $I_1,II_{1,1}$ gives
\begin{align*}
III_{1,1}&\leq C\int_0^T\sum_{j=0}^{+\infty}\sum_{l=j-4}^{+\infty}\sum_{m=l-3}^{l+2}(2^{l})^{\frac 12+21\delta}\left(\|P_m\partial\nabla Z^2u\|_{L^2}\|P_l\partial^{\leq1} u\|_{L^{\infty}}\right)\left((2^{j})^{\frac 12+21\delta}\|\partial_t(Z^2u)_j\|_{L^2}\right)dt
\\&\quad+C\int_0^T\sum_{j=0}^{+\infty}\sum_{l=j-2}^{j+2}(2^{l})^{\frac 12+21\delta}\left(\left\| P_{\leq l-4}\partial\nabla Z^2u\right\|_{L^2} \|P_l\partial^{\leq1} u\|_{_{L^{\infty}}}\right)\left((2^{j})^{\frac 12+21\delta}\|\partial_t(Z^2u)_j\|_{L^2}\right)dt
\\&\leq \frac{C\varepsilon}{\delta^3}\int_0^T(1+t)^{-1}M_{high}(u(t))dt.
\end{align*}

The term $III_{1,2}$ is similar to terms $I_2$ and $III_{1,1}$ and we obtain
\begin{align*}
III_{1,2}&\leq C\int_0^T\sum_{j=0}^{+\infty}(2^j)^{1+42\delta}\sum_{l=j-4}^{+\infty}\|P_j\partial\nabla Z^2u\|_{L^2}\|P_l\partial^{\leq1} u\|_{L^{\infty}}\|\partial_t(Z^2u)_j\|_{L^2}dt
\\&\leq \frac{C\varepsilon}{\delta^3}\int_0^T(1+t)^{-1}M_{high5}(u(t))dt.
\end{align*}

Now we consider the term $III_{1,3}$. Proceeding similarly to the estimates for $I_3,II_{1,3}$,
\begin{align*}
III_{1,3}&\leq C\int_0^T\sum_{j=0}^{+\infty}(2^j)^{-\frac 12+21\delta}\sum_{l=j-2}^{j+2}\left\|P_{\leq l-3}\partial^{\leq1}\nabla u \right\|_{L^{\infty}}\left\|P_l\partial\nabla  Z^2 u\right\|_{L^2}(2^j)^{\frac12+21\delta}\|\partial_t(Z^2u)_j\|_{L^2}dt
\\&\leq \frac{C\varepsilon}{\delta^3}\int_0^T(1+t)^{-1}M_{high5}(u(t))dt.
\end{align*}

We next turn to the term $III_2$. It suffices to control
\begin{align*}
&III_{2,1}=\sum_{\substack{|a|+|b|\leq2\\|b|\leq1}}\sum_{j,l,m=0}^{+\infty}(2^j)^{1+42\delta}\int P_j\left(\partial^{\leq1}(Z^au)_l\partial^2(Z^bu)_m\right)\partial_t(Z^{2}u)_jdx,
\\&III_{2,2}=\sum_{|a|+|b|\leq2}\sum_{j,l,m=0}^{+\infty}(2^j)^{1+42\delta}\int P_j\left(\partial^{\leq1}(Z^au)_l\partial^{\leq1}(Z^bu)_m\right)\partial_t(Z^{2}u)_jdx.
\end{align*}

For terms $III_{2,2}$, the cases $|a|=0$ or $|b|=0$ are simple. When $|a|=1,|b|=1$, if $j\geq20$ and $l\geq m$, using $P_l(\partial^{\leq1}Zu)=\nabla\cdot w$, where $w^i$ be the solution of
$\Delta w^i=\partial_iP_l(\partial^{\leq1}Zu)$, and a similar fashion of discussion on $II_4$, yields that it's acceptable. When $j\leq20$, if $l\geq20$ or $m\geq20$, the discussion on $II_4$ also applies. Then it suffices to estimate
\begin{align}\label{term1}
\frac{1+t}{\varepsilon}\sum_{l\leq 20}\|\partial^{\leq1}(Zu)_l\|_{L^4}^4.
\end{align}
Using $t\widetilde{\Omega}_{ij}=x^iL_j-x^jL_i$ shows that the cases $Z=\partial$ or $Z=\Omega$ are fine. By calculation, we get
\begin{align*}
\int(\partial^{\leq1}(Lu)_l)^4dx&\leq \int (L\partial^{\leq1}u_l)(\partial^{\leq1}(Lu)_l)^3dx+\frac{C\varepsilon}{\delta}(1+t)^{-1}(M_{high}(u(t)))^{\frac12}\left(\int(\partial^{\leq1}(Lu)_l)^4dx\right)^{\frac12}
\\&\leq \frac{d}{dt}\int (x\partial^{\leq1}u_l)(\partial^{\leq1}(Lu)_l)^3dx+\frac{C\varepsilon}{\delta(1+t)}(M_{high}(u(t)))^{\frac12}\left(\int(\partial^{\leq1}(Lu)_l)^4dx\right)^{\frac12},
\end{align*}
which combining Young's inequality and the identity $\frac{1+t}{\varepsilon}\frac{d}{dt}f(t)=\frac{d}{dt}\left(\frac{1+t}{\varepsilon}f(t)\right)-\frac{f(t)}{\varepsilon}$ yields that terms \eqref{term1} are admissible.

For terms $III_{2,1}$, similarly to $II_2$, we consider the following three cases: \textbf{Case 1}: $l=\min\{j,l,m\}$; \textbf{Case 2}: $m=\min\{j,l,m\}$; \textbf{Case 3}: $j=\min\{j,l,m\}$. Now we arrive at \textbf{Case 3} and we only need to deal with
\[
   C\left(M_{high5}(u(t))\right)^{1/2}\left(\sum_{j=0}^{+\infty}\left(\sum_{l=j}^{+\infty}\sum_{m=l-2}^{l+2}(2^j)^{\frac 12+21\delta}\left\|\partial^{\leq1}(Z^au)_l\partial^2(Z^bu)_m\right\|_{L^2(\mathbb R^3)}\right)^{2}\right)^{1/2}.
\]
Observe that $2^j\leq 2^m$ and $2^j\leq 2^l$. When $|a|=2$, applying \eqref{small2} yields 
\begin{align*}
\sum_{j=0}^{+\infty}\left(\sum_{l=j}^{+\infty}\sum_{m=l-2}^{l+2}(2^l)^{\frac12+21\delta}\left\|\partial^{\leq1}(Z^2u)_l\right\|_{L^2(\mathbb R^3)}\left\|\partial^2u_m\right\|_{L^{\infty}(\mathbb R^3)}\right)^{2}
 \leq \frac{C\varepsilon^2}{\delta^6}(1+t)^{-2}M_{high5}(u(t)).
\end{align*}
Otherwise, when $|a|,|b|\leq1$, the cases where $|a|=0$ or $|b|=0$ are straightforward as in Case 3 of $II_2$; it therefore suffices to assume $|a|=|b|=1$.
In this case, $\partial^2(Zu)_m=\partial\nabla(Zu)_m$. Using identity \eqref{space_div} gives
\begin{align*}
&\sum_{j=0}^{+\infty}\left(\sum_{l=j}^{+\infty}\sum_{m=l-2}^{l+2}(2^l)^{\frac 12+21\delta}\left\|\partial^{\leq1}(Zu)_l\partial\nabla(Zu)_m\right\|_{L^2(\mathbb R^3)}\right)^2
\\ \leq& \frac{C}{(1+t)^2}\sum_{j=0}^{+\infty}\left(\sum_{l=j}^{+\infty}\sum_{m=l-2}^{l+2}(2^l)^{\frac 12+21\delta}\Big(\left\|\partial^{\leq1}(Zu)_lL\partial(Zu)_m\right\|_{L^2(\mathbb R^3)}
+\left\|\langle x\rangle\partial^{\leq1}(Zu)_l\partial^2(Zu)_m\right\|_{L^2(\mathbb R^3)}\Big)\right)^2,
\end{align*}
which combining Bernstein's inequality and Lemma \ref{radial_decay} yields that in \textbf{Case 3},
\begin{align}\label{bound1}
III_{2,1}\leq \frac{C\varepsilon}{\delta^3}(1+t)^{-1}M_{high}(u(t)).
\end{align}
For \textbf{Case 1}, it suffices to estimate
\begin{align*}
&\sum_{j=0}^{+\infty}\sum_{m=j-2}^{j+2}(2^{j})^{1+42\delta}\left|\int P_j\left((P_{\leq j}\partial^{\leq1} Z^au)(\partial^2Z^bu)_m\right)(\partial_tZ^2u)_jdx\right|.
\end{align*}
When $|a|=2$, by using Bernstein's inequality and Lemma \ref{product_radial}, we get
\begin{align*}
\left\|(P_{j_1}\partial^{\leq1} Z^2u)Z\partial u_m\right\|_{L^2}&\leq C\left\|P_{j_1}\partial^{\leq1} Z^2u\right\|_{L^{\infty}}\left\|Z\partial u_m\right\|_{L^2}
\leq C(2^{j_1})^{\frac32}\left\|P_{j_1}\partial^{\leq1} Z^2u\right\|_{L^{2}}\left\|Z\partial u_m\right\|_{L^2},
\\ \left\|(P_{j_1}\partial^{\leq1} Z^2u)r\partial^2 u_m\right\|_{L^2}
&\leq C\left(\left\|(P_{j_1}\partial^{\leq1} Z^2u)r R_0(\partial^2 u_m)\right\|_{L^2}+\sum_{k=1}^{+\infty}\left\|(P_{j_1}\partial^{\leq1} Z^2u)r R_k(\partial^2 u_m)\right\|_{L^2}\right)
\\&\leq C\left(\left\|\partial^2 u_m\right\|_{L^2}+\sum_{k=1}^{+\infty}2^k\left\|R_k(\partial^2 u_m)\right\|_{L^2}\right)(2^{j_1})^{1/2}\left\|P_{j_1}\partial^{\leq1} Z^2u\right\|_{L^2},
\end{align*}
which combining \eqref{space_div} implies
\begin{align}\label{difficulty3}
&\quad\sum_{j=0}^{+\infty}\sum_{m=j-2}^{j+2}(2^{j})^{1+42\delta}\left|\int P_j\left((P_{\leq j}\partial^{\leq1} Z^2u)(\partial\nabla u)_m\right)(\partial_tZ^2u)_jdx\right|
\nonumber\\ &\leq C\Big(M_{high5}(u(t))\Big)^{\frac12}(1+t)^{-1}
\nonumber\\&\quad\times\left(\sum_{j=0}^{+\infty}\Big(\sum_{m=j-2}^{j+2}(2^{m})^{\frac12+21\delta}\sum_{j_1=0}^j\left(\left\|(P_{j_1}\partial^{\leq1} Z^2u)Z\partial u_m\right\|_{L^2}+\left\|(P_{j_1}\partial^{\leq1} Z^2u)r\partial^2 u_m\right\|_{L^2}\right)\Big)^2\right)^{\frac12}
\nonumber\\&\leq \frac{C\varepsilon}{\delta^3}(1+t)^{-1}M_{high5}(u(t)).
\end{align}
In the regime $|a|\leq 1$, it is enough to consider the nontrivial case $|a|=1,|b|=1$, as all other possibilities are immediate as in Case 1 of $II_4$. Invoking Lemma \ref{radial_decay} along with the identity \eqref{space_div} yields that the term $III_{2,1}$ is bounded by 
\begin{align}\label{difficulty2}
&\quad C\sum_{j=0}^{+\infty}\sum_{m=j-2}^{j+2}(2^{j})^{1+42\delta}\left|\int P_j\left((P_{\leq j}\partial^{\leq1} Zu)(\partial\nabla Zu)_m\right)(\partial_tZ^2u)_jdx\right|
\nonumber\\ &\leq C(1+t)^{-1}\left(M_{high5}(u(t))\right)^{\frac12}
\nonumber\\&\quad \times\left(\sum_{j=0}^{+\infty}\left(\sum_{m=j-2}^{j+2}(2^{m})^{\frac12+21\delta}\left(\|rP_{\leq j}\partial^{\leq1} Zu\|_{L^{\infty}}\|(\partial^2 Zu)_m\|_{L^2}+\|P_{\leq j}\partial^{\leq1} Zu\|_{L^{\infty}}\|Z(\partial Zu)_m\|_{L^2}\right)\right)^2\right)^{\frac12}
\nonumber\\&\leq \frac{C\varepsilon}{\delta^3}(1+t)^{-1}M_{high}(u(t)).
\end{align}
In \textbf{Case 2}, $m\leq l$. When $|a|=|b|=1$, the discussion in \textbf{Case 1} also applies. When $|a|=2$, we get
\begin{align*}
&\quad\sum_{j=0}^{+\infty}\sum_{l=j-2}^{j+2}(2^{j})^{1+42\delta}\left|\int P_j\left((\partial^{\leq1} (Z^2u)_l(P_{\leq j}\partial\nabla u)\right)(\partial_tZ^2u)_jdx\right|
\nonumber\\ &\leq \sum_{j=0}^{+\infty}\sum_{l=j-2}^{j+2}(2^{j})^{1+42\delta}\|(\partial_tZ^2u)_j\|_{L^2}\left\|\partial^{\leq1} (Z^2u)_l\right\|_{L^2}\left\|P_{\leq j}(\partial\nabla u)\right\|_{L^{\infty}}
\nonumber\\&\leq \frac{C\varepsilon}{\delta^3}(1+t)^{-1}M_{high5}(u(t)).
\end{align*}
When $|a|=0$ or $|b|=0$, the cases are simple. Therefore in \textbf{Case 2}, terms $III_{2,1}$ satisfy the same bounds as those given in \eqref{bound1}

Combing the discussion above, we have
\begin{align}\label{high5}
M_{high5}(u(T))\leq &CM_{high}(u(0))+\frac{C\varepsilon}{\delta^3} M_{high}(u(T))+\frac{C\varepsilon}{\delta^3}\int (1+t)^{-1}M_{high}(u(t))dt.
\end{align}

\subsubsection{\textbf{Estimates for $M_{high2}(u(t))$}}\label{high2_sec}
By Lemma \ref{ALP_commu} and noticing $\Box u=\partial_t^2u-\partial_r^2u-\frac2r\partial_ru-\frac1{r^2}(\Omega\cdot \Omega)u$, we get $[\Box,R_k]u=0$. Thus
\begin{align}\label{high3}
\Box u_{j,k}+u_{j,k}=R_kP_j\left((C^{i\alpha\gamma}\partial_{\gamma}u+C^{i\alpha}u)\partial_{i}\partial_{\alpha}u+H(u)\right).
\end{align}
We now perform energy estimates for equation \eqref{high3} with $k\geq j+20$ and $j\geq0$. Multiplying it by $\partial_tu_{j,k}$ and integrating over $[0,T]\times\mathbb R^3$ yields
\[
M_{high2}(u(T))\leq M_{high2}(u(0))+IV_1+IV_2+IV_3,
\]
where
\begin{align*}
&IV_1=\mathcal{O}(1)\sum_{j=0}^{+\infty}\sum_{k=j+20}^{+\infty}(2^{j})^{3+50\delta}(2^{k})^{2+2\delta}\int_0^T\int R_k\left(P_j(\partial^{\leq1}u\partial\nabla u)-\partial^{\leq1}u\partial\nabla P_ju\right)\partial_tu_{j,k}dxdt,
\\&IV_2=\mathcal{O}(1)\sum_{j=0}^{+\infty}\sum_{k=j+20}^{+\infty}(2^{j})^{3+50\delta}(2^{k})^{2+2\delta}\int_0^T\int R_k\left(\partial^{\leq1}u\partial\nabla P_ju\right)\partial_tu_{j,k}dxdt,
\\&IV_3=\sum_{j=0}^{+\infty}\sum_{k=j+20}^{+\infty}(2^{j})^{3+50\delta}(2^{k})^{2+2\delta}\int_0^T\int R_kP_j\left(H(u)\right)\partial_tu_{j,k}dxdt.
\end{align*}

In Section \ref{high2_sec}, let
\[
W_{j,k}^1=(2^{j})^{3+50\delta}(2^{k})^{2+2\delta},\quad \sum_{j,k}=\sum_{j=0}^{+\infty}\sum_{k=j+20}^{+\infty}.
\]
An application of Lemma \ref{ALP_ortho} yields 
\[
IV_1\leq IV_{1,1}+IV_{1,2}+IV_{1,3},
\]
where
\begin{align*}
&IV_{1,1}\leq C\sum_{j,k}W_{j,k}^1\sum_{k_2=k-1}^{k+1}\int_0^T\left\|P_j\left((R_{\leq k}\partial^{\leq1} u)(R_{k_2}\partial\nabla u)\right)-(R_{\leq k}\partial^{\leq1} u)(R_{k_2}\partial\nabla u_j)\right\|_{L^2}\|\partial_tu_{j,k}\|_{L^2}dt,
\\&IV_{1,2}\leq C\sum_{j,k}W_{j,k}^1\sum_{k_1=k-1}^{k+1}\int_0^T\left\|P_j\left((R_{k_1}\partial^{\leq1} u)(R_{\leq k}\partial\nabla u)\right)-(R_{k_1}\partial^{\leq1} u)(R_{\leq k}\partial\nabla u_j)\right\|_{L^2}\|\partial_tu_{j,k}\|_{L^2}dt,
\\&IV_{1,3}\leq C\sum_{j,k}W_{j,k}^1\sum_{k_1=k}^{+\infty}\sum_{k_2=k_1-1}^{k_1+1}\int_0^T\left\|P_j\left((R_{k_1}\partial^{\leq1} u)(R_{k_2}\partial\nabla u)\right)-(R_{k_1}\partial^{\leq1} u)(R_{k_2}\partial\nabla u_j)\right\|_{L^2}\|\partial_tu_{j,k}\|_{L^2}dt.
\end{align*}
For the term $IV_{1,1}$, applying Lemma \ref{para-product2} yields that 
\[
IV_{1,1}\leq IV_{1,1,1}+IV_{1,1,2}+IV_{1,1,3},
\]
where
\begin{align*}
&IV_{1,1,1}=C\sum_{j,k}W_{j,k}^1\sum_{k_2=k-1}^{k+1}\int_0^T\left\|\sum_{l=0}^{+\infty} P_j\left(\left(P_{\leq l+2}g\right)P_lf\right)\right\|_{L^2}\|\partial_tu_{j,k}\|_{L^2}dt,
\\&IV_{1,1,2}=C\sum_{j,k}W_{j,k}^1\sum_{k_2=k-1}^{k+1}\int_0^T\left\|\sum_{l=0}^{+\infty}\left(P_{\leq l+2}P_jg\right)P_lf\right\|_{L^2}\|\partial_tu_{j,k}\|_{L^2}dt,
\\&IV_{1,1,3}=C\sum_{j,k}W_{j,k}^12^{-j}\sum_{k_2=k-1}^{k+1}\int_0^T\left\|\sum_{l=j-2}^{j+2}L(P_{\leq l-3}f,P_lg)\right\|_{L^2}\|\partial_tu_{j,k}\|_{L^2}dt,
\end{align*} 
with 
\[
f=R_{\leq k}\partial^{\leq1} u,\quad g=R_{k_2}\partial\nabla u.
\]
We begin with terms $IV_{1,1,1}$. A necessary condition for the corresponding summation not to vanish is $j\leq l+4$. By Lemma \ref{xw1} and using the decomposition $\nabla=\frac{x}r\partial_r-\frac{x}{r^2}\wedge\Omega$, for any function $\bar{f}$,
\begin{align*}
\|P_{l_2}R_{k_3}\nabla\bar{f}\|_{L^2(\mathbb R^3)}=\|R_{k_3}\nabla P_{l_2}\bar{f}\|_{L^2}\leq C\sum_{k_4=k_3-1}^{k_3+1}\|\nabla R_{k_4} P_{l_2}\bar{f}\|_{L^2}
\leq C2^{l_2}\sum_{k_4=k_3-1}^{k_3+1}\| R_{k_4} P_{l_2}\bar{f}\|_{L^2}.
\end{align*}
Using Lemma \ref{sob1} and the above estimate then yields
\begin{align}\label{iv_111}
&\quad\|P_{\leq l+2}R_{k_2}\partial\nabla u\|_{L^{\infty}(\mathbb R^3)}
\nonumber\\&\leq C\sum_{l_1=0}^{l+2}\|P_{l_1}R_{k_2}\partial\nabla u\|_{L^{\infty}(\mathbb R^3)}
\nonumber\\&\leq C(1+t)^{-1}\sum_{l_1=0}^{l+2}\left((2^{l_1})^{\frac32}\|P_{l_1}R_{k_2}\partial\nabla u\|_{L^2}+\sum_{k_3=k_2-1}^{k_2+1}\sum_{l_2=l_1-1}^{l_1+1}\left((2^{l_1})^{1/2}\|P_{l_2}R_{k_3}Z^{\leq 1}\partial\nabla u\|_{L^{2}}\right)\right)
\nonumber\\&\leq \frac{C(2^l)^{1-20\delta}}{1+t}\left(\sum_{l_2=0}^{+\infty}\sum_{k_3=k_2-2}^{k_2+2}\left((2^{l_2})^{3+40\delta}\|P_{l_2}R_{k_3}\partial u\|^2_{L^2}+(2^{l_2})^{1+40\delta}\|P_{l_2}R_{k_3}\partial Z^{\leq 1} u\|^2_{L^{2}}\right)\right)^{\frac12}.
\end{align}
By using inequality \eqref{iv_111} and the estimate $\|R_{\leq k}\bar{f}\|_{L^2(\mathbb R^3)}\leq C\|\bar{f}\|_{L^2(\mathbb R^3)}$, we have
\begin{align*}
&\quad IV_{1,1,1}
\nonumber\\&\leq C\int_0^T\left(M_{high2}(u(t))\right)^{\frac12}\left(\sum_{j,k}W_{j,k}^1\sum_{k_2=k-1}^{k+1}\left(\sum_{l=j-4}^{+\infty}\left\|P_{\leq l+2}R_{k_2}\partial\nabla u\right\|_{L^{\infty}}\left\|R_{\leq k}P_l\partial^{\leq1} u\right\|_{L^2}\right)^2\right)^{\frac12}dt
\nonumber\\&\leq \int_0^T\frac{CM_{high}(u(t))}{1+t}\left(\sum_{j=0}^{+\infty}(2^j)^{3+50\delta}\left(\sum_{l=j-4}^{+\infty}(2^{l})^{1-20\delta}
\left\|P_l\partial^{\leq1} u\right\|_{L^2}\right)^2\right)^{\frac12}dt.
\end{align*}
Therefore 
\begin{align}\label{bound2}
IV_{1,1,1}\leq \frac{C\varepsilon}{\delta^3}\int_0^T(1+t)^{-1}M_{high}(u(t))dt.
\end{align}
Terms $IV_{1,1,2}$ can be handled analogously to $IV_{1,1,1}$. Employing the identity $\nabla=\frac{x}r\partial_r-\frac{x}{r^2}\wedge\Omega$ yields
\begin{align*}
\|\nabla(P_{\leq l-3}f)\|_{L^{\infty}} \leq C\|(R_{\leq k}\partial_r\partial^{\leq1} u)\|_{L^{\infty}}+C\left\|R_{\leq k}\left(\frac{\Omega}r\partial^{\leq1} u\right)\right\|_{L^{\infty}}\leq C\|\partial^{\leq1}\nabla u\|_{L^{\infty}},
\end{align*}
which combining inequality \eqref{small1} and the discussion on $I_3$ shows that terms $IV_{1,1,3}$ have the same bound as those given in \eqref{bound2}. 

Now we come to terms $IV_{1,2}$. Similarly to $IV_{1,1}$, 
\[
IV_{1,2}\leq IV_{1,2,1}+IV_{1,2,2}+IV_{1,2,3},
\]
where $IV_{1,2,i}$ is equal to $IV_{1,1,i}$ with 
\[
f=R_{k_2}\partial^{\leq1} u,\quad g=R_{\leq k}\partial\nabla u.
\]
For terms $IV_{1,2,1}$, we rewrite $\sum_{l=0}^{+\infty} P_j\left(\left(P_{\leq l+2}g\right)P_lf\right)=\sum_{l=j-4}^{+\infty}\sum_{m=0}^{l+2} P_j\left(P_mgP_lf\right)$ and consider the following three cases: $l=\min\{l,m,j\}$, $m=\min\{l,m,j\}$ or $j=\min\{l,m,j\}$. 
\begin{itemize}
  \item When $m=\min\{l,m,j\}$, using $\|P_{\leq l+2}R_{\leq k}f\|_{L^{\infty}}\leq C\|f\|_{L^{\infty}}$ and inequality \eqref{small1} implies
\begin{align*}
IV_{1,2,1}&\leq C\int_0^T\sum_{j,k}W_{j,k}^1\sum_{k_2=k-1}^{k+1}\sum_{l=j-2}^{j+2}\left\|P_{\leq l+2}R_{\leq k}\partial\nabla u\right\|_{L^{\infty}}\left\|R_{k_2}P_l\partial^{\leq1} u\right\|_{L^2}\|\partial_tu_{j,k}\|_{L^2}dt,
\end{align*}
which satisfies the same bound as those given in \eqref{bound2}.
  \item When $l=\min\{l,m,j\}$, we  get $j-2\leq m\leq j+2$. Therefore the estimate
\begin{align}\label{add2}
\|P_{m}R_{\leq k}\partial\nabla u\|_{L^{\infty}(\mathbb R^3)}\leq \|\partial\nabla u_m\|_{L^{\infty}(\mathbb R^3)}
\leq \frac{C\varepsilon}{\delta}(1+t)^{-1}(2^m)^{-\delta},
\end{align}
shows that terms $IV_{1,2,1}$ are bounded by
\begin{align*}
C\int_0^T\sum_{j,k}W_{j,k}^1\sum_{k_2=k-1}^{k+1}\sum_{l=j-4}^{+\infty}\sum_{m=j-2}^{j+2}\left\|P_{m}R_{\leq k}\partial\nabla u\right\|_{L^{\infty}}\left\|R_{k_2}P_l\partial^{\leq1} u\right\|_{L^2}\|\partial_tu_{j,k}\|_{L^2}dt
\end{align*}
with the same bound as those given in \eqref{bound2}.
  \item When $j=\min\{l,m,j\}$, using \eqref{add2} and Lemma \ref{LP2} yields that terms $IV_{1,2,1}$ are bounded by
\begin{align*}
C\int_0^T\sum_{j,k}W_{j,k}^1\sum_{k_2=k-1}^{k+1}\sum_{l=j-4}^{+\infty}\sum_{m=l-2}^{l+2}\left\|P_{m}R_{\leq k}\partial\nabla u\right\|_{L^{\infty}}\left\|R_{k_2}P_l\partial^{\leq1} u\right\|_{L^2}\|\partial_tu_{j,k}\|_{L^2}dt,
\end{align*}
which has the same bound as those given in \eqref{bound2}.
\end{itemize}
Now we turn to terms $IV_{1,2,2}$. A straight calculation shows that 
\begin{align*}
\left\|\sum_{l=j-4}^{+\infty}\left(P_{\leq l+2}P_jR_{\leq k}\partial\nabla u\right)P_lR_{k_2}\partial^{\leq1} u\right\|_{L^2}
\leq \frac{C\varepsilon}{\delta}(1+t)^{-1}(2^j)^{-\delta}\sum_{l=j-4}^{+\infty}\left\|P_lR_{k_2}\partial^{\leq1} u\right\|_{L^2},
\end{align*}
which has the same bound as those given in \eqref{bound2}.

By using Bernstein's inequality $\|\nabla P_j\bar{f}\|_{L^{\infty}}\leq C2^j\|P_j\bar{f}\|_{L^{\infty}}$ and  Lemma \ref{sob1}, 
\begin{align*}
&\quad\|\nabla(P_{\leq l-3}R_{k_2}\partial^{\leq1} u)\|_{L^{\infty}}
\\&\leq C(1+t)^{-1}\sum_{m=0}^{l-3}2^m\left((2^{m})^{\frac12}\sum_{m_1=m-1}^{m+1}\sum_{k_3=k_2-1}^{k_2+1}\|P_{m_1}R_{k_3}\partial^{\leq1}Z^{\leq1} u\|_{L^2}
+(2^{m})^{\frac32}\|P_{m}R_{k_2}\partial^{\leq1}u\|_{L^2}\right)
\\&\leq C(1+t)^{-1}\left(\sum_{m=0}^{l-3}(2^{m})^{2-40\delta}\right)^{\frac12}\left(\sum_{m_1=0}^{+\infty}\sum_{k_3=k_2-1}^{k_2+1}(2^{m_1})^{1+40\delta}\|P_{m_1}R_{k_3}\partial^{\leq1}Z^{\leq1} u\|_{L^2}^2\right)^{\frac12}
\\&\quad+C(1+t)^{-1}\left(\sum_{m=0}^{l-3}(2^m)^{2-50\delta}\right)^{\frac12}\left(\sum_{m=0}^{+\infty}(2^{m})^{3+50\delta}\|P_{m}R_{k_2}\partial^{\leq1}u\|^2_{L^2}\right)^{\frac12},
\end{align*}
which combing a similar discussion of $IV_{1,1,3}$ shows that
\begin{align*}
IV_{1,2,3}\leq &C\sum_{j,k}W_{j,k}^1(2^j)^{-1}\sum_{k_2=k-1}^{k+1}\int_0^T\sum_{l=j-2}^{j+2}\|\nabla(P_{\leq l-3}R_{k_2}\partial^{\leq1} u)\|_{L^{\infty}}\|P_lR_{\leq k}\partial\nabla u\|_{L^2}\|\partial_tu_{j,k}\|_{L^2}dt
\\ \leq &C\int_0^T(1+t)^{-1}M_{high}(u(t))\left(\sum_{j=0}^{+\infty}\left((2^j)^{\frac32+5\delta}\sum_{l=j-2}^{j+2}\|P_l\partial\nabla u\|_{L^2}\right)^2\right)^{\frac12}dt,
\end{align*}
with the same bound as those given in \eqref{bound2}.

Now let's consider terms $IV_{1,3}$. Similarly to $IV_{1,1}$, we get 
\[
IV_{1,3}\leq IV_{1,3,1}+IV_{1,3,2}+IV_{1,3,3}, 
\]
where
\begin{align*}
&IV_{1,3,1}=C\sum_{j,k}W_{j,k}^1\sum_{k_1=k}^{+\infty}\sum_{k_2=k_1-1}^{k_1+1}\int_0^T\left\|\sum_{l=0}^{+\infty} P_j\left(\left(P_{\leq l+2}R_{k_2}\partial\nabla u\right)P_lR_{k_1}\partial^{\leq1} u\right)\right\|_{L^2}\|\partial_tu_{j,k}\|_{L^2}dt,
\\&IV_{1,3,2}=C\sum_{j,k}W_{j,k}^1\sum_{k_1=k}^{+\infty}\sum_{k_2=k_1-1}^{k_1+1}\int_0^T\left\|\sum_{l=0}^{+\infty}\left(P_{\leq l+2}P_jR_{k_2}\partial\nabla u\right)P_lR_{k_1}\partial^{\leq1} u\right\|_{L^2}\|\partial_tu_{j,k}\|_{L^2}dt,
\\&IV_{1,3,3}=C\sum_{j,k}W_{j,k}^12^{-j}\sum_{k_1=k}^{+\infty}\sum_{k_2=k_1-1}^{k_1+1}\int_0^T\left\|\sum_{l=j-2}^{j+2}L(P_{\leq l-3}R_{k_1}\partial^{\leq1} u,P_lR_{k_2}\partial\nabla u)\right\|_{L^2}\|\partial_tu_{j,k}\|_{L^2}dt.
\end{align*}
H\"older inequality shows that terms $IV_{1,3,1}$ are bounded by
\begin{align*}
&C\int_0^T\left(M_{high2}(u(t))\right)^{\frac12}\left(\sum_{j,k}W_{j,k}^1\left(\sum_{k_1=k}^{+\infty}\sum_{k_2=k_1-1}^{k_1+1}\sum_{l=j-4}^{+\infty}\left\|P_{\leq l+2}R_{k_2}\partial\nabla u\right\|_{L^{\infty}}\left\|R_{k_1}P_l\partial^{\leq1} u\right\|_{L^2}\right)^2\right)^{\frac12}dt,
\end{align*}
which combining the estimate \eqref{iv_111} shows that terms $IV_{1,3,1}$ satisfy the same bound as those given in \eqref{bound2}. 
Terms $IV_{1,3,2}$ can be treated similarly to $IV_{1,3,1}$. Analogously to $IV_{1,1,3}$, we get the bound of $IV_{1,3,3}$ as follows
\begin{align*}
&C\sum_{j,k}(2^{j})^{2+50\delta}(2^{k})^{2+2\delta}\sum_{k_1=k}^{+\infty}\sum_{k_2=k_1-1}^{k_1+1}\int_0^T\sum_{l=j-2}^{j+2}\|\nabla(P_{\leq l-3}R_{k_1}\partial^{\leq1} u)\|_{L^{\infty}}\|P_lR_{k_2}\partial\nabla u\|_{L^2}\|\partial_tu_{j,k}\|_{L^2}dt,
\end{align*} 
which combining Lemma \ref{sob1} and $M_{low4}(u(t))$ yields that terms $IV_{1,3,3}$ have the same bound as those given in \eqref{bound2}. 

Therefore combing the discussion above shows that
\[
IV_1\leq \frac{C\varepsilon}{\delta^3}\int_0^T(1+t)^{-1}M_{high}(u(t))dt.
\]

Since
\[
IV_3\leq C\sum_{j=0}^{+\infty}\sum_{k=j+20}^{+\infty}W_{j,k}^1\int_0^T\int \left|R_kP_j\left(\sum_{j_1,k_1=0}^{+\infty}R_{k_1}P_{j_1}\partial^{\leq1}u\sum_{j_2,k_2=0}^{+\infty}R_{k_2}P_{j_2}\partial^{\leq1}u\right)\partial_tu_{j,k}\right|dxdt,
\]
we only need to consider the following four cases: ($k_1=\min\{k,k_1,k_2\}$ or $k=\min\{k,k_1,k_2\}$) and ($j_1=\min\{j,j_1,j_2\}$ or $j=\min\{j,j_1,j_2\}$). By Lemma \ref{ALP_ortho} and Lemma \ref{LP2}, it's not difficult to treat the term $IV_3$ and we omit the details. We finally get 
\[
IV_3\leq \frac{C\varepsilon}{\delta^3}\int_0^T(1+t)^{-1}M_{high}(u(t))dt.
\]

Thus it suffices to estimate the most difficult term $IV_2$. Let 
\[
v=P_ju.
\]
Using the identity $f=\sum_{k=0}^{+\infty}R_kf$, the definition \eqref{Rk_def} and the orthogonal property of $\varphi_{l,m}$ yields
\begin{align}\label{cc6}
\int_{\mathbb R^3}(R_kf)(R_kg)dx=\int_{\mathbb R^3}f(R_kg)dx,
\end{align}
holds for any scalar real-valued functions $f,g$. Thus using the definition \eqref{ALP_def} gives
\begin{align*}
IV_2&=\sum_{j,k}W_{j,k}^1\int_0^T\int \left((C^{i\alpha\gamma}\partial_{\gamma}u+C^{i\alpha}u)\partial_{i}\partial_{\alpha}v\right)R_k\partial_tvdxdt
\\&=\sum_{j,k}W_{j,k}^1\int_0^T\int_0^{+\infty}\int_{S^2} \left((C^{i\alpha\gamma}\partial_{\gamma}u(t,r\eta)+C^{i\alpha}u(t,r\eta))\partial_{i}\partial_{\alpha}v(t,r\eta)\right)R_k\partial_tv(t,r\eta)d_{S^2(\eta)}r^2drdt
\\&=\sum_{j,k}W_{j,k}^1\int_0^T\int_0^{+\infty}\int_{S^2} \left(C^{i\alpha\gamma}\partial_{\gamma}u(t,r\eta)\partial_{i}\partial_{\alpha}v(t,r\eta)\right)\int_{S^2}\partial_tv(t,r\xi)\widetilde{Q}_k(\langle \xi,\eta\rangle)d_{S^2(\xi)}d_{S^2(\eta)}r^2drdt
\\&\quad+\sum_{j,k}W_{j,k}^1\int_0^T\int_0^{+\infty}\int_{S^2} \left(C^{i\alpha}u(t,r\eta)\partial_{i}\partial_{\alpha}v(t,r\eta)\right)\int_{S^2}\partial_tv(t,r\xi)\widetilde{Q}_k(\langle \xi,\eta\rangle)d_{S^2(\xi)}d_{S^2(\eta)}r^2drdt.
\end{align*}
The second term on the right-hand side of the above inequality can be estimated similarly to the first term; we omit the details here. 
Recall that $\partial_i=\frac {x^i}r\partial_r+\frac{x^l}{r^2}\widetilde{\Omega}_{li}$. To better understand the action of $\partial_i$,
we use polar coordinates: $\eta^1=\cos\theta,\ \eta^2=\sin\theta\cos\phi,\ \eta^3=\sin\theta\sin\phi$. Thus $\widetilde{\Omega}_{li}$ is the summation of $\partial_{\theta},\partial_{\phi}$ and $d_{S^2(\eta)}=\sin\theta d\theta d\phi$, which combining integration by parts with respect $\theta$ or $\phi$ yields that for any scalar functions $f,g$ and any $i,j=1,2,3$, 
\begin{align}\label{int_by_parts}
\int_{\mathbb S^2}g\widetilde{\Omega}_{ij}fd_{S^2}=-\int_{\mathbb S^2}f\widetilde{\Omega}_{ij}gd_{S^2}.
\end{align}

 Consequently,
\begin{align*}
IV_2=&\sum_{j,k}\int_0^T \Big(IV_{2,1}+IV_{2,2}\Big)dt+terms\ involving\ C^{i\alpha}u\partial_{i}\partial_{\alpha}v,
\end{align*}
where
\begin{align}
\label{radial1}&IV_{2,1}=W_{j,k}^1
\int_0^{+\infty}\int_{S^2} C^{i\alpha\gamma}\partial_{\gamma}u(t,r\eta)\eta^i\partial_r\partial_{\alpha}v(t,r\eta)\int_{S^2}\partial_tv(t,r\xi)\widetilde{Q}_k(\langle \xi,\eta\rangle)d_{S^2(\xi)}d_{S^2(\eta)}r^2dr,
\\ \label{angular1}&IV_{2,2}=W_{j,k}^1\int_0^{+\infty}\int_{S^2} C^{i\alpha\gamma}\partial_{\gamma}u(t,r\eta)\frac{\eta^l}r\widetilde{\Omega}_{li}\partial_{\alpha}v(t,r\eta)\int_{S^2}\partial_tv(t,r\xi)\widetilde{Q}_k(\langle \xi,\eta\rangle)d_{S^2(\xi)}d_{S^2(\eta)}r^2dr.
\end{align}
We first consider the term $IV_{2,1}$. In the case $\alpha=0$, we write $IV_{2,1}^{\alpha=0}$ for $IV_{2,1}$. Then integrating by parts with respect to $r$ gives
\begin{align}
\label{radial2}&\quad IV_{2,1}^{\alpha=0}
\nonumber\\&=IV_{2,1,1}^{\alpha=0}-W_{j,k}^1\int_0^{+\infty}\int_{S^2} \left(C^{i0\gamma}\partial_{\gamma}u(t,r\eta)\eta^i\partial_{t}v(t,r\eta)\right)\int_{S^2}\partial_r\partial_tv(t,r\xi)\widetilde{Q}_k(\langle \xi,\eta\rangle)d_{S^2(\xi)}d_{S^2(\eta)}r^2dr
\nonumber\\&=IV_{2,1,1}^{\alpha=0}-W_{j,k}^1\int_0^{+\infty}\int_{S^2}\int_{S^2}\left(C^{i0\gamma}\partial_{\gamma}u(t,r\xi)\xi^i\partial_{t}v(t,r\xi)\right)\partial_r\partial_tv(t,r\eta)\widetilde{Q}_k(\langle \xi,\eta\rangle)d_{S^2(\xi)}d_{S^2(\eta)}r^2dr,
\end{align}
where we have swapped the variables $\xi$ and $\eta$ in the second term of the last inequality and 
\begin{align*}
IV_{2,1,1}^{\alpha=0}=-W_{j,k}^1\int_0^{+\infty}\int_{S^2} C^{i0\gamma}\left((\partial_r+\frac2r)\partial_{\gamma}u(t,r\eta)\right)\eta^i\partial_{t}v(t,r\eta)\int_{S^2}\partial_tv(t,r\xi)\widetilde{Q}_k(\langle \xi,\eta\rangle)d_{S^2(\xi)}d_{S^2(\eta)}r^2dr.
\end{align*}
Combining \eqref{radial1} and \eqref{radial2}, we get
\[
     IV_{2,1}^{\alpha=0}=\frac12 IV_{2,1,1}^{\alpha=0}+IV_{2,1,2}^{\alpha=0},
\]
where
\begin{align*}
IV_{2,1,2}^{\alpha=0}=-\frac12W_{j,k}^1\int_0^{+\infty}\int_{S^2}\int_{S^2}\left(f_0(t,r,\xi)-f_0(t,r,\eta)\right)\partial_{t}v(t,r\xi)\partial_r\partial_tv(t,r\eta)\widetilde{Q}_k(\langle \xi,\eta\rangle)d_{S^2(\xi)}d_{S^2(\eta)}r^2dr,
\end{align*}
with $f_0(t,r,\xi)=C^{i0\gamma}\partial_{\gamma}u(t,r\xi)\xi^i$. We first estimate $IV_{2,1,1}^{\alpha=0}$, which is given by
\[
 IV_{2,1,1}^{\alpha=0}=IV_{2,1,1,1}^{\alpha=0}+IV_{2,1,1,2}^{\alpha=0},
\]
where
\begin{align*}
&IV_{2,1,1,1}^{\alpha=0}=-W_{j,k}^1\int R_k\left(C^{i0\gamma}\left(\partial_r\partial_{\gamma}u\right)\frac{x^i}r\partial_{t}v\right)R_k\partial_tvdx,
\\&IV_{2,1,1,2}^{\alpha=0}=-W_{j,k}^1\int R_k\left(C^{i0\gamma}\frac2r\partial_\gamma u\frac{x^i}r\partial_{t}v\right)R_k\partial_tvdx.
\end{align*}
Terms $IV_{2,1,1,2}^{\alpha=0}$ will be canceled by the subsequent terms. Applying Lemma \ref{xw1} yields
\begin{align}\label{xw13}
IV_{2,1,1,1}^{\alpha=0}\leq CW_{j,k}^1\sum_{k_3=k-2}^{k+2}\left\|\sum_{k_1,k_2}R_{k_3}\left(\left(R_{k_1}\nabla\partial u\right)\partial_{t}R_{k_2}v\right)\right\|_{L^2}\|R_k\partial_tv\|_{L^2}.
\end{align}
Lemma \ref{ALP_ortho} shows that the sum is zero unless $\max\{k_3,k_1,k_2\}\leq med\{k_3,k_1,k_2\}+2$. The cases $k_3=\min\{k_3,k_1,k_2\}$ and $k_1=\min\{k_3,k_1,k_2\}$ are straightforward, so we assume $k_2=\min\{k_3,k_1,k_2\}$, which impies $|k_3-k_1|\leq2$. Employing Lemma \ref{sob1} then controls the right-hand side of \eqref{xw13} by
\begin{align*}
&C(2^j)^{3+50\delta}(2^k)^{2+2\delta}\sum_{k_1=k-4}^{k+4}\|R_{k_1}\partial\nabla u\|_{L^{\infty}}\|R_{\leq k+4}\partial_{t}v\|_{L^2}\|R_k\partial_tv\|_{L^2}
\\ \leq &\frac{C}{\delta(1+t)}(2^{j})^{-3\delta}\Big((2^{k})^{1+\delta}\sum_{k_1=k-6}^{k+6}\big(\sum_{j_1=0}^{+\infty}\left((2^{j_1})^{3+\delta}\|P_{j_1}R_{k_1}\partial\nabla u\|^2_{L^2}+(2^{j_1})^{1+\delta}\|P_{j_1}R_{k_1}Z^{\leq1}\partial\nabla u\|^2_{L^2}\right)\big)^{1/2}\Big)
\\&\times\Big((2^j)^{\frac52+28\delta}\|\partial_{t}v\|_{L^2}\Big)\Big((2^j)^{\frac32+25\delta}(2^k)^{1+\delta}\|R_k\partial_tv\|_{L^2}\Big).
\end{align*}
Thus
\begin{align}\label{bound3}
\sum_{j,k}IV_{2,1,1,1}^{\alpha=0}\leq\frac{C\varepsilon}{\delta^3}(1+t)^{-1}M_{high}(u(t)).
\end{align}
Now we treat terms $IV_{2,1,2}^{\alpha=0}$. By Lemma \ref{three}, to ensure that $\int_{\mathbb S^2}R_{k_1}fR_{k_2}gR_{k_3}hd_{S^2}$ is nonzero, we get
\begin{align}\label{add3}
\max\{k_1,k_2,k_3\}\leq med\{k_1,k_2,k_3\}+1.
\end{align}
By using \eqref{cc6} and lemma \ref{ALP_ortho},
\begin{align*}
IV_{2,1,2}^{\alpha=0}=-\frac12W_{j,k}^1\sum_{k_1,k_2,k_3=0}^{+\infty}\int_{\mathbb R^3}\Big(R_k(R_{k_1}f_0\partial_tR_{k_2}v)-R_{k_1}f_0R_kR_{k_2}\partial_{t}v\Big)R_{k_3}\partial_r\partial_tvdx.
\end{align*}
By using inequality \eqref{add3}, in order that the summation term does not vanish, the following three cases are distinguished according to the minimum of the indices
\begin{itemize}
  \item \textbf{case 1}: If $k_1=\min\{k_1,k_2,k_3\}$, then $k-2\leq k_2\leq k+2$ and $k-4\leq k_3\leq k+4$.
  \item \textbf{case 2}: If $k_2=\min\{k_1,k_2,k_3\}$, then $k_1\geq k-2$ and $k_1-4\leq k_3\leq k_1+4$.
  \item \textbf{case 3}: If $k_3=\min\{k_1,k_2,k_3\}$, then $k_1\geq k-2$ and $k_1-4\leq k_2\leq k_1+4$.
\end{itemize}
By lemma \ref{diff1}, we obtain
\begin{align}\label{xw14}
&\Big\|R_k(\sum_{k_1}R_{k_1}f_0\sum_{k_2}\partial_tR_{k_2}v)-\sum_{k_1}R_{k_1}f_0\sum_{k_2}R_kR_{k_2}\partial_{t}v\Big\|_{L^2(S_{\xi}^2)}
\nonumber\\ \leq &C(2^{k})^{-1}\left\|\sum_{k_1}\Omega (R_{k_1}f_0(t,r,\xi))\right\|_{L^{\infty}(S_{\xi}^2)}\left\|\sum_{k_2}R_{k_2}\partial_{t}v(t,r\xi)\right\|_{L^2(S_{\xi}^2)}.
\end{align}
\begin{itemize}
  \item In \textbf{case 1} under the additional condition $k\geq j+20$, since
\begin{align*}
\left\|\sum_{k_1=0}^{k+2}\Omega (R_{k_1}f_0(t,r,\xi))\right\|_{L^{\infty}}
&\leq C\left(\left\|\Omega \partial u\right\|_{L^{\infty}}+\left\|\partial u\right\|_{L^{\infty}}\right)
\\&\leq C(1+t)^{-1}\left(\left\|rL\partial u\right\|_{L^{\infty}}+\left\|r\nabla\partial u\right\|_{L^{\infty}}\right)+C\left\|\partial u\right\|_{L^{\infty}}
\\&\leq \frac{C}{\delta^2}(1+t)^{-1}\left(M_{high}(u(t))\right)^{1/2}+\frac{C\varepsilon}{\delta^3}(1+t)^{-1},
\end{align*}
where we used Lemma \ref{radial_decay}, then using the lower-order energies: $M_{low1}(u(t))$ and the second part of $M_{low2}(u(t))$, yields that $\sum_{j,k}IV_{2,1,2}^{\alpha=0}$
satisfies the same bound as those given in \eqref{bound3}.
   
  \item In \textbf{case 2}, noting $\sum_{k=j+20}^{+\infty}\sum_{k_1=k-2}^{+\infty}=\sum_{k_1=j+18}^{+\infty}\sum_{k=j+20}^{k_1+2},\quad\sum_{k=j+20}^{k_1+2}(2^k)^{1+2\delta}\leq C(2^{k_1})^{1+2\delta}$, and  using a modification of the proof on Lemma \ref{radial_decay} gives 
\begin{align*}
& \sum_{j,k}IV_{2,1,2}^{\alpha=0}
\\ \leq & \sum_{j=0}^{+\infty}\sum_{k_1=j+18}^{+\infty}\sum_{k_3=k_1-4}^{k_1+4}\sum_{k=j+20}^{k_1+2}(2^k)^{1+2\delta}(2^j)^{3+50\delta}\|\Omega R_{k_1}f_0\|_{L^{\infty}}\left\|\partial_{t}v\right\|_{L^2}
\left\|R_{k_3}\partial_t\partial_rv\right\|_{L^2}
\\ \leq & C\left(\sum_{k_1=18}^{+\infty}(2^{k_1})^{2\delta}\|\Omega R_{k_1}f_0\|^2_{L^{\infty}}\right)^{\frac12}\Big(M_{high}(u(t))\Big)^{\frac12}\left(\sum_{j=0}^{+\infty}\left((2^j)^{\frac52+25\delta}\left\|\partial_{t}v\right\|_{L^2}\right)^2\right)^{\frac12}
\\ \leq&\frac{C}{\delta}(1+t)^{-1}M_{high}(u(t))\left(M_{low1}(u(t))\right)^{\frac12},
\end{align*}
which satisfies the same bound as those given in \eqref{bound3}.
\item The treatment of \textbf{case 3} is analogous to that of \textbf{case 2} and we omit the details here.
\end{itemize}

We now turn to the term $IV_{2,2}$ in the case $\alpha=0$. We  write $IV_{2,2}^{\alpha=0}$ for $IV_{2,2}$ in this case. Using 
\[
\widetilde{\Omega}_{li}x^l=-2x^i,\ (\xi^l\partial_{\xi^i}-\xi^i\partial_{\xi^l})\langle \xi,\eta\rangle=-(\eta^{l}\partial_{\eta^i}-\eta^{i}\partial_{\eta^l})\langle \xi,\eta\rangle
\]
and identity \eqref{int_by_parts} yields
\begin{align*}
&\quad IV_{2,2}^{\alpha=0}-IV_{2,2,2}^{\alpha=0}-IV_{2,2,3}^{\alpha=0}
\\&=-W_{j,k}^1\int_0^{+\infty}\int_{S^2} C^{i0\gamma}\partial_{\gamma}u(t,r\eta)\frac{\eta^l}r\partial_{t}v(t,r\eta)\int_{S^2}\widetilde{\Omega}_{li}\partial_tv(t,r\xi)\widetilde{Q}_k(\langle \xi,\eta\rangle)d_{S^2(\xi)}d_{S^2(\eta)}r^2dr
\\&=-W_{j,k}^1\int_0^{+\infty}\int_{S^2}\int_{S^2} C^{i0\gamma}\partial_{\gamma}u(t,r\xi)\frac{\xi^l}r\partial_{t}v(t,r\xi)\widetilde{\Omega}_{li}\partial_tv(t,r\eta)\widetilde{Q}_k(\langle \xi,\eta\rangle)d_{S^2(\xi)}d_{S^2(\eta)}r^2dr,
\end{align*}
where the last equality follows from the change of variables $\xi\leftrightarrow\eta$ in the last equality and
\begin{align*}
&IV_{2,2,2}^{\alpha=0}=W_{j,k}^1\int R_k\left(C^{i0\gamma}\frac2r\partial_\gamma u\frac{x^i}r\partial_{t}v\right)R_k\partial_tvdx,
\\&IV_{2,2,3}^{\alpha=0}=W_{j,k}^1\int R_k\left(C^{i0\gamma}\widetilde{\Omega}_{il}\partial_\gamma u\frac{x^l}{r^2}\partial_{t}v\right)R_k\partial_tvdx. 
\end{align*}
Adding \eqref{angular1} and the above identity, we obtain
\begin{align*}
IV_{2,2}^{\alpha=0}=IV_{2,2,1}^{\alpha=0}+\frac12IV_{2,2,2}^{\alpha=0}+\frac12IV_{2,2,3}^{\alpha=0},
\end{align*}
where 
\begin{align*}
&IV_{2,2,1}^{\alpha=0}
\\=&\frac12W_{j,k}^1\int_0^{+\infty}\int_{S^2}\int_{S^2} \left(f_{i,l}(t,r,\eta)-f_{i,l}(t,r,\xi)\right)\partial_{t}v(t,r\xi)\left(\frac {\widetilde{\Omega}_{li}}r\partial_tv(t,r\eta)\right)\widetilde{Q}_k(\langle \xi,\eta\rangle)d_{S^2(\xi)}d_{S^2(\eta)}r^2dr,
\end{align*}
with $f_{i,l}(t,r,\xi)=C^{i0\gamma}\partial_{\gamma}u(t,r\xi)\xi^l$. The sum of $\frac12IV_{2,2,2}^{\alpha=0}$ and $\frac12IV_{2,1,1,2}^{\alpha=0}$ is zero. Terms $IV_{2,2,3}^{\alpha=0}$ can be treated analogously to  $IV_{2,1,1,1}^{\alpha=0}$, while terms $IV_{2,2,1}^{\alpha=0}$ are estimated similarly to $IV_{2,1,2}^{\alpha=0}$.
 
For $\alpha=m$, $m=1,2,3$, similarly, we get
\begin{align*}
IV^{\alpha=m}_2=\sum_{j,k}\int_0^T\left(IV^{\alpha=m}_{2,1}+IV^{\alpha=m}_{2,2}+IV^{\alpha=m}_{2,3}\right)dt+terms\ involving\ C^{i\alpha}u\partial_{i}\partial_{\alpha}v,
\end{align*}
where
\begin{align}
\nonumber&IV^{\alpha=m}_{2,1}=W_{j,k}^1\int_0^{+\infty}\int_{S^2} \left(C^{im\gamma}\partial_{\gamma}u(t,r\eta)\eta^i\eta^m\partial_r^2v(t,r\eta)\right)\int_{S^2}\partial_tv(t,r\xi)\widetilde{Q}_k(\langle \xi,\eta\rangle)d_{S^2(\xi)}d_{S^2(\eta)}r^2dr,
\\ \nonumber&IV^{\alpha=m}_{2,2}=W_{j,k}^1\int_0^{+\infty}\int_{S^2} C^{im\gamma}\partial_{\gamma}u(t,r\eta)\left(\frac{\eta^n}r\widetilde{\Omega}_{ni}(\frac{\eta^l}r\widetilde{\Omega}_{lm})v\right)(t,r\eta)\int_{S^2}\partial_tv(t,r\xi)\widetilde{Q}_k(\langle \xi,\eta\rangle)d_{S^2(\xi)}d_{S^2(\eta)}r^2dr,
\\&\nonumber IV^{\alpha=m}_{2,3}=W_{j,k}^1\int_0^{+\infty}\int_{S^2} C^{im\gamma}\partial_{\gamma}u(t,r\eta)\left(\eta^i\partial_r(\frac{\eta^l}r\widetilde{\Omega}_{lm})v\right)(t,r\eta)\int_{S^2}\partial_tv(t,r\xi)\widetilde{Q}_k(\langle \xi,\eta\rangle)d_{S^2(\xi)}d_{S^2(\eta)}r^2dr
\\&\nonumber\qquad\qquad+W_{j,k}^1\int_0^{+\infty}\int_{S^2} C^{im\gamma}\partial_{\gamma}u(t,r\eta)\left(\frac{\eta^l}r\widetilde{\Omega}_{li}(\eta^m\partial_r)v\right)(t,r\eta)\int_{S^2}\partial_tv(t,r\xi)\widetilde{Q}_k(\langle \xi,\eta\rangle)d_{S^2(\xi)}d_{S^2(\eta)}r^2dr.
\end{align}
We begin with terms $IV^{\alpha=m}_{2,1}$. By integration by parts,
\begin{align*}
IV^{\alpha=m}_{2,1}&=-W_{j,k}^1\int_0^{+\infty}\int_{S^2} \left(C^{im\gamma}\partial_{\gamma}u(t,r\eta)\eta^i\eta^m\partial_rv(t,r\eta)\right)\int_{S^2}\partial_t\partial_rv(t,r\xi)\widetilde{Q}_k(\langle \xi,\eta\rangle)d_{S^2(\xi)}d_{S^2(\eta)}r^2dr
\\&\quad+IV^{\alpha=m}_{2,1,2}+IV^{\alpha=m}_{2,1,3}
\\&=W_{j,k}^1\int_0^{+\infty}\int_{S^2} \left(C^{im\gamma}\partial_{\gamma}u(t,r\eta)\eta^i\eta^m\partial_t\partial_rv(t,r\eta)\right)\int_{S^2}\partial_rv(t,r\xi)\widetilde{Q}_k(\langle \xi,\eta\rangle)d_{S^2(\xi)}d_{S^2(\eta)}r^2dr
\\&\quad+IV^{\alpha=m}_{2,1,2}+IV^{\alpha=m}_{2,1,3}+IV^{\alpha=m}_{2,1,4}+IV^{\alpha=m}_{2,1,5}
\\&=W_{j,k}^1\int_0^{+\infty}\int_{S^2}\int_{S^2} \left(C^{im\gamma}\partial_{\gamma}u(t,r\xi)\xi^i\xi^m\partial_t\partial_rv(t,r\xi)\right)\partial_rv(t,r\eta)\widetilde{Q}_k(\langle \xi,\eta\rangle)d_{S^2(\xi)}d_{S^2(\eta)}r^2dr
\\&\quad+IV^{\alpha=m}_{2,1,2}+IV^{\alpha=m}_{2,1,3}+IV^{\alpha=m}_{2,1,4}+IV^{\alpha=m}_{2,1,5},
\end{align*}
where we have swapped the variables $\xi$ and $\eta$ in the first term of the last equality, and 
\begin{align*}
&IV^{\alpha=m}_{2,1,2}=-W_{j,k}^1\int C^{im\gamma}\frac2r\partial_{\gamma}u\frac{x^i}r\frac{x^m}r\partial_rvR_k(\partial_tv)dx,
\\&IV^{\alpha=m}_{2,1,3}=-W_{j,k}^1\int C^{im\gamma}\partial_r\partial_{\gamma}u\frac{x^i}r\frac{x^m}r\partial_rvR_k(\partial_tv)dx,
\\&IV^{\alpha=m}_{2,1,4}=W_{j,k}^1\int C^{im\gamma}\partial_t\partial_{\gamma}u\frac{x^i}r\frac{x^m}r\partial_rvR_k(\partial_tv)dx,
\\&IV^{\alpha=m}_{2,1,5}=-W_{j,k}^1\frac{d}{dt}\int \left(C^{im\gamma}\partial_{\gamma}u\frac{x^i}r\frac{x^m}r\partial_rv\right)R_k\partial_rvdx.
\end{align*}
Adding the first and last equalities in the above formula gives 
\begin{align*}
IV^{\alpha=m}_{2,1}=IV^{\alpha=m}_{2,1,1}+IV^{\alpha=m}_{2,1,2}+IV^{\alpha=m}_{2,1,3}+\frac12IV^{\alpha=m}_{2,1,4}+\frac12IV^{\alpha=m}_{2,1,5},
\end{align*}
where
\begin{align*}
&IV^{\alpha=m}_{2,1,1}
\\=&\frac12W_{j,k}^1\int_0^{+\infty}\int_{S^2}\int_{S^2} \left(\widetilde{f}_{i,m}(t,r,\xi)-\widetilde{f}_{i,m}(t,r,\eta)\right)\partial_t\partial_rv(t,r\xi)\partial_rv(t,r\eta)\widetilde{Q}_k(\langle \xi,\eta\rangle)d_{S^2(\xi)}d_{S^2(\eta)}r^2dr,
\end{align*}
with $\widetilde{f}_{i,m}(t,r,\eta)=C^{im\gamma}\partial_{\gamma}u(t,r\eta)\eta^i\eta^m$. Terms $IV^{\alpha=m}_{2,1,3},IV^{\alpha=m}_{2,1,4}$ can be treated analogous to $IV_{2,1,1,1}^{\alpha=0}$, while terms $IV_{2,1,1}^{\alpha=m}$ are estimated similarly to $IV_{2,1,2}^{\alpha=0}$. For terms $IV^{\alpha=m}_{2,1,5}$, we get
\begin{align*}
\int_0^T\sum_{j,k}IV^{\alpha=m}_{2,1,5}dt\leq \frac{C\varepsilon}{\delta^2}(M_{high}(u(0))+M_{high}(u(T))).
\end{align*}
Terms $IV^{\alpha=m}_{2,2},IV^{\alpha=m}_{2,3}$ are analogous to $IV^{\alpha=m}_{2,1}$. We get that $IV^{\alpha=m}_{2,2}$ is the sum of $IV^{\alpha=m}_{2,2,i}$ and $IV^{\alpha=m}_{2,3}$ is the sum of $IV^{\alpha=m}_{2,3,i}$, respectively, where $IV^{\alpha=m}_{2,2,i},IV^{\alpha=m}_{2,3,i}$ are similar to $IV^{\alpha=m}_{2,1,i}$ for $i=1,3,4,5$ and 
\begin{align*}
&IV^{\alpha=m}_{2,2,2}=W_{j,k}^1\int C^{im\gamma}\frac2r\partial_{\gamma}u\frac{x^i}r\frac{x^l}{r^2}\widetilde{\Omega}_{lm}vR_k(\partial_tv)dx,
\\&IV^{\alpha=m}_{2,3,2}=-IV^{\alpha=m}_{2,2,2}-IV^{\alpha=m}_{2,1,2}.
\end{align*}
Thus the sum of $IV^{\alpha=m}_{2,1,2},IV^{\alpha=m}_{2,2,2},IV^{\alpha=m}_{2,3,2}$ is zero. 

A similar discussion to that above leads to 
\[
M_{high2}(u(T))\leq CM_{high}(u(0))+\frac{C\varepsilon}{\delta^3}\int_0^T(1+t)^{-1}M_{high}(u(t))dt.
\]

\subsubsection{\textbf{Estimates for $M_{high4}(u(t))$}}\label{sec_high4}
A straight calculation shows that
\begin{align}
\label{NLG_LP2}\Box (Zu)_{j,k}+(Zu)_{j,k}
=&R_kP_j\left((C^{i\alpha\gamma}\partial_{\gamma}u+C^{i\alpha}u)\partial_{i}\partial_{\alpha}Zu\right)+R_kP_j\left((C^{i\alpha\gamma}\partial_{\gamma}Zu+C^{i\alpha}Zu)\partial_{i}\partial_{\alpha}u\right)
\nonumber\\&+R_kP_j\left(\left(\overline{C}^{\alpha\beta\gamma}\partial_{\gamma}u+\overline{C}^{\alpha\beta}u\right)\partial_{\alpha}\partial_{\beta}u+Z(H(u))\right).
\end{align}
Let 
\[
W^2_{j,k}=(2^{j})^{1+40\delta}(2^{k})^{2+2\delta},\ \sum_{j,k}=\sum_{j=0}^{+\infty}\sum_{k=j+20}^{+\infty}.
\]
We now perform energy estimates for equation \eqref{NLG_LP2} with $k\geq j+20$ and $j\geq0$. Multiplying the equation by $\partial_t(Zu)_{j,k}$ and integrating over $\mathbb R^3$ yields 
\begin{align*}
&\sum_{j,k}(2^{j})^{1+40\delta}(2^{k})^{2+2\delta}\frac{d}{dt}\left(\left\|\partial^{\leq1} (Zu)_{j,k}\right\|^2_{L^2(\mathbb R^3)}\right)=V_1+V_2+V_3+V_4,
\end{align*}
where
\begin{align*}
&V_1=\mathcal{O}(1)\sum_{j,k}W^2_{j,k}\int R_k\left(P_j(\partial^{\leq1}u\partial\nabla Zu)-\partial^{\leq1}u\partial\nabla P_jZu\right)\partial_t(Zu)_{j,k}dx,
\\&V_2=\mathcal{O}(1)\sum_{j,k}W^2_{j,k}\int R_k\left(\partial^{\leq1}u\partial\nabla P_j(Zu)\right)\partial_t(Zu)_{j,k}dx,
\\&V_3=\sum_{j,k}W^2_{j,k}\int R_kP_j\left((C^{i\alpha\gamma}\partial_{\gamma}Zu+C^{i\alpha}Zu)\partial_{i}\partial_{\alpha}u\right)\partial_t(Zu)_{j,k}dx,
\\&V_4=\sum_{j,k}W^2_{j,k}\int R_kP_j\left(\left(\overline{C}^{\alpha\beta\gamma}\partial_{\gamma}u+\overline{C}^{\alpha\beta}u\right)\partial_{\alpha}\partial_{\beta}u+Z(H(u))\right)\partial_t(Zu)_{j,k}dx.
\end{align*}

We first focus on the terms $V_3$. Using
\begin{align*}
R_kP_j(\partial^{\leq1}Zu\partial_{i}\partial_{\alpha}u)&=\sum_{k_1,k_2,j_1,j_2=0}^{+\infty}R_kP_j\left(\left(R_{k_1}P_{j_1}\partial^{\leq1}Zu\right)\left(R_{k_2}P_{j_2}\partial_{\alpha}\partial_{i}u\right)\right),
\end{align*}
it suffices to consider the following nine cases: \textbf{Case 1A}; \textbf{Case 1B}; \textbf{Case 1C}; \textbf{Case 2A}; \textbf{Case 2B}; \textbf{Case 2C}; \textbf{Case 3A}; \textbf{Case 3B}; \textbf{Case 3C}, where \textbf{Case 1A} denotes \textbf{Case 1} and \textbf{Case A}, with 
\begin{itemize}
  \item \textbf{Case 1}: $j_1=\min\{j,j_1,j_2\}$; \textbf{Case 2}: $j_2=\min\{j,j_1,j_2\}$; \textbf{Case 3}: $j=\min\{j,j_1,j_2\}$.
  \item \textbf{Case A}: $k_1=\min\{k,k_1,k_2\}$; \textbf{Case B}: $k_2=\min\{k,k_1,k_2\}$; \textbf{Case C}: $k=\min\{k,k_1,k_2\}$.
\end{itemize}
Applying Lemma \ref{LP2} and Lemma \ref{ALP_ortho} gives $\max\{j,j_1,j_2\}\leq med\{j,j_1,j_2\}+2$ and  $\max\{k,k_1,k_2\}\leq med\{k,k_1,k_2\}+2$. In \textbf{Case 1A}, using H\"older inequality, identity \eqref{space_div}, Lemma \ref{radial_decay} and Bernstein's inequality yields terms $V_3$ are bounded by
\begin{align*}
&C\sum_{j,k}W^2_{j,k}\sum_{k_2=k-2}^{k+2}\sum_{j_2=j-2}^{j+2}\|\left(R_{\leq k}P_{\leq j}\partial^{\leq1}Zu\right)\left(R_{k_2}P_{j_2}\partial_{\alpha}\partial_{i}u\right)\|_{L^2}\|\partial_t(Zu)_{j,k}\|_{L^2}
\\ \leq &\frac{C}{1+t}\left(\left\|\langle r\rangle(\partial^{\leq1} Zu)\right\|_{L^{\infty}}(M_{low}(u(t)))^{\frac12}+\left\|\partial^{\leq1} Zu\right\|_{L^{\infty}}(M_{high}(u(t)))^{\frac12}\right)(M_{high4}(u(t)))^{\frac12},
\end{align*}
with the same bound as those in \eqref{bound3}. In \textbf{Case 1B}, applying identity \eqref{space_div} and Lemma \ref{product_radial} gives
\begin{align*}
&\quad\left\|\sum_{k_2=1}^k\left(R_{k_1}P_{\leq j}\partial^{\leq1}Zu\right)\left(R_{k_2}P_{j_2}\partial_{\alpha}\partial_{i}u\right)\right\|_{L^2(\mathbb R^3)}
\\&\leq \frac{C}{1+t}\left(\left\|\left(R_{k_1}P_{\leq j}\partial^{\leq1}Zu\right)\sum_{k_2=1}^k\left(R_{k_2}ZP_{j_2}\partial u\right)\right\|_{L^2}+\left\|\left(R_{k_1}P_{\leq j}\partial^{\leq1}Zu\right)\sum_{k_2=1}^{k}\left(R_{k_2}(xP_{j_2}\partial\partial_tu)\right)\right\|_{L^2}\right)
\\&\leq C(1+t)^{-1}\bigg(\sum_{j_3=j_2-2}^{j_2+2}\left\|R_{k_1}P_{\leq j}\partial^{\leq1}Zu\right\|_{L^3}\left\|P_{j_3}Z\partial u\right\|_{L^6}
\\&\qquad\qquad\qquad\quad+\sum_{j_1=0}^j(2^{j_1})^{\frac12}\left\|R_{k_1}P_{j_1}\partial^{\leq1}Zu\right\|_{L^2}\left(\left\|P_{j_2}\partial\partial_tu\right\|_{L^2}+
\sum_{k_2=1}^{k+1}2^{k_2}\left\|R_{k_2}P_{j_2}\partial\partial_tu\right\|_{L^2}\right)\bigg).
\end{align*}
Using $\|R_0(xP_{j_2}\partial^2u)\|_{L^{\infty}}^2\leq C2^{j_2}\|P_{j_2}\partial^2u\|_{L^2}^2$, the embedding theorems $\dot{H}^1\hookrightarrow L^6$ and $\dot{H}^{\frac12}\hookrightarrow L^3$ shows that the bound on $V_3$ is the same as those in \eqref{bound3} in \textbf{Case 1B}. \textbf{Case 1C} is analogous to \textbf{Case 1B} and we obtain 
\begin{align*}
&\quad\left\|\left(R_{k_1}P_{\leq j}\partial^{\leq1}Zu\right)\left(R_{k_2}P_{j_2}\partial_{\alpha}\partial_{i}u\right)\right\|_{L^2(\mathbb R^3)}
\\&\leq C(1+t)^{-1}\bigg(\sum_{j_3=j_2-2}^{j_2+2}\left\|R_{k_1}P_{\leq j}\partial^{\leq1}Zu\right\|_{L^3}\left\|R_{k_2}P_{j_3}Z\partial u\right\|_{L^6}
\\&\qquad\qquad\qquad\quad+\sum_{j_1=0}^j(2^{j_1})^{\frac12}\left\|R_{k_1}P_{j_1}\partial^{\leq1}Zu\right\|_{L^2}\sum_{k_3=k_2-1}^{k_2+1}2^{k_3}\left\|R_{k_3}P_{j_2}\partial\partial_tu\right\|_{L^2}\bigg),
\end{align*}
which implies that terms $V_3$ satisfy the same bound as those in \eqref{bound3} in \textbf{Case 1C}. 
\textbf{Case 2A} and \textbf{Case 2C} are treated analogously to \textbf{Case 1B} and \textbf{Case 1C}, respectively. Employing inequality \eqref{small1} provides the bound for terms $V_3$ in \textbf{Case 2B}. Furthermore, \textbf{Case 3A} is treated similarly to \textbf{Case 1A}, while we deal with \textbf{Case 3B} in a similar fashion of \textbf{Case 1B} and handle \textbf{Case 3C} similarly to \textbf{Case 1C}. Thus terms $V_3$ satisfy the same bound as those in \eqref{bound3}.

Terms $V_4$ can be handled in a similar fashion of $V_3$ by using Lemma \ref{sob1} and inequality \eqref{small1}, so we omit the details here.

For terms $V_1$, applying Lemma \ref{ALP_ortho}, we estimate them as follows 
\[
V_1=V_{1,1}+V_{1,2}+V_{1,3},
\]
where
\begin{align*}
&V_{1,1}\leq C\sum_{j,k}W^2_{j,k}\sum_{k_2=k-2}^{k+2}\left\|P_j\left((R_{\leq k}\partial^{\leq1} u)(R_{k_2}\partial\nabla Zu)\right)-(R_{\leq k}\partial^{\leq1} u)R_{k_2}\partial\nabla (Zu)_j\right\|_{L^2}\|\partial_t(Zu)_{j,k}\|_{L^2},
\\&V_{1,2}\leq C\sum_{j,k}W^2_{j,k}\sum_{k_2=k-2}^{k+2}\left\|P_j\left((R_{k_2}\partial^{\leq1} u)(R_{\leq k}\partial\nabla Zu)\right)-(R_{k_2}\partial^{\leq1} u)R_{\leq k}\partial\nabla (Zu)_j\right\|_{L^2}\|\partial_t(Zu)_{j,k}\|_{L^2},
\\&V_{1,3}\leq C\sum_{j,k}W^2_{j,k}\sum_{k_1=k}^{+\infty}\sum_{k_2=k_1-2}^{k_1+2}\left\|P_j\left((R_{k_1}\partial^{\leq1} u)(R_{k_2}\partial\nabla Zu)\right)-(R_{k_1}\partial^{\leq1} u)R_{k_2}\partial\nabla (Zu)_j\right\|_{L^2}\|\partial_t(Zu)_{j,k}\|_{L^2}.
\end{align*}
For $k-2\leq k_2\leq k+2$, applying the discussion on terms $II_1$ to $V_{1,1}$(using Lemma \ref{para-product2}) gives
\begin{align*}
\left\|P_j\left((R_{\leq k}\partial^{\leq1} u)(R_{k_2}\partial\nabla Zu)\right)-(R_{\leq k}\partial^{\leq1} u)(R_{k_2}\partial\nabla P_jZu)\right\|_{L^2(\mathbb R^3)}\leq V_{1,1,1}+V_{1,1,2}+V_{1,1,3},
\end{align*}
where
\begin{align*}
&V_{1,1,1}=C\sum_{l=j-4}^{+\infty}\sum_{m=0}^{l+2}\left\| P_j\left(\left(P_{m} g\right)P_lf\right)\right\|_{L^2(\mathbb R^3)},
\\&V_{1,1,2}=C\sum_{l=j-4}^{+\infty}\left\|\left(P_{\leq l+2}P_jg\right)P_lf\right\|_{L^2(\mathbb R^3)},
\\&V_{1,1,3}=C2^{-j}\sum_{l=j-2}^{j+2}\left\|L(P_{\leq l-3}f,P_l g)\right\|_{L^2(\mathbb R^3)},
\end{align*} 
with
\[
f=R_{\leq k}\partial^{\leq1} u,\ g=R_{k_2}\partial\nabla Zu.
\]
We next consider terms $V_{1,1,1}$. By H\"older inequality and lemma \ref{sob_lemma},
\begin{align*}
V_{1,1,1}&\leq C\sum_{l=j-4}^{+\infty}\sum_{k_3=k_2-2}^{k_2+2}\sum_{m=0}^{l+2}2^m\left\|P_{m} R_{k_3}\partial Zu\right\|_{L^2(\mathbb R^3)}\left\|P_lR_{\leq k}\partial^{\leq1} u\right\|_{L^{\infty}(\mathbb R^3)}
\\&\leq \frac{C}{\delta}\sum_{k_3=k_2-2}^{k_2+2}\left(\sum_{m=0}^{+\infty}(2^m)^{1+40\delta}\left\|P_{m} R_{k_3}\partial Zu\right\|^2_{L^2(\mathbb R^3)}\right)^{\frac12}\sum_{l=j-4}^{+\infty}(2^l)^{\frac12-20\delta}\left\|P_l\partial^{\leq1} u\right\|_{L^{\infty}(\mathbb R^3)}
\\&\leq \frac{C\varepsilon}{\delta^2}(1+t)^{-1}(2^j)^{-\frac12-21\delta}\sum_{k_3=k_2-2}^{k_2+2}\left(\sum_{m=0}^{+\infty}(2^m)^{1+40\delta}\left\|P_{m} R_{k_3}\partial Zu\right\|^2_{L^2(\mathbb R^3)}\right)^{\frac12}.
\end{align*}
Then terms $V_{1,1}$ involving $V_{1,1,1}$ satisfy the same bound as those in \eqref{bound3}. Terms $V_{1,1,2}$ are similar to $V_{1,1,1}$. A similar fashion of discussion of $I_3$ yields
\begin{align*}
V_{1,1,3}&\leq C2^{-j}\sum_{l=j-2}^{j+2}\|\nabla P_{\leq l-3}R_{\leq k}\partial^{\leq1} u\|_{L^{\infty}}\|P_l R_{k_2}\partial\nabla Zu\|_{L^2}
\\&\leq C\left(\sum_{m=0}^{j-1}2^m\|P_m\partial^{\leq1} u\|_{L^{\infty}}\right)\sum_{l=j-2}^{j+2}\sum_{k_3=k_2-2}^{k_2+2}\|P_l R_{k_3}\partial Zu\|_{L^2},
\end{align*} 
which implies that terms $V_{1,1}$ involving $V_{1,1,3}$ have the same bound as those in \eqref{bound3}.

For terms $V_{1,2}$, a similar discussion of $V_{1,1}$ gives that
\begin{align*}
\left\|P_j\left((R_{k_2}\partial^{\leq1} u)(R_{\leq k}\partial\nabla Zu)\right)-(R_{k_2}\partial^{\leq1} u)(R_{\leq k}\partial\nabla P_jZu)\right\|_{L^2(\mathbb R^3)}\leq V_{1,2,1}+V_{1,2,2}+V_{1,2,3},
\end{align*}
where $V_{1,2,i}$ is equal to $V_{1,1,i}$ with 
\[
f=R_{k_2}\partial^{\leq1} u,\ g=R_{\leq k}\partial\nabla Zu.
\] 
Using identity \eqref{space_div}, Lemma \ref{xw1} and Lemma \ref{product_radial} gives 
\begin{align*}
V_{1,2,1}&\leq \frac{C}{1+t}\sum_{l=j-4}^{+\infty}\sum_{m=0}^{l+2}\left(\left\|\left( R_{\leq k}ZP_{m}\partial Zu\right)P_lR_{k_2}\partial^{\leq1} u\right\|_{L^2}+\left\|\left( R_{\leq k}(xP_{m}\partial_t\partial Zu)\right)P_lR_{k_2}\partial^{\leq1} u\right\|_{L^2}\right)
\\&\leq \frac{C}{1+t}\sum_{l=j-4}^{+\infty}\sum_{m=0}^{l+4}\left(\left\|P_{m} R_{\leq k}Z\partial Zu\right\|_{L^{\infty}}+(2^m)^{\frac12}\|P_m\partial_t\partial Zu\|_{L^2}\right)\left\|P_lR_{k_2}\partial^{\leq1} u\right\|_{L^2}
\\&\quad+\frac{C}{1+t}\sum_{l=j-4}^{+\infty}\sum_{m=0}^{l+4}(2^l)^{\frac12}\left(\sum_{k_1=1}^{k+1}2^{k_1}\left\|P_{m} R_{k_1}\partial_t\partial Zu\right\|_{L^2}+\left\|P_{m}\partial_t\partial Zu\right\|_{L^2}\right)\left\|P_lR_{k_2}\partial^{\leq1} u\right\|_{L^2},
\end{align*}
where we used $R_{\leq k}(xP_{m}\partial_t\partial Zu)=R_{0}(xP_{m}\partial_t\partial Zu)+\sum_{k_2=1}^kR_{k_2}(xP_{m}\partial_t\partial Zu)$.
A similar fashion of discussion holds for $V_{1,2,2}$. Similarly to $I_3$,
\begin{align*}
V_{1,2,3}\leq C\sum_{m=0}^{j-1}\|\nabla P_{m}R_{k_2}\partial^{\leq1} u\|_{L^{\infty}}\sum_{l=j-2}^{j+2}\|P_l \partial Zu\|_{L^2}.
\end{align*}
Thus using Lemma \ref{sob1} shows that the bound of $V_{1,2}$ is the same as those in \eqref{bound3}. Similarly to $V_{1,1}$, it's not difficult to get the bound of $V_{1,3}$ by using $M_{low4}(u(t))$ and $M_{low2}(u(t))$. Thus terms $V_1$ have the same bound as those in \eqref{bound3}.

Thus it suffices to estimate the most difficult terms $V_2$. Let $v=(Zu)_{j}$ in Section \ref{sec_high4}. Since the exponent $46\delta$ (from $(2^j)^{46\delta}$ in $M_{low3}$) exceeds the sum of $\delta$ and the exponent $40\delta$ (from $(2^j)^{40\delta}$ in $M_{high4}(u(t))$), and since $30\delta+12\delta$ derived from $(2^j)^{30\delta}(2^k)^{12\delta}$ in $M_{low4}$ is no less than $40\delta+2\delta$ (from $(2^j)^{40\delta}(2^k)^{2\delta}$ in $M_{high4}(u(t))$), it follows that $V_{2,1,1,1}^{\alpha=0}$ and $V_{2,1,2}^{\alpha=0}$ are acceptable, respectively. 
Therefore, the same reasoning as that used for $IV_2$ can be applied directly to $V_2$.

The combination of discussion above shows that
\[
M_{high4}(u(T))\leq CM_{high}(u(0))+\frac{C\varepsilon}{\delta^3}\int_0^T(1+t)^{-1}M_{high}(u(t))dt.
\]
\subsubsection{\textbf{Conclusion of higher-order energy estimates}}
In view of the preceding estimates, for every $T>0$, we get
\[
M_{high}(u(T))\leq C_2\varepsilon^2+\frac{C_1\varepsilon}{\delta^3}\int_0^T(1+t)^{-1}M_{high}(u(t))dt.
\]
Gronwall's inequality now implies
\[
M_{high}(u(t))\leq C_2\varepsilon^2(1+t)^{\frac{C_1\varepsilon}{\delta^3}}.
\]
Thus for any fixed small positive constant $\delta$, if $\varepsilon$ in $(0,\delta^6]$ is small enough, we have
\[
M_{high}(u(t))\leq C_2\varepsilon^2(1+t)^{\delta^2}.
\]

\subsection{Lower-order energy estimates}\label{sec_low}
We now proceed to the ghost energy estimates for system \eqref{NLG1}, where Lemma \ref{product_KG} plays a pivotal role in estimating $M_{low}(u(t))$. First, multiplying $\Box v+v$ by $e^{\arctan(r-t)}\partial_tv$ yields the following identity
\begin{align}\label{ghost1}
\big(e^{\arctan(r-t)}\partial_tv\big)\big(\Box v+v \big)
&=\frac12\partial_t\big(e^{\arctan(r-t)}\left(|\partial v|^2+v^2\right)\big)-\partial_k(\partial_kve^{\arctan(r-t)}\partial_tv)
\nonumber\\&\quad+\frac12\frac1{1+(t-r)^2}\left(|\partial v|^2+2\partial_rv\partial_tv+v^2\right)e^{\arctan(r-t)}.
\end{align}
Employing the angular-radial decomposition of the gradient $\nabla=\frac xr\partial_r-\frac x{r^2}\wedge\Omega$ implies
\[
|\partial v|^2+2\partial_rv\partial_tv=\left((\partial_t+\partial_r)v\right)^2+\frac1{r^2}|\Omega v|^2\geq 0.
\]

\subsubsection{\textbf{Estimates for $M_{low1}(u(t))$}}\label{sec_low1}
Let
\[
N_{1,j}(u(t))=\int e^{\arctan(r-t)}\left(|\partial^{\leq1} u_j|^2\right)dx.
\]
Direct computations lead to the following key identity, with the quasilinear term rewritten in divergence form:
\begin{align}\label{div1}
2\partial_{\gamma}u\partial_{i}\partial_{\alpha}u=\partial_i(\partial_{\gamma}u\partial_{\alpha}u)+\partial_{\alpha}(\partial_iu\partial_{\gamma}u)-\partial_{\gamma}(\partial_iu\partial_{\alpha}u).
\end{align}
We further express the nonlinear terms as follows for convenience:
\begin{align*}
&C^{i\alpha\gamma}\partial_{\gamma}u\partial_{i}\partial_{\alpha}u=\mathcal{O}(\partial(\partial u\partial u)),
\\&C^{i\alpha}u\partial_{i}\partial_{\alpha}u=\mathcal{O}(\nabla(u\partial u))+\mathcal{O}(\partial^{\leq1}u\partial^{\leq1}u).
\end{align*}
Set $v=u_j$ in equality \eqref{ghost1}. Integrating \eqref{ghost1} on $\mathbb R^3$and carrying out integration by parts, we get
\begin{align}\label{low1}
\frac d{dt}N_{1,j}(u(t))+\int e^{\arctan(r-t)}\frac1{1+(t-r)^2}\left((u_j)^2+\left|(\partial_t+\partial_r)u_j\right|^2\right)dx\leq VI,
\end{align}
where
\begin{align*}
VI=\sum_{j_1=0}^{+\infty}\sum_{j_2=0}^{+\infty}\int \Big(P_j\left(\mathcal{O}(\partial(\partial u_{j_1}\partial u_{j_2}))+\mathcal{O}(\nabla(u_{j_1}\partial u_{j_2}))+\mathcal{O}(\partial^{\leq1} u_{j_1}\partial^{\leq1}u_{j_2})\right)\Big)e^{\arctan(r-t)}\partial_tu_jdx,
\end{align*} 

We then apply Lemma \ref{product_KG} to estimate the term $VI$ and deduce that
\[
VI\leq VI_1+VI_2+VI_3+VI_4+good\ terms,
\]
where
\begin{align*}
&VI_{1}=\mathcal{O}(1)\sum_{j_1,j_2=0}^{+\infty}\int \partial\Big(P_j\big(\Box(\partial u_{j_1}\partial u_{j_2})+\partial u_{j_1}\partial u_{j_2}\big)\Big)e^{\arctan(r-t)}\partial_tu_jdx,
\\&VI_{2}=\mathcal{O}(1)\sum_{j_1,j_2=0}^{+\infty}\int\partial\Big(P_j\big(\partial u_{j_1}(\Box \partial u_{j_2}+\partial u_{j_2})\big)\Big)e^{\arctan(r-t)}\partial_tu_jdx,
\\&VI_{3}=\mathcal{O}(1)\sum_{j_1,j_2=0}^{+\infty}\int\partial\Big(P_j\big(\partial u_{j_2}(\Box \partial u_{j_1}+\partial u_{j_1})\big)\Big)e^{\arctan(r-t)}\partial_tu_jdx,
\\&VI_{4}=\mathcal{O}(1)\sum_{j_1,j_2=0}^{+\infty}\int\partial\Big(P_j\big(\partial_t\partial u_{j_1}\partial_t\partial u_{j_2}-\nabla \partial u_{j_1}\cdot\nabla \partial u_{j_2}+\partial u_{j_1}\partial u_{j_2}\big)\Big)e^{\arctan(r-t)}\partial_tu_jdx.
\end{align*}
Here, the notation "good terms" refers to the terms involving $\mathcal{O}(\nabla(u_{j_1}\partial u_{j_2}))$ and $\mathcal{O}(\partial^{\leq1} u_{j_1}\partial^{\leq1}u_{j_2})$, which possess more favorable analytical properties. (For the terms involving $\mathcal{O}(\partial^{\leq1} u_{j_1}\partial^{\leq1}u_{j_2})$, we can shift the time derivative $\partial_t$ from $\partial_tu_j$ onto $\mathcal{O}(\partial^{\leq1} u_{j_1}\partial^{\leq1}u_{j_2})$ via integration by parts in the time variable $t$.) We therefore omit the tedious computational details.

We now turn to the analysis of $VI_{2}$. Then
\begin{align*}
VI_{2}&\leq C\sum_{j_1,j_2=0}^{+\infty}\int \left(\left|P_j\Big(\partial u_{j_1}P_{j_2}\partial^{2}(\Box u+u)\Big)\right|+\left|P_j\Big(\partial^{2} u_{j_1}P_{j_2}\partial(\Box u+u)\Big)\right|\right)|\partial_tu_j|dx
\\&\leq VI_{2,1}+VI_{2,2}+good\ terms\ involving \ C^{i\alpha}u\partial_i\partial_{\alpha}u\ and \ semilinear\ terms\ H(u),
\end{align*}
where
\begin{align*}
&VI_{2,1}= C\sum_{j_1,j_2=0}^{+\infty}\left(\left\|P_j\Big(\partial^{2} u_{j_1}P_{j_2}\partial(\partial u\partial\nabla u)\Big)\right\|_{L^2}\right)\|\partial_t u_j\|_{L^2},
\\&VI_{2,2}=C\sum_{j_1,j_2=0}^{+\infty}\left(\left\|P_j\Big(\partial u_{j_1}P_{j_2}\partial^{2}(\partial u\partial\nabla u)\Big)\right\|_{L^2}\right)\|\partial_t u_j\|_{L^2}.
\end{align*}
Using $\partial_t^2u=(\Box u+u)+\Delta u-u$ gives $\|\partial^3u\|_{L^2}+\|\partial^2u\|_{L^2}\leq C\varepsilon$, which implies
\begin{align}\label{dd1}
\|\partial(\partial u\partial^2 u)\|_{L^{2}}+\|\partial^2(\partial u\partial u)\|_{L^{2}} \leq C\left(\|\partial^2u\|_{L^{\infty}}\|\partial^2 u\|_{L^2}+\|\partial u\|_{L^{\infty}}\|\partial^3 u\|_{L^2}\right) \leq \frac{C\varepsilon^2}{\delta^3}(1+t)^{-1}.
\end{align}
For terms $VI_{2,1}$, we consider three cases: $j=\min\{j,j_1,j_2\}$, $j_1=\min\{j,j_1,j_2\}$, $j_2=\min\{j,j_1,j_2\}$. Using estimate \eqref{dd1}, inequality  \eqref{small2}, Lemma \ref{LP2}, notation \eqref{sum_syb} and $(1+j)/(2^{j\delta})\leq C/\delta$ gives
\begin{align*}
VI_{2,1}&\leq C\sum_{j_1\thicksim j_2\backsim j}\left(\left\|\partial^{2} u_{j_1}\right\|_{L^{\infty}}\|P_{j_2}\partial(\partial u\partial\nabla u)\|_{L^{2}}\|\partial_t u_j\|_{L^2}\right)
\leq \frac{C\varepsilon^4}{\delta^5}(1+t)^{-2}(2^j)^{-\frac52}.
\end{align*}
A straight calculation shows that
\begin{align*}
\|P_{j_2}\partial^{2}(\partial u\partial\nabla u)\|_{L^2}\leq C\sum_{j_3\backsim j_4\backsim j_2}\left\|P_{j_2}(\partial u_{j_3}\partial^3\nabla u_{j_4})\right\|_{L^2}+C\|\partial^2 u\|_{L^{\infty}}\|\partial^3 u\|_{L^2}
\leq \frac{C\varepsilon^2}{\delta}(1+t)^{-1}(2^{j_2})^{\frac12},
\end{align*}
which combining Lemma \ref{LP2} implies
\begin{align*}
VI_{2,2} \leq \frac{C\varepsilon^4}{\delta^5}(1+t)^{-2}(2^j)^{-2-28\delta},
\end{align*}
where we consider three cases: $j=\min\{j,j_1,j_2\},j_1=\min\{j,j_1,j_2\},j_2=\min\{j,j_1,j_2\}$.

Terms $VI_{3}$ are the same as $VI_{2}$. For terms $VI_1$, by integration by parts, we have
\begin{align*}
VI_{1}\leq VI_{1,1}+VI_{1,2}+VI_{1,3}+VI_{1,4}+VI_{1,5},
\end{align*}
where $\partial_{t,r}=(\partial_t,\partial_r)$ and
\begin{align*}
&VI_{1,1}=\mathcal{O}(1)\sum_{j_1,j_2=0}^{+\infty}\frac d{dt}\int \partial\left(P_j\left(\partial_t(\partial u_{j_1}\partial u_{j_2})\right)\right)e^{\arctan(r-t)}\partial_tu_jdx
\\&\qquad\qquad+\mathcal{O}(1)\sum_{j_1,j_2=0}^{+\infty}\frac{d}{dt}\int \partial\left(P_j\left(\partial u_{j_1}\partial u_{j_2}\right)\right)e^{\arctan(r-t)}\partial_t^2u_jdx,
\\&VI_{1,2}=\mathcal{O}(1)\sum_{j_1,j_2=0}^{+\infty}\int \frac1{1+(t-r)^2}\partial_{t,r}\left(P_j(\partial(\partial u_{j_1}\partial u_{j_2}))\right)e^{\arctan(r-t)}\partial_tu_jdx,
\\&VI_{1,3}=\mathcal{O}(1)\sum_{j_1,j_2=0}^{+\infty}\int \frac1{1+(t-r)^2}\partial\left(P_j\left(\partial u_{j_1}\partial u_{j_2}\right)\right)e^{\arctan(r-t)}\partial_t^2u_jdx,
\\&VI_{1,4}=\mathcal{O}(1)\sum_{j_1,j_2=0}^{+\infty}\int \frac1{1+(t-r)^2}\partial\left(P_j\left(\partial u_{j_1}\partial u_{j_2}\right)\right)e^{\arctan(r-t)}\partial_t\partial_ru_jdx,
\\&VI_{1,5}=C\int \left|P_j\big(\partial(\partial u\partial u)\big)\right||\Box\partial_tu_j+\partial_tu_j|dx.
\end{align*}

For terms $VI_{1,2}$, by H\"older inequality, Young's inequality and inequality \eqref{dd1},
\begin{align*}
VI_{1,2}\leq \frac{C\varepsilon^4}{\delta^6}(1+t)^{-2}+\frac1{32}\int \frac1{1+(t-r)^2}e^{\arctan(r-t)}|(\partial_t+\partial_r)u_j|^2dx+VI_{1,2,1},
\end{align*}
where using integration by parts with respect to $\partial_r$ and Hardy's inequality yields
\begin{align*}
VI_{1,2,1}&=-\mathcal{O}(1)\sum_{j_1,j_2=0}^{+\infty}\int \frac1{1+(t-r)^2}\partial_{t,r}\left(P_j(\partial(\partial u_{j_1}\partial u_{j_2}))\right)e^{\arctan(r-t)}\partial_ru_jdx
\\&\leq C\left\|\nabla^{\leq 1}\partial^2\left(P_j(\partial u\partial u)\right)\right\|^2_{L^2}+\frac1{32}\int \frac1{1+(t-r)^2}e^{\arctan(r-t)}|u_j|^2dx
\\&\leq \frac{C\varepsilon^4}{\delta^6}(1+t)^{-2}(2^{j})^{2}+\frac1{32}\int \frac1{1+(t-r)^2}e^{\arctan(r-t)}|u_j|^2dx.
\end{align*}

Now let's deal with terms $VI_{1,3}$. By integration by parts with respect to $t$, we have
\begin{align*}
VI_{1,3}\leq VI_{1,3,1}+VI_{1,3,2},
\end{align*}
where
\begin{align*}
&VI_{1,3,1}=\mathcal{O}(1)\sum_{j_1,j_2=0}^{+\infty}\frac{d}{dt}\int \frac1{1+(t-r)^2}\partial\left(P_j\left(\partial u_{j_1}\partial u_{j_2}\right)\right)e^{\arctan(r-t)}\partial_tu_jdx,
\\&VI_{1,3,2}=-\mathcal{O}(1)\sum_{j_1,j_2=0}^{+\infty}\int \partial_t\left(\frac1{1+(t-r)^2}\partial\left(P_j\left(\partial u_{j_1}\partial u_{j_2}\right)\right)e^{\arctan(r-t)}\right)\partial_tu_jdx.
\end{align*}
The terms $VI_{1,3,2}$ can be treated similarly to $VI_{1,2}$, yielding
\begin{align*}
VI_{1,3,2} \leq &\frac{C\varepsilon^4}{\delta^6}(1+t)^{-2}(2^{j})^{2}+\frac1{32}\int \frac1{1+(t-r)^2}e^{\arctan(r-t)}|(\partial_t+\partial_r)u_j|^2dx
\\&+\frac1{32}\int \frac1{1+(t-r)^2}e^{\arctan(r-t)}|u_j|^2dx.
\end{align*}

For terms $VI_{1,4}$, we have 
\begin{align*}
VI_{1,4}=VI_{1,4,1}-VI_{1,3},
\end{align*}
where
\begin{align*}
VI_{1,4,1}=\mathcal{O}(1)\sum_{j_1,j_2=0}^{+\infty}\int \frac1{1+(t-r)^2}\partial\left(P_j\left(\partial u_{j_1}\partial u_{j_2}\right)\right)e^{\arctan(r-t)}\partial_t(\partial_t+\partial_r)u_jdx.
\end{align*}
Integrating by parts with respect to $\partial_t$ and using a similar discussion as for terms $VI_{1,2}$ yields
\begin{align*}
VI_{1,4,1} \leq \frac{C\varepsilon^4}{\delta^6}(1+t)^{-2}+\frac1{32}\int \frac1{1+(t-r)^2}e^{\arctan(r-t)}|(\partial_t+\partial_r)u_j|^2dx+VI_{1,4,2},
\end{align*}
where 
\begin{align*}
VI_{1,4,2}=\mathcal{O}(1)\sum_{j_1,j_2=0}^{+\infty}\frac{d}{dt}\int \frac1{1+(t-r)^2}\partial\left(P_j\left(\partial u_{j_1}\partial u_{j_2}\right)\right)e^{\arctan(r-t)}(\partial_t+\partial_r)u_jdx.
\end{align*}

We now turn to $VI_{1,5}$. Using \eqref{dd1} gives
\begin{align*}
VI_{1,5}\leq C\left\|P_j\big(\partial(\partial u\partial u)\big)\right\|_{L^2}\|\partial P_j(\partial^{\leq 2} u\partial^{\leq 1} u)\|_{L^2}
\leq \frac{C\varepsilon^4}{\delta^6}(1+t)^{-2}.
\end{align*}

Finally let's handle the terms $VI_{4}$ as follows
\begin{align*}
&VI_{4}=VI_{4,1}+VI_{4,2},
\end{align*}
where
\begin{align*}
&VI_{4,1}=\mathcal{O}(1)\sum_{j_1,j_2=0}^{+\infty}\int\partial\Big(P_j\Big(\partial_t\partial u_{j_1}\partial_t\partial u_{j_2}-\nabla \partial u_{j_1}\cdot\nabla \partial u_{j_2}+\partial u_{j_1}\partial u_{j_2}\Big)\Big)e^{\arctan(r-t)}(\partial_t+\partial_r)u_jdx,
\\&VI_{4,2}=\mathcal{O}(1)\sum_{j_1,j_2=0}^{+\infty}\int\partial\Big(P_j\Big(\partial_t\partial u_{j_1}\partial_t\partial u_{j_2}-\nabla \partial u_{j_1}\cdot\nabla \partial u_{j_2}+\partial u_{j_1}\partial u_{j_2}\Big)\Big)e^{\arctan(r-t)}\partial_ru_jdx.
\end{align*}
We begin with $VI_{4,1}$. Applying identities \eqref{cc5}, \eqref{lemm2_add} and the estimate $\|[S,P_j]f\|_{L^2}\leq C\|f\|_{L^2}$ gives
\[
\left\|\langle t-r\rangle\partial P_jf\right\|_{L^2}\leq C\left\|ZP_jf\right\|_{L^2}+C\left\|SP_jf\right\|_{L^2}\leq C\left\|Z^{\leq1}P_jf\right\|_{L^2}+C\left\|P_j(Sf)\right\|_{L^2},
\]
where $f=\partial_t\partial u_{j_1}\partial_t\partial u_{j_2}-\nabla \partial u_{j_1}\cdot\nabla \partial u_{j_2}+\partial u_{j_1}\partial u_{j_2}$. Using a modification of Lemma \ref{product_KG}, we can estimate $P_jS(\partial_t\partial u_{j_1}\partial_t\partial u_{j_2}-\nabla \partial u_{j_1}\cdot\nabla \partial u_{j_2}+\partial u_{j_1}\partial u_{j_2})$,
which implies
\begin{align*}
VI_{4,1}\leq VI_{4,1,1}+VI_{4,1,2}+\frac1{32}\int \frac1{1+(t-r)^2}e^{\arctan(r-t)}|(\partial_t+\partial_r)u_j|^2dx,
\end{align*}
where
\begin{align*}
&VI_{4,1,1}\leq C\sum_{j_1\backsim j_2\backsim j}\|(\partial^{\leq1}Z^{\leq1}\partial u_{j_1})\partial^{\leq1}\partial u_{j_2}\|^2_{L^2},
\\&VI_{4,1,2}\leq C\sum_{j_1\backsim j_2\backsim j}\left\|S\partial u_{j_2}(P_{j_1}\partial(\Box u+u))\right\|_{L^2}^2.
\end{align*}
We begin with $VI_{4,1,1}$. Using \eqref{small2} gives
\begin{align*}
VI_{4,1,1}\leq C\sum_{j_1\backsim j_2\backsim j}\Big(\|\partial^{\leq1}Z^{\leq1}\partial u_{j_1}\|^2_{L^2}\|\partial^{\leq1}\partial u_{j_2}\|^2_{L^{\infty}}\Big)
\leq \frac{C\varepsilon^4}{\delta^6}(1+t)^{-2}.
\end{align*}
For terms $VI_{4,1,2}$, since
\begin{align*}
&\sum_{j_3,j_4=0}^{+\infty}\|P_{j_1}(\partial^{\leq 1} u_{j_3}\partial^{\leq 3} u_{j_4})\|_{L^{2}}\leq C\sum_{j_3\backsim j_4\backsim j_1}\|\partial^{\leq 1} u_{j_3}\|_{L^{\infty}}\|\partial^{\leq 3} u_{j_4}\|_{L^{2}}\leq \frac{C\varepsilon^2}{\delta^3}(1+t)^{-1}(2^{j_1})^{-\frac12},
\\&\sum_{j_3=0}^{+\infty}\sum_{j_4\leq j_3}\|P_{j_1}(\partial^{\leq 2} u_{j_3}\partial^{\leq 2} u_{j_4})\|_{L^{2}}\leq C\sum_{\substack{j_3\backsim j_4\thicksim j_1\\j_4\leq j_3}}\|\partial^{\leq 2} u_{j_3}\|_{L^{2}}\|\partial^{\leq 2} u_{j_4}\|_{L^{\infty}}\leq \frac{C\varepsilon^2}{\delta^3}(1+t)^{-1}(2^{j_1})^{-\frac32},
\end{align*}
where we considered three cases: $j_1,j_3$ or $j_4$ is the minimum among $\{j_1,j_3,j_4\}$, it follows that
\begin{align*}
\|P_{j_1}\partial(\Box u+u)\|_{L^{2}}\leq \frac{C\varepsilon^2}{\delta^3}(1+t)^{-1}(2^{j_1})^{-\frac12},
\end{align*}
which combining the estimate $\|S\partial u_{j_2}\|_{L^{\infty}}\leq C(t\|\partial^2 u_{j_2}\|_{L^{\infty}}+\|r\nabla\partial u_{j_2}\|_{L^{\infty}})\leq \frac{C\varepsilon}{\delta}(2^{j_2})^{-\delta}$, implies 
\begin{align*}
VI_{4,1,2} &\leq C\sum_{j_1\backsim j_2\backsim j}\left\|S(\partial u_{j_2})\right\|_{L^{\infty}}^2\left\|P_{j_1}\partial(\Box u+u)\right\|_{L^{2}}^2
\leq \frac{C\varepsilon^6}{\delta^9}(1+t)^{-2}.
\end{align*}

For terms $VI_{4,2}$, by integration by parts, we have
\begin{align*}
VI_{4,2}\leq VI_{4,2,1}+VI_{4,2,2},
\end{align*}
where
\begin{align*}
&VI_{4,2,1}=C\sum_{j_1\backsim j_2\backsim j}\int\left|\nabla^{\leq1}\partial\Big(P_j\Big(\partial_t\partial u_{j_1}\partial_t\partial u_{j_2}-\nabla \partial u_{j_1}\cdot\nabla \partial u_{j_2}+\partial u_{j_1}\partial u_{j_2}\Big)\Big)\right|e^{\arctan(r-t)}|u_j|dx,
\\&VI_{4,2,2}=C\sum_{j_1\backsim j_2\backsim j}\int\left|\frac1r\partial\Big(P_j\Big(\partial_t\partial u_{j_1}\partial_t\partial u_{j_2}-\nabla \partial u_{j_1}\cdot\nabla \partial u_{j_2}+\partial u_{j_1}\partial u_{j_2}\Big)\Big)\right|e^{\arctan(r-t)}|u_j|dx.
\end{align*}
By a similar discussion of terms $VI_{4,1}$, if $0<\varepsilon\leq \delta^6$, we have
\[
VI_{4,2,1}\leq \frac{C\varepsilon^4}{\delta^6}(1+t)^{-2}(2^j)^2+\frac1{32}\int \frac1{1+(t-r)^2}e^{\arctan(r-t)}|u_j|^2dx.
\]
Applying Hardy's inequality, the terms $VI_{4,2,2}$ satisfy the same bounds as $VI_{4,2,1}$.

Combining the discussion above shows that there exists a positive constant $C_{low1}$ such that
\begin{align}\label{low1_per}
\frac{d}{dt}N_{1,j}(u(t))\leq C_{low1}\frac{\varepsilon^4}{\delta^6}(1+t)^{-2}(2^j)^2+VI_{1,1}+VI_{1,3,1}+VI_{1,4,2}.
\end{align}
Noting that $e^{-\pi/2}\leq e^{arctan(r-t)}\leq e^{\pi/2}$, we get $0.2\leq e^{arctan(r-t)}\leq 5$. Then there exists $C_0$ depending only on $\widetilde{C}$ such that
\begin{align}\label{C_0_1}
\sum_{j=0}^{+\infty}(2^{j})^{5+64\delta}N_{1,j}(u(0)) \leq \frac{C_0}{1000}\varepsilon^2.
\end{align}
By using the high-order energy estimates, we get
\[
N_{1,j}(u(t))\leq 5\|\partial^{\leq1} u_j\|^2_{L^2}\leq 5C_2\varepsilon^2(2^{j})^{-5-64\delta}(1+t)^{\delta^2},
\]
which implies that for $t\geq 0$,
\[
\frac{1}{\delta}\left(N_{1,j}(u(t))\right)^{\frac{1}{\delta}-1}\leq \frac{1}{\delta}(5C_2)^{\frac{1}{\delta}-1}(\varepsilon)^{\frac2{\delta}-2}(2^{j})^{-\frac5{\delta}-59+64\delta}(1+t)^{\delta-\delta^2}.
\]
Multiplying both sides of inequality \eqref{low1_per} by $\frac{1}{\delta}\left(N_{1,j}(u(t))\right)^{\frac{1}{\delta}-1}$, for $t\geq 0$, we have
\begin{align}\label{low2}
\frac d{dt}\left(\left(N_{1,j}(u(t))\right)^{\frac1{\delta}}\right)\leq& 
C_{low1}\frac{1}{\delta^7}(5C_2)^{\frac{1}{\delta}-1}(\varepsilon)^{\frac2{\delta}+2}(2^{j})^{-\frac5{\delta}-57+64\delta}(1+t)^{-2+\delta-\delta^2}
\nonumber\\&+\frac{1}{\delta}\left(N_{1,j}(u(t))\right)^{\frac{1}{\delta}-1}(VI_{1,1}+VI_{1,3,1}+VI_{1,4,2}).
\end{align}
Hence, for any $T\geq 0$ and $0<\delta<0.5$, integrating both sides of \eqref{low2} over the interval $[0,T]$ yields
\begin{align}\label{add6}
\left(N_{1,j}(u(T))\right)^{\frac1{\delta}}\leq& \left(N_{1,j}(u(0))\right)^{\frac1{\delta}}+\frac{2C_{low1}}{\delta^7}(5C_2)^{\frac{1}{\delta}-1}\varepsilon^{\frac2{\delta}+2}(2^j)^{-\frac5{\delta}-57+64\delta}
\nonumber\\&+\int_0^T\frac{1}{\delta}\left(N_{1,j}(u(t))\right)^{\frac{1}{\delta}-1}(VI_{1,1}+VI_{1,3,1}+VI_{1,4,2})dt.
\end{align}
Let 
\[
R_{1,1}=\mathcal{O}(1)\int \partial\left(P_j\left(\partial_t(\partial u\partial u)\right)\right)e^{\arctan(r-t)}\partial_tu_jdx
+\mathcal{O}(1)\int \partial\left(P_j\left(\partial u\partial u\right)\right)e^{\arctan(r-t)}\partial_t^2u_jdx.
\]
Thus
\begin{align*}
\left|R_{1,1}\right|\leq C\int|P_j(\partial^{\leq2}(\partial u\partial u))||\partial u_j|dx+C\int|P_j(\partial(\partial u\partial u))||\partial^2 u_j|dx
\leq \frac{C\varepsilon^3}{\delta^3}(1+t)^{-1}(2^{j})^{-\frac32},
\end{align*}
which implies
\begin{align}\label{add4}
\left|\frac{1}{\delta}\left(N_{1,j}(u(t))\right)^{\frac{1}{\delta}-1}R_{1,1}\right|\leq \frac{C}{\delta^4}(5C_2)^{\frac{1}{\delta}-1}(\varepsilon)^{\frac2{\delta}+1}(2^{j})^{-\frac5{\delta}-60.5+64\delta}.
\end{align}
Using \eqref{low1} gives
\begin{align}\label{add5}
\left|\frac{d}{dt}\left(N_{1,j}(u(t))\right)\right|\leq C\int \left|P_j\left(\partial^{\leq2} u\partial^{\leq1} u\right)\right||\partial_tu_j|dx\leq \frac{C\varepsilon^3}{\delta^3}(1+t)^{-1}(2^j)^{-\frac52},
\end{align} 
Applying \eqref{add4} and \eqref{add5}, for $\delta$ in $(0,0.5]$, we obtain
\begin{align}\label{add7}
&\quad\left|\int_0^T\frac{1}{\delta}\left(N_{1,j}(u(t))\right)^{\frac{1}{\delta}-1}(VI_{1,1})dt\right|
\nonumber\\&\leq \left|\int_0^T\frac{d}{dt}\left(\frac{1}{\delta}\left(N_{1,j}(u(t))\right)^{\frac{1}{\delta}-1}R_{1,1}\right)dt
-\int_0^T\left(\frac{1-\delta}{\delta^2}\left(N_{1,j}(u(t))\right)^{\frac{1}{\delta}-2}\right)\frac{d}{dt}\left(N_{1,j}(u(t))\right)R_{1,1}dt\right|
\nonumber\\&\leq \frac{C}{\delta^4}(5C_2)^{\frac{1}{\delta}-1}(\varepsilon)^{\frac2{\delta}+1}(2^{j})^{-\frac5{\delta}-60.5+64\delta}
+\frac{C}{\delta^8}(5C_2)^{\frac{1}{\delta}-2}(\varepsilon)^{\frac2{\delta}+2}(2^{j})^{-\frac5{\delta}-58+128\delta}.
\end{align}
Terms involving $VI_{1,3,1},VI_{1,4,2}$ can be treated similarly and they own the same bound of terms involving $VI_{1,1}$. Since for any positive constants $a,b$ and any $\delta$ in $(0,1)$, the inequality
\[
(a+b)^{\delta}\leq 2(a^{\delta}+b^{\delta}),
\]
holds, using \eqref{add6} and \eqref{add7}, we deduce that
\begin{align*}
N_{1,j}(u(T))\leq& 2N_{1,j}(u(0))+C\left(\frac{2C_{low1}}{\delta^7}\right)^{\delta}(5C_2)^{1-\delta}\varepsilon^{2+2\delta}(2^j)^{-5-57\delta+64\delta^2}
\\&+C\left(\frac{1}{\delta^4}\right)^{\delta}(5C_2)^{1-\delta}(\varepsilon)^{2+\delta}(2^{j})^{-5-60\delta+64\delta^2}
+C\left(\frac{1}{\delta^8}\right)^{\delta}(5C_2)^{1-2\delta}(\varepsilon)^{2+2\delta}(2^{j})^{-5-58\delta+128\delta^2}.
\end{align*}
Consequently, if $\delta$ in $(0,0.001]$ is a fixed constant and $\varepsilon$ in $(0,\delta^6]$ is small enough, for any positive time $t$, using \eqref{C_0_1} gives
\[
M_{low1}(u(t))\leq 5\sum_{j=0}^{+\infty}(2^j)^{5+56\delta}N_{1,j}(u(t))\leq \frac{C_0}{32}\varepsilon^2.
\]

\subsubsection{\textbf{Estimates for $M_{low3}(u(t))$}}\label{section1}
Here we do more refined handling to $\partial_{\gamma}u\partial_{i}\partial_{\alpha}u$ as follows: for $\alpha=\gamma=0$, we get $2\partial_tu\partial_{i}\partial_tu=\partial_i\left((\partial_tu)^2\right)$; and for the case that at least one of $\alpha,\gamma$ is $1,2,3$, we employ identity  \eqref{div1}. It is convenient to rewrite the nonlinear terms as follows.
\begin{align*}
&C^{i\alpha\gamma}Z(\partial_{\gamma}u\partial_{\alpha}\partial_iu)=\mathcal{O}(\nabla(\partial Zu\partial u))+\mathcal{O}(\partial_t(\nabla Zu\nabla u))+\mathcal{O}(\partial(\partial u\partial u)),
\\&C^{i\alpha}Z(u\partial_{\alpha}\partial_iu)=\mathcal{O}(\nabla(Zu\partial u+u\partial Zu))+\mathcal{O}(\partial(u\partial u))+\mathcal{O}(\partial Z^{\leq1}u\partial u).
\end{align*}
Therefore we have
\begin{align}\label{low6}
&\Box (Zu)_j+(Zu)_j
\nonumber\\=&P_j\left(\mathcal{O}(\nabla(\partial^{\leq1} Zu\partial^{\leq1} u))+\mathcal{O}(\partial_t(\nabla Zu\nabla u))\right)
+P_j\left(\mathcal{O}(\partial(\partial^{\leq1}u\partial u))\right)+P_j(\mathcal{O}(\partial^{\leq1}Z^{\leq1}u\partial^{\leq1}u)).
\end{align}

In Section \ref{section1}, let 
\begin{align*}
v=(Zu)_j,\quad N_{2,j}(u(t))=\int e^{\arctan(r-t)}\left(|\partial^{\leq1} (Zu)_j|^2\right)dx.
\end{align*}
Similarly to \eqref{low1},
\begin{align}\label{low4}
\frac12\frac d{dt}N_{2,j}(u(t))+\frac12\int e^{\arctan(r-t)}\frac1{1+(t-r)^2}\left(v^2+(\partial_tv+\partial_rv)^2\right)dx\leq VII_1+VII_2+VII_3,
\end{align}
where
\begin{align*}
&VII_1=\sum_{j_1,j_2=0}^{+\infty}\int P_j \left(\mathcal{O}\left(\partial_t(\nabla u_{j_1}\nabla (Zu)_{j_2})\right)+\mathcal{O}\left(\nabla(\partial^{\leq1}u_{j_1}\partial^{\leq1} (Zu)_{j_2})\right)\right)e^{\arctan(r-t)}\partial_tvdx,
\\&VII_2=\sum_{j_1,j_2=0}^{+\infty}\int P_j\mathcal{O}\left(\partial(\partial^{\leq1} u_{j_1}\partial u_{j_2})\right)e^{\arctan(r-t)}\partial_tvdx,
\\&VII_3=\sum_{j_1,j_2=0}^{+\infty}\int P_j\mathcal{O}\left(\partial^{\leq1}u_{j_1}\partial^{\leq1}(Z^{\leq1}u)_{j_2})\right)e^{\arctan(r-t)}\partial_tvdx.
\end{align*}

Terms $VII_2$ are similar to $VI$. Now we focus on $VII_1$.  Applying Lemma \ref{product_KG} to $VII_1$ yields
\[
VII_1\leq VII_{1,1}+VII_{1,2}+VII_{1,3}+VII_{1,4}+good \ terms,
\]
where 
\begin{align*}
&VII_{1,1}=\mathcal{O}(1)\sum_{j_1,j_2=0}^{+\infty}\int \partial_t\Big(P_j\left(\Box(w_1w_2)+w_1w_2\right)\Big)e^{\arctan(r-t)}\partial_tvdx,
\\&VII_{1,2}=\mathcal{O}(1)\sum_{j_1,j_2=0}^{+\infty}\int\partial_t\Big(P_j\big(w_1(\Box w_2+w_2)\big)\Big)e^{\arctan(r-t)}\partial_tvdx,
\\&VII_{1,3}=\mathcal{O}(1)\sum_{j_1,j_2=0}^{+\infty}\int\partial_t\Big(P_j\big(w_2(\Box w_1+w_1)\big)\Big)e^{\arctan(r-t)}\partial_tvdx,
\\&VII_{1,4}=\mathcal{O}(1)\sum_{j_1,j_2=0}^{+\infty}\int\partial_t\Big(P_j\Big(\partial_tw_1\partial_tw_2-\nabla w_1\cdot\nabla w_2+w_1w_2\Big)\Big)e^{\arctan(r-t)}\partial_tvdx
\\&\qquad\qquad+\mathcal{O}(1)\sum_{j_1,j_2=0}^{+\infty}\int\nabla\Big(P_j\Big(\partial_tw_3\partial_tw_4-\nabla w_3\cdot\nabla w_4+w_3w_4\Big)\Big)e^{\arctan(r-t)}\partial_tvdx,
\end{align*} 
with 
\[
w_1=\nabla u_{j_1},\quad w_2=\nabla (Zu)_{j_2},\quad w_3=\partial^{\leq1} u_{j_1},\quad w_4=\partial^{\leq1} (Zu)_{j_2}.
\]
The definitions of $w_1,w_2,w_3,w_4$ are applicable to, and only to, Section \ref{section1}.
Here, by "good terms", we mean terms involving $\nabla(\partial^{\leq1} u_{j_1}\partial^{\leq1} (Zu)_{j_2})$, which can be handled in a similar fashion to $VI_{i}$, $i=1,2,3$, using $\|\nabla P_jf\|_{L^2}\leq C2^j\|P_jf\|_{L^2}$ for any function $f(x)$.

Consider first $VII_{1,2}$. Then
\begin{align*}
VII_{1,2}&\leq C\sum_{j_1,j_2=0}^{+\infty}\int \left(\left|P_j\Big(\nabla u_{j_1}P_{j_2}\partial_t\nabla Z(\Box u+u)\Big)\right|+\left|P_j\Big(\partial_t\nabla u_{j_1}P_{j_2}\nabla Z(\Box u+u)\Big)\right|\right)|\partial_tv|dx
\\&\leq VII_{1,2,1}+VII_{1,2,2}+good\ terms \ involving \ semilinear\ terms\ H(u),
\end{align*}
where 
\begin{align*}
&VII_{1,2,1}\leq C\sum_{j_1,j_2=0}^{+\infty}\left\|P_j\left(\left(\partial_t\nabla u_{j_1}\right)\left(P_{j_2}\nabla Z(\partial^{\leq1} u\partial\nabla u)\right)\right)\left(\partial_tv\right)\right\|_{L^1},
\\&VII_{1,2,2}\leq C\sum_{j_1,j_2=0}^{+\infty}\left\|P_j\left(\left(\nabla u_{j_1}\right)\left(P_{j_2}\partial_t\nabla Z(\partial^{\leq1} u\partial\nabla u)\right)\right)\left(\partial_tv\right)\right\|_{L^1}.
\end{align*}
For terms $VII_{1,2,1}$, we consider three cases depending on which of $j,j_1,j_2$ is the minimum among $\{j,j_1,j_2\}$. Using Lemma \ref{LP2}, notation \eqref{sum_syb}, the identity \eqref{space_div} and the estimate
\begin{align*}
\|P_{j_2}\nabla Z(\partial^{\leq1} u\partial\nabla u)\|_{L^{2}} 
\leq C2^{j_2}\Big(\|\partial^{\leq1} Z^{\leq 1} u\|_{L^2}\|\partial\nabla u\|_{L^{\infty}}+\|\partial^{\leq1} u\|_{L^{\infty}}\|\partial^{\leq2} Z^{\leq 1} u)\|_{L^{2}}\Big)
\leq \frac{C\varepsilon^2}{\delta^3}\frac1{1+t}2^{j_2},
\end{align*}
we have
\begin{align*}
VII_{1,2,1}\leq& \frac{C}{1+t}\sum_{j_1\thicksim j_2\backsim j}\left(\left(\left\|Z\partial  u_{j_1}\right\|_{L^2}\|\partial_tv\|_{L^{\infty}}+\left\|\partial^2  u_{j_1}\right\|_{L^2}\|r\partial_tv\|_{L^{\infty}}\right)\|P_{j_2}\nabla Z(\partial^{\leq1} u\partial\nabla u)\|_{L^{2}}\right)
\\ \leq &\frac{C\varepsilon^4}{\delta^4}(1+t)^{-2+2\delta^2}2^j.
\end{align*}

Using inequality \eqref{small1}, Lemma \ref{LP2}, notation \eqref{sum_syb} and Lemma \ref{radial_decay} yields
\begin{align}\label{add15}
&\quad\left\|P_{j_2}\nabla\partial_tZ(\partial^{\leq1} u\partial\nabla u)\right\|_{L^2(\mathbb R^3)}
\nonumber\\& \leq C2^{j_2}\left(\frac{\varepsilon^2}{\delta^3}(1+t)^{-1}+\left\|P_{j_2}(\partial^{\leq1} u\partial\nabla\partial_tZ u)\right\|_{L^2(\mathbb R^3)}+\left\|P_{j_2}(\partial^{\leq1} Zu\partial^2\nabla u)\right\|_{L^2(\mathbb R^3)}\right)
\nonumber\\& \leq \frac{C\varepsilon^2}{\delta^3}2^{j_2}(1+t)^{-1}+C2^{j_2}\sum_{j_3,j_4=0}^{+\infty}\left(\left\|P_{j_2}(\partial^{\leq1} u_{j_3}\partial^2\nabla (Zu)_{j_4})\right\|_{L^2}+\left\|P_{j_2}(\partial^{\leq1} (Zu)_{j_3}\partial^2\nabla u_{j_4})\right\|_{L^2}\right)
\nonumber\\& \leq \frac{C\varepsilon^2}{\delta^3}2^{j_2}(1+t)^{-1}+C(1+t)^{-1}2^{j_2}\sum_{j_3\backsim j_4\backsim j_2}\left(\left\|\langle r\rangle\partial^{\leq1} u_{j_3}Z\partial^2 (Zu)_{j_4}\right\|_{L^2}+\left\|\langle r\rangle\partial^{\leq1} (Zu)_{j_3}Z\partial^2 u_{j_4}\right\|_{L^2}\right)
\nonumber\\& \leq \frac{C\varepsilon^2}{\delta^3}2^{j_2}(1+t)^{-1}+\frac{C}{\delta^3}(1+t)^{-1+\delta^2}2^{j_2}\varepsilon^2(2^{j_2})^{\frac12},
\end{align}
which combining a similar fashion of discussion of $VII_{1,2,1}$ implies
\begin{align*}
VII_{1,2,2} \leq \frac{C\varepsilon^4}{\delta^4}(1+t)^{-2+2\delta^2}(2^j)^{\frac32}.
\end{align*}

Now let's consider terms $VII_{1,3}$. Using Lemma \ref{LP2} and notation \eqref{sum_syb} yields that
\begin{align*}
VII_{1,3}\leq VII_{1,3,1}+good\ terms\ invovling\ semilinear\ terms,
\end{align*}
where
\begin{align*}
VII_{1,3,1}=C\sum_{j_1\backsim j_2\backsim j}\|P_j\big(\partial^{\leq1} \nabla (Zu)_{j_2}P_{j_1}\partial^{\leq1}\nabla(\partial^{\leq1} u\partial\nabla u)\big)\partial_tv\|_{L^1},
\end{align*}
Using Lemma \ref{LP2} and inequality \eqref{small2} gives the following estimates
\begin{align*}
\|P_{j_1}\partial^{\leq1}\nabla(\partial^{\leq1} u\partial\nabla u)\|_{L^2}=\sum_{j_3\backsim j_4\backsim j_1}\|P_{j_1}\partial^{\leq1}\nabla(\partial^{\leq1} u_{j_3}\partial\nabla u_{j_4})\|_{L^2}\leq \frac{C\varepsilon^2}{\delta^2}(1+t)^{-1}(2^{j_1})^{\frac12},
\end{align*}
which combining identity \eqref{space_div} and $F(P_j(xf)-xP_jf)=-i\partial_{\xi}\chi_j(\xi)\widehat{f}$  yields
\begin{align*}
&VII_{1,3,1}
\\ \leq &\frac{C}{1+t}\sum_{j_1\backsim j_2\backsim j}\Big(\|P_j\big(x\partial^{\leq2} (Zu)_{j_2}P_{j_1}\partial^{\leq1}\nabla(\partial^{\leq1} u\partial\nabla u)\big)\partial_tv\|_{L^1}
\\&\qquad\qquad\qquad+\|P_j\big(Z\partial^{\leq1} (Zu)_{j_2}P_{j_1}\partial^{\leq1}\nabla(\partial^{\leq1} u\partial\nabla u)\big)\partial_tv\|_{L^1}\Big)
\\ \leq &\frac{C}{1+t}\sum_{j_1\backsim j_2\backsim j}\|P_{j_1}\partial^{\leq1}\nabla(\partial^{\leq1} u\partial\nabla u)\|_{L^2}\left(\|x\partial_tv\|_{L^{\infty}}\|\partial^{\leq2} (Zu)_{j_2}\|_{L^2}+\|Z^{\leq1}\partial^{\leq1} (Zu)_{j_2}\|_{L^2}\|\partial_tv\|_{L^{\infty}}\right)
\\ \leq &\frac{C\varepsilon^4}{\delta^4}(1+t)^{-2+2\delta^2}(2^j)^{\frac12}.
\end{align*}

For terms $VII_{1,1}$, similarly to terms $VI_{1}$, we have
\begin{align*}
VII_{1,1} \leq VII_{1,1,1}+VII_{1,1,2}+VII_{1,1,3}+VII_{1,1,4}+VII_{1,1,5},
\end{align*}
where
\begin{align*}
&VII_{1,1,1}=\mathcal{O}(1)\sum_{j_1,j_2=0}^{+\infty}\frac d{dt}\int \partial\left(P_j\left(\partial_t(w_1w_2)\right)\right)e^{\arctan(r-t)}\partial_tvdx
\\&\qquad\qquad+\mathcal{O}(1)\sum_{j_1,j_2=0}^{+\infty}\frac{d}{dt}\int \partial\left(P_j\left(w_1w_2\right)\right)e^{\arctan(r-t)}\partial_t^2vdx,
\\&VII_{1,1,2}=\mathcal{O}(1)\sum_{j_1,j_2=0}^{+\infty}\int \frac1{1+(t-r)^2}\partial_{t,r}\left(P_j(\partial(w_1w_2))\right)e^{\arctan(r-t)}\partial_tvdx,
\\&VII_{1,1,3}=\mathcal{O}(1)\sum_{j_1,j_2=0}^{+\infty}\int \frac1{1+(t-r)^2}\partial\left(P_j\left(w_1w_2\right)\right)e^{\arctan(r-t)}\partial_t^2vdx,
\\&VII_{1,1,4}=\mathcal{O}(1)\sum_{j_1,j_2=0}^{+\infty}\int \frac1{1+(t-r)^2}\partial\left(P_j\left(w_1w_2\right)\right)e^{\arctan(r-t)}\partial_t\partial_rvdx,
\\&VII_{1,1,5}=C\sum_{j_1,j_2=0}^{+\infty}\int \left|P_j\big(\partial(w_1w_2)\big)\right||\Box\partial_tv+\partial_tv|dx.
\end{align*}
For terms $VII_{1,1,2}$, by H\"older inequality and Young's inequality,
\begin{align*}
VII_{1,1,2}\leq VII_{1,1,2,1}+VII_{1,1,2,2},
\end{align*}
where
\begin{align*}
&VII_{1,1,2,1}=C\sum_{j_1,j_2=0}^{+\infty}\left\|\partial^2\left(P_j(\partial u_{j_1}\partial (Zu)_{j_2})\right)\right\|^2_{L^2}+\frac1{32}\int \frac1{1+(t-r)^2}e^{\arctan(r-t)}|(\partial_t+\partial_r)v|^2dx,
\\&VII_{1,1,2,2}=\mathcal{O}(1)\sum_{j_1,j_2=0}^{+\infty}\int \frac1{1+(t-r)^2}\partial_{t,r}\left(P_j(\partial(\partial u_{j_1}\partial (Zu)_{j_2}))\right)e^{\arctan(r-t)}\partial_rvdx
\end{align*}
For $VII_{1,1,2,1}$, we divide them into three cases: $j,j_1$, or $j_2$ is the minimum among $\{j,j_1,j_2\}$. Thus
\begin{align*}
VII_{1,1,2,1} \leq \frac{C\varepsilon^4}{\delta^6}(1+t)^{-2+2\delta^2}(2^j)^2+\frac1{32}\int \frac1{1+(t-r)^2}e^{\arctan(r-t)}|(\partial_t+\partial_r)v|^2dx.
\end{align*}
For $VII_{1,1,2,2}$, by integration by parts with respect to $\partial_r$ from $\partial_rv$ and Hardy's inequality, we get 
\begin{align*}
VII_{1,1,2,2}\leq \frac{C\varepsilon^4}{\delta^6}(1+t)^{-2+2\delta^2}(2^{j})^{4}+\frac1{32}\int \frac1{1+(t-r)^2}e^{\arctan(r-t)}|v|^2dx.
\end{align*}

Similarly to terms $VI_{1,3},VI_{1,4}$, we get
\begin{align*}
VII_{1,1,3}+VII_{1,1,4}\leq  &\frac{C\varepsilon^4}{\delta^6}(1+t)^{-2+2\delta^2}(2^{j})^{4}+\frac1{32}\int \frac1{1+(t-r)^2}e^{\arctan(r-t)}|(\partial_t+\partial_r)v|^2dx
\\&+\frac1{32}\int \frac1{1+(t-r)^2}e^{\arctan(r-t)}|v|^2dx+VII_{1,1,3,1}+VII_{1,1,4,2},
\end{align*}
where
\begin{align*}
&VII_{1,1,3,1}=\mathcal{O}(1)\sum_{j_1,j_2=0}^{+\infty}\frac{d}{dt}\int \frac1{1+(t-r)^2}\partial\left(P_j\left(w_1w_2\right)\right)e^{\arctan(r-t)}\partial_tvdx,
\\&VII_{1,1,4,2}=\mathcal{O}(1)\sum_{j_1,j_2=0}^{+\infty}\frac{d}{dt}\int \frac1{1+(t-r)^2}\partial\left(P_j\left(w_1w_2\right)\right)e^{\arctan(r-t)}(\partial_t+\partial_r)vdx.
\end{align*}

Now we come to $VII_{1,1,5}$. By a similar discussion of terms $VII_{1,1,2,1}$, we have
\begin{align*}
VII_{1,1,5}\leq C\sum_{j_1,j_2=0}^{+\infty}\left\|P_j\big(\partial(\partial u_{j_1}\partial (Zu)_{j_2})\big)\right\|_{L^2}\|P_j\partial_tZ (\partial^{\leq 1} u\partial^{\leq 2} u)\|_{L^2}
\leq \frac{C\varepsilon^4}{\delta^6}(1+t)^{-2+2\delta^2}2^j.
\end{align*}

Finally, it suffices to deal with terms $VII_{1,4}$.
\begin{align*}
&VII_{1,4}=VII_{1,4,1}+VII_{1,4,2}+VII_{1,4,3}+VII_{1,4,4},
\end{align*}
where
\begin{align*}
&VII_{1,4,1}=\mathcal{O}(1)\sum_{j_1\backsim j_2\backsim j}\int\partial\Big(P_j\Big(\partial_tw_1\partial_tw_2-\nabla w_1\cdot\nabla w_2+w_1w_2\Big)\Big)e^{\arctan(r-t)}(\partial_t+\partial_r)vdx,
\\&VII_{1,4,2}=\mathcal{O}(1)\sum_{j_1\backsim j_2\backsim j}\int\partial\Big(P_j\Big(\partial_tw_1\partial_tw_2-\nabla w_1\cdot\nabla w_2+w_1w_2\Big)\Big)e^{\arctan(r-t)}\partial_rvdx,
\\&VII_{1,4,3}=\mathcal{O}(1)\sum_{j_1\backsim j_2\backsim j}\int\nabla\Big(P_j\Big(\partial_tw_3\partial_tw_4-\nabla w_3\cdot\nabla w_4+w_3w_4\Big)\Big)e^{\arctan(r-t)}(\partial_t+\partial_r)vdx,
\\&VII_{1,4,4}=\mathcal{O}(1)\sum_{j_1\backsim j_2\backsim j}\int\nabla\Big(P_j\Big(\partial_tw_3\partial_tw_4-\nabla w_3\cdot\nabla w_4+w_3w_4\Big)\Big)e^{\arctan(r-t)}\partial_rvdx.
\end{align*}
Let
\[
VII_{1,4,1,j_1,j_2}=\Big|\int\partial\Big(P_j\Big(\partial_tw_1\partial_tw_2-\nabla w_1\cdot\nabla w_2+w_1w_2\Big)\Big)e^{\arctan(r-t)}(\partial_t+\partial_r)vdx\Big|.
\]
Similarly to terms $VI_{4,1}$, we get 
\begin{align}\label{combin2}
VII_{1,4,1,j_1,j_2}\leq \left(VII_{1,4,1,1,j_1,j_1}+VI_{1,4,1,2,j_1,j_2}\right)\left(\int \frac{e^{\arctan(r-t)}}{1+(t-r)^2}|(\partial_t+\partial_r)v|^2dx\right)^{\frac12},
\end{align}
where 
\begin{align*}
&VII_{1,4,1,1,j_1,j_2}\leq C\Big(\|(\partial^{\leq1}Z^{\leq1}\nabla u_{j_1})\partial^{\leq1}\nabla (Zu)_{j_2}\|_{L^2}+\|(\partial^{\leq1}Z^{\leq1}\nabla (Zu)_{j_2})\partial^{\leq1}\nabla u_{j_1}\|_{L^2}\Big),
\\&VII_{1,4,1,2,j_1,j_2}\leq C\Big(\|(S\nabla (Zu)_{j_2})P_{j_1}\nabla(\Box u+u)\|_{L^2}+\|(S\nabla u_{j_1})P_{j_2}\nabla Z(\Box u+u)\|_{L^2}\Big).
\end{align*}
First we come to $VII_{1,4,1,1,j_1,j_2}$. By using \eqref{space_div}, we have
\begin{align}\label{add0823}
VII_{1,4,1,1,j_1,j_2}\leq \frac{C\varepsilon^2}{\delta^2}(1+t)^{-1+\delta^2}\begin{cases}
& (2^{j_1})^{\frac12-12\delta}, \ when \ j=min\{j,j_1,j_2\},
\\& (2^{j})^{\frac12-12\delta}, \ when \ j_1=min\{j,j_1,j_2\},
\\& (2^{j})^{\frac12-12\delta}, \ when \ j_2=min\{j,j_1,j_2\}.
\end{cases}
\end{align}
By calculation, we get
\begin{align}\label{add8}
\|(S\nabla (Zu)_{j_2})P_{j_1}\nabla(\Box u+u)\|_{L^2}\leq &C2^{j_1}\|\partial(\nabla (Zu)_{j_2})\|_{L^2}\|(1+t+r)P_{j_1}(\partial^{\leq1}u\partial^{\leq2}u)\|_{L^{\infty}}
\nonumber\\ \leq &\frac{C\varepsilon^3}{\delta^3}(1+t)^{-1+\delta^2}(2^{j_2})^{-\frac12-12\delta}(2^{j_1})^{1-12\delta}.
\end{align}
Using \eqref{space_div} shows that
\begin{align*}
\|P_{j_2}Z(\Box u +u)\|_{L^2}&\leq C\left(\|P_{j_2}(\partial^{\leq1}Zu\partial\nabla u)\|_{L^2}+\|P_{j_2}(\partial^{\leq1}u\partial\nabla Z u)\|_{L^2}\right)+\frac{C\varepsilon^2}{\delta}(1+t)^{-1}(2^{j_2})^{-1}
\\&\leq\frac{C\varepsilon^2}{\delta^2}(1+t)^{-1+\delta^2}(2^{j_2})^{-\frac12-12\delta},
\end{align*}
which combining $\|S\nabla u_{j_1}\|_{L^{\infty}}\leq C\left(t\|\partial\nabla u_{j_1}\|_{L^{\infty}}+\|r\partial\nabla u_{j_1}\|_{L^{\infty}}\right)\leq \frac{C\varepsilon}{\delta}(2^{j_1})^{-\delta}$ implies that 
\begin{align}\label{add9}
\|(S\nabla u_{j_1})P_{j_2}\nabla Z(\Box u+u)\|_{L^2}\leq \frac{C\varepsilon^3}{\delta^3}(1+t)^{-1+\delta^2}(2^{j_2})^{\frac12-12\delta}(2^{j_1})^{-\delta}.
\end{align}
Combining \eqref{add8} and \eqref{add9} yields
\begin{align}\label{add19}
VII_{1,4,1,2,j_1,j_2}\leq \frac{C\varepsilon^3}{\delta^3}(1+t)^{-1+\delta^2}\left((2^{j_2})^{-\frac12-12\delta}(2^{j_1})^{1-\delta}+(2^{j_2})^{\frac12-12\delta}(2^{j_1})^{-\delta}\right).
\end{align}
Thus using $(1+j)(2^j)^{-12\delta}\leq \frac{C}{\delta}$ gives 
\begin{align*}
VII_{1,4,1}\leq &\frac1{32}\int \frac1{1+(t-r)^2}e^{\arctan(r-t)}|(\partial_t+\partial_r)v|^2dx+\frac{C\varepsilon^4}{\delta^6}(1+t)^{-2+2\delta^2}2^j
\\&+\frac{C\varepsilon^6}{\delta^6}(1+t)^{-2+2\delta^2}(2^j)^{2-2\delta}
+C\sum_{j_1=j}^{+\infty}\sum_{j_2=j_1-2}^{j_1+2}VII_{1,4,1,j_1,j_2}.
\end{align*}
The last term becomes problematic when estimates \eqref{add0823} and \eqref{add19} are employed. Therefore, we require the following estimate, which exhibits lower time decay yet higher frequency decay. When $j_1\geq j$ and $j_1-2\leq j_2\leq j_1+2$,
\begin{align}\label{combin1}
&VII_{1,4,1,j_1,j_2}
\nonumber\\ \leq &C\int\Big|P_j\Big(\partial^{\leq2}\nabla u_{j_1}\partial^{\leq1}\nabla(Zu)_{j_2}\Big)(\partial v)\Big|dx+C\int\Big|P_j\Big(\partial^{\leq1}\nabla u_{j_1}\partial^{\leq2}\nabla(Zu)_{j_2}\Big)(\partial v)\Big|dx
\nonumber\\ \leq &\frac{C}{t+1}\int\Big|P_j\Big(x\partial^{\leq3} u_{j_1}\partial^{\leq1}\nabla(Zu)_{j_2}\Big)(\partial v)\Big|dx+\frac{C}{t+1}\int\Big|P_j\Big(x\partial^{\leq2} u_{j_1}\partial^{\leq2}\nabla(Zu)_{j_2}\Big)(\partial v)\Big|dx
\nonumber\\&+\frac{C}{t+1}\int\Big|P_j\Big(Z\partial^{\leq2} u_{j_1}\partial^{\leq1}\nabla(Zu)_{j_2}\Big)(\partial v)\Big|dx+\frac{C}{t+1}\int\Big|P_j\Big(Z\partial^{\leq1} u_{j_1}\partial^{\leq2}\nabla(Zu)_{j_2}\Big)(\partial v)\Big|dx
\nonumber\\ \leq &\frac{C}{t+1}\left(\|\partial^{\leq3} u_{j_1}\partial^{\leq1}\nabla(Zu)_{j_2}\|_{L^1}+\|\partial^{\leq2} u_{j_1}\partial^{\leq2}\nabla(Zu)_{j_2}\|_{L^1}\right)\|x\partial v\|_{L^{\infty}}
\nonumber\\&+\frac{C}{t+1}\left(\|Z\partial^{\leq2} u_{j_1}\partial^{\leq1}\nabla(Zu)_{j_2}\|_{L^1}+\|Z\partial^{\leq1} u_{j_1}\partial^{\leq2}\nabla(Zu)_{j_2}\|_{L^1}\right)\|\partial v\|_{\infty}
\nonumber\\ \leq &\frac{C\varepsilon^3}{\delta}(1+t)^{-1+\delta^2} (2^{j_1})^{-1-12\delta}(2^j)^{-12\delta}.
\end{align}
Since $0<\varepsilon\leq\delta^6$, the combination of inequalities \eqref{combin1}, \eqref{add0823}, \eqref{add19} and \eqref{combin2} shows that
\begin{align}\label{cc1}
&\sum_{j_1=j}^{+\infty}\sum_{j_2=j_1-2}^{j_1+2}VII_{1,4,1,j_1,j_2}
\nonumber\\ \leq &C\sum_{j_1=j}^{+\infty}\sum_{j_2=j_1-2}^{j_1+2}\left(\frac{C\varepsilon^2}{\delta^3}(1+t)^{-1+\delta^2}(2^{j_1})^{\frac12-12\delta}\right)^{\frac12}\left(\int \frac1{1+(t-r)^2}e^{\arctan(r-t)}|(\partial_t+\partial_r)v|^2dx\right)^{\frac14}
\nonumber\\&\qquad\qquad\qquad\times\Big(\frac{\varepsilon^3}{\delta}(1+t)^{-1+\delta^2} (2^{j_1})^{-1-12\delta}(2^j)^{-12\delta}\Big)^{\frac12}
\nonumber\\ \leq &\frac{C}{\delta^{2}}\varepsilon^{\frac52}(1+t)^{-1+\delta^2}(2^j)^{-\frac14}\left(\int \frac1{1+(t-r)^2}e^{\arctan(r-t)}|(\partial_t+\partial_r)v|^2dx\right)^{\frac14}.
\end{align}
We employ Young's inequality: for any positive constants $a,b,\epsilon_0$ and $p,q>0$ with $\frac1p+\frac1q=1$, 
\[
ab\leq \frac{(a/\epsilon_0)^p}p+\frac{(b\epsilon_0)^q}q.
\]
Choosing $p=\frac43,q=4$ and taking $\epsilon_0$ sufficiently small, we obtain
\begin{align*}
VII_{1,4,1}\leq &\frac1{16}\int \frac1{1+(t-r)^2}e^{\arctan(r-t)}|(\partial_t+\partial_r)v|^2dx
+\frac{C\varepsilon^4}{\delta^6}(1+t)^{-2+2\delta^2}(2^j)^2
\\&+\frac{C}{\delta^{3}}\varepsilon^{\frac{10}3}(1+t)^{-\frac43+\frac43\delta^2}(2^j)^{-\frac13}.
\end{align*}

Using equality \eqref{lemm2_add}, for any scalar function $f$, we get
\begin{align*}
(t-r)\left((1+t+r)\partial_kf-L_kf\right)=(t-r)\partial_kf-x^kSf+x^j\widetilde{\Omega}_{kj}f+rL_kf,
\end{align*}
which implies
\begin{align}\label{add10}
\left|\partial_kf-\frac{L_kf}{1+t+r}\right|&\leq \frac{C}{(1+t+r)\langle t-r\rangle}\left(|Lf|+(1+t+r+|t-r|)|\partial_kf|+|xSf|+|xZf|\right).
\end{align}
Using inequality \eqref{add10}, the estimate \eqref{kg_pro1} and $\|[S,P_j]f\|_{L^2}\leq C\|f\|_{L^2}$ gives
\begin{align*}
&VII_{1,4,3}\leq VII_{1,4,3,3}+\sum_{j_1\backsim j_2\backsim j}\left|VII_{1,4,3,j_1,j_2}\right|,
\end{align*}
where
\begin{align}
&VII_{1,4,3,3}=\mathcal{O}(1)\sum_{j_1,j_2=0}^{+\infty}\int\frac1{1+t+r}L\Big(P_j\Big(\partial_tw_3\partial_tw_4-\nabla w_3\cdot\nabla w_4+w_3w_4\Big)\Big)e^{\arctan(r-t)}(\partial_t+\partial_r)vdx,
\nonumber\\&\label{combin3}
\left|VII_{1,4,3,j_1,j_2}\right|\leq \left(VII_{1,4,3,1,j_1,j_2}\right)\left(\int \frac1{\langle t-r\rangle^2}e^{\arctan(r-t)}|(\partial_t+\partial_r)v|^2dx\right)^{\frac12}+VII_{1,4,3,2,j_1,j_2},
\end{align}
with
\begin{align*}
&VII_{1,4,3,1,j_1,j_2}
\\ = &C\Big(\|(\partial^{\leq3} u_{j_1})\partial^{\leq2}(Zu)_{j_2}\|_{L^2}+\|(\partial^{\leq3}(Zu)_{j_2})\partial^{\leq2}u_{j_1}\|_{L^2}\Big)
\\&+\frac{C}{1+t}\Big(\|\langle x\rangle(\partial^{\leq2}Z^{\leq1}u_{j_1})\partial^{\leq2}(Zu)_{j_2}\|_{L^2}+\|\langle x\rangle(\partial^{\leq2}Z^{\leq1} (Zu)_{j_2})\partial^{\leq2}u_{j_1}\|_{L^2}\Big)
\\&+\frac{C}{1+t}\Big(\left\|xP_j\left(S\partial^{\leq1} (Zu)_{j_2}P_{j_1}\partial^{\leq1}(\Box u+u)\right)\right\|_{L^2}+\left\|xP_j\left(S\partial^{\leq1} u_{j_1}P_{j_2}\partial^{\leq1} Z(\Box u+u)\right)\right\|_{L^2}\Big),
\\ &VII_{1,4,3,2,j_1,j_2}=\frac{C}{1+t}\int\left(|\nabla\partial^{\leq2} u_{j_1}\partial^{\leq2} (Zu)_{j_2}|+|\partial^{\leq2} u_{j_1}\nabla\partial^{\leq2} (Zu)_{j_2}|\right)|\partial v|dx.
\end{align*}
By applying integration by parts to $L$, using
\[
\left|L\left(\frac1{1+t+r}\right)\right|\leq \frac{C}{1+t},\ |L(t-r)|\leq C(t+r),\ |L\partial_rf-\partial_rLf|\leq C|\partial Z^{\leq1}f|+C\left|\frac{Lf}r\right|,
\]
we have verified that every term from $VII_{1,4,3,3}$ is well-behaved, apart from the following one
\[
VII_{1,4,3,3,1}=\mathcal{O}(1)\sum_{j_1,j_2=0}^{+\infty}\frac{d}{dt}\int\frac{x}{1+t+r}\Big(P_j\Big(\partial_tw_3\partial_tw_4-\nabla w_3\cdot\nabla w_4+w_3w_4\Big)\Big)e^{\arctan(r-t)}(\partial_t+\partial_r)vdx.
\] 
Using similar discussion of $VII_{1,4,1,1,j_1,j_2}$ and $VII_{1,4,1,2,j_1,j_2}$ yields that terms $VII_{1,4,3,1,j_1,j_2}$ own similar bound of the sum of $VII_{1,4,1,1,j_1,j_2}$ and $VII_{1,4,1,2,j_1,j_2}$. Using identity \eqref{space_div} and applying $L^{\infty}$-norm to $\partial v$ shows that terms $VII_{1,4,3,2,j_1,j_2}$ are fine. Then terms $VII_{1,4,3}$ are acceptable.

The discussion on terms $VI_{4,2}$ applies to  terms $VII_{1,4,2}$ and $VII_{1,4,4}$. All these terms are fine.

Integrating by parts in time, we transfer $\partial_t$ from $\partial_tv$ to $\partial^{\leq1}u_{j_1}\partial^{\leq1}(Z^{\leq1}u)_{j_2}$. Using the identity 
\[
\partial_t=\frac1{1+t}(S-r\partial_r+\partial_t),
\]
to treat the similar terms of $VII_{1,4}$, we can treat terms $VII_3$ analogously to $VII_1$ and omit the details. 

Thus the combination of the discussion above shows that 
\begin{align}\label{low3_per}
\frac{d}{dt}N_{2,j}(u(t))\leq & C_{low3}\left( \frac{\varepsilon^4}{\delta^6}(1+t)^{-2+2\delta^2}(2^j)^4+\frac{\varepsilon^{\frac{10}3}}{\delta^{3}}(1+t)^{-\frac43+\frac43\delta^2}(2^j)^{\frac53}\right)
\nonumber\\&+VII_{1,1,1}+VII_{1,1,3,1}+VII_{1,1,4,2}+VII_{1,4,3,3,1}+other\ similar\ terms.
\end{align}

Then there exists $C_0$ depending only on $\widetilde{C}$ such that
\begin{align}\label{C_0_2}
\sum_{j=0}^{+\infty}(2^{j})^{3+54\delta}N_{2,j}(u(0)) \leq \frac{C_0}{1000}\varepsilon^2.
\end{align}
By using the high-order energy estimates, we get
\[
N_{2,j}(u(t))\leq 5\|\partial^{\leq1} (Zu)_j\|^2_{L^2}\leq 5C_2\varepsilon^2(2^{j})^{-3-54\delta}(1+t)^{\delta^2}.
\]
For any $t\geq 0$, analogously to \eqref{low2}, we get
\begin{align}\label{add11}
\frac d{dt}\left(\left(N_{2,j}(u(t))\right)^{\frac1{\delta}}\right)\leq& 
C_{low3}\frac{1}{\delta^7}(5C_2)^{\frac{1}{\delta}-1}(\varepsilon)^{\frac2{\delta}+2}(2^{j})^{-\frac3{\delta}-47+54\delta}(1+t)^{-2+\delta+\delta^2}
\nonumber\\&+C_{low3}\frac{1}{\delta^4}(5C_2)^{\frac{1}{\delta}-1}(\varepsilon)^{\frac2{\delta}+\frac43}(2^{j})^{-\frac3{\delta}-49}(1+t)^{-\frac43+\delta+\delta^2}
\nonumber\\&+\frac{1}{\delta}\left(N_{2,j}(u(t))\right)^{\frac{1}{\delta}-1}(VII_{1,1,1}+VII_{1,1,3,1}+VII_{1,1,4,2}+VII_{1,4,3,3,1})
\nonumber\\&+other\ analogous\ term.
\end{align}
We come to terms involving $VII_{1,1,1}$ in the above estimate. Let 
\begin{align*}
R_{2,1}=&\mathcal{O}(1)\sum_{j_3,j_4=0}^{+\infty}\int \partial\left(P_j\left(\partial_t(\partial^{\leq1} u_{j_3}\partial^{\leq1} (Zu)_{j_4})\right)\right)e^{\arctan(r-t)}\partial_t(Zu)_jdx
\\&+\mathcal{O}(1)\sum_{j_3,j_4=0}^{+\infty}\int \partial\left(P_j\left(\partial^{\leq1} u_{j_3}\partial^{\leq1} (Zu)_{j_4}\right)\right)e^{\arctan(r-t)}\partial_t^2(Zu)_jdx.
\end{align*}
Thus considering three cases depending on which of $j,j_3,j_4$ is the minimum among $\{j,j_3,j_4\}$ and placing $L^{\infty}$ norms on $j_3$ part shows that
\begin{align*}
\left|R_{2,1}\right|\leq \frac{C\varepsilon^3}{\delta^3}(1+t)^{-1}(2^{j})^{-\frac12},
\end{align*}
which implies
\begin{align}\label{add12}
\left|\frac{1}{\delta}\left(N_{2,j}(u(t))\right)^{\frac{1}{\delta}-1}R_{2,1}\right|\leq \frac{C}{\delta^4}(5C_2)^{\frac{1}{\delta}-1}(\varepsilon)^{\frac2{\delta}+1}(2^{j})^{-\frac3{\delta}-51.5+54\delta}.
\end{align}
Using \eqref{low4} gives
\begin{align}\label{add13}
\left|\frac{d}{dt}\left(N_{2,j}(u(t))\right)\right|&\leq C\int \left(\left|P_j\left(\partial^{\leq2}Z^{\leq1} u\partial^{\leq1} u\right)\right|+\left|P_j\left(\partial^{\leq2} u\partial^{\leq1} Z^{\leq1}u\right)\right|\right)|\partial_t(Zu)_j|dx
\nonumber\\&\leq \frac{C\varepsilon^3}{\delta^3}(1+t)^{-1}(2^j)^{-1.5}.
\end{align} 
Applying \eqref{add12} and \eqref{add13}, for $\delta$ in $(0,0.001]$, we obtain
\begin{align}\label{add14}
&\quad\left|\int_0^T\frac{1}{\delta}\left(N_{2,j}(u(t))\right)^{\frac{1}{\delta}-1}(VII_{1,1,1})dt\right|
\nonumber\\&\leq \left|\int_0^T\frac{d}{dt}\left(\frac{1}{\delta}\left(N_{2,j}(u(t))\right)^{\frac{1}{\delta}-1}R_{2,1}\right)dt
-\int_0^T\left(\frac{1-\delta}{\delta^2}\left(N_{2,j}(u(t))\right)^{\frac{1}{\delta}-2}\right)\frac{d}{dt}\left(N_{2,j}(u(t))\right)R_{2,1}dt\right|
\nonumber\\&\leq \frac{C}{\delta^4}(5C_2)^{\frac{1}{\delta}-1}(\varepsilon)^{\frac2{\delta}+1}(2^{j})^{-\frac3{\delta}-51.5+54\delta}
+\frac{C}{\delta^8}(5C_2)^{\frac{1}{\delta}-2}(\varepsilon)^{\frac2{\delta}+2}(2^{j})^{-\frac3{\delta}-50+128\delta}.
\end{align}
Terms involving $VII_{1,1,3,1},VII_{1,1,4,2},VII_{1,4,3,3,1}$ can be treated similarly and they own the same bound of terms involving $VII_{1,1,1}$. 

Similarly to the end of Section \ref{sec_low1}, if $\delta$ in $(0,0.001]$ is a fixed constant and $\varepsilon$ in $(0,\delta^6]$ is small enough, for any  $t>0$,  integrating both sides of \eqref{add11} over the interval $[0,t]$ and using \eqref{C_0_2} gives
\[
M_{low3}(u(t))\leq 5\sum_{j=0}^{+\infty}(2^j)^{3+46\delta}N_{2,j}(u(t))\leq \frac{C_0}{32}\varepsilon^2.
\]

\subsubsection{\textbf{Estimates for $M_{low2}(u(t))$}}\label{sec_low3}
Let 
\[
N_{3,j,k}(u(t))=\int e^{\arctan(r-t)}\left(|\partial^{\leq1} u_{j,k}|^2\right)dx.
\]
Similarly to inequality \eqref{low1}, applying Lemma \ref{product_KG} gives
\begin{align}\label{low3}
&\frac d{dt}N_{3,j,k}(u(t))+\int e^{\arctan(r-t)}\frac1{1+(t-r)^2}\left((u_{j,k})^2+\left|(\partial_t+\partial_r)u_{j,k}\right|^2\right)dx
\nonumber\\ \leq &VIII_1+VIII_2+VIII_3+VIII_4+good\ terms,
\end{align}
where
\begin{align*}
&VIII_{1}=\mathcal{O}(1)\int \Big(R_kP_j\partial\Big(\Box(\partial^{\leq1} u\partial  u)+\partial^{\leq1} u\partial  u\Big)\Big)e^{\arctan(r-t)}\partial_tu_{j,k}dx,
\\&VIII_{2}=\mathcal{O}(1)\int\Big(R_kP_j\partial\big(\partial^{\leq1} u(\Box \partial  u+\partial  u)\big)\Big)e^{\arctan(r-t)}\partial_tu_{j,k}dx,
\\&VIII_{3}=\mathcal{O}(1)\int\Big(R_kP_j\partial\big(\partial  u(\Box \partial^{\leq1} u+\partial^{\leq1} u)\big)\Big)e^{\arctan(r-t)}\partial_tu_{j,k}dx,
\\&VIII_{4}=\mathcal{O}(1)\int\Big(R_kP_j\Big(\partial\big(\partial_t\partial^{\leq1} u\partial_t\partial  u-\nabla \partial^{\leq1} u\cdot\nabla \partial  u+\partial^{\leq1} u\partial  u\big)\Big)\Big)e^{\arctan(r-t)}\partial_tu_{j,k}dx.
\end{align*}
Here, by "good terms" we mean terms involving semi-linear terms $H(u)$, which enjoy better properties. Hence we omit the details. Using the estimate $\|R_kP_jf\|_{L^2}\leq C\|P_jf\|_{L^2}$, we reduce the discussion of terms $VIII_i$ to that of terms $VI_i$, $i=1,2,3,4$, respectively. Here for terms $VIII_1$, we will employ $[\Box,R_k]=[\Box,P_j]=0$. Thus we get
\begin{align}\label{low2_per}
\frac{d}{dt}N_{3,j,k}(u(t))\leq C_{low2}\frac{\varepsilon^4}{\delta^6}(1+t)^{-2}(2^j)^2+VIII_{1,1}+terms\ similar\ to\ VIII_{1,1},
\end{align}
where
\begin{align*}
&VIII_{1,1}=\mathcal{O}(1)\frac d{dt}\int \partial_t\left(R_kP_j\left(\partial(\partial^{\leq1} u\partial  u)\right)\right)e^{\arctan(r-t)}\partial_tu_{j,k}dx.
\end{align*}

\textbf{Step1: let's estimate $\sum_{j=0}^{+\infty}\sum_{k=j+20}^{+\infty}\Big((2^j)^{3+42\delta}(2^{k})^{2+2\delta}\Big)\|\partial^{\leq1}u_{j,k}\|_{L^{2}}^2$}. Since
\[
M_{high2}(u(t))+M_{high4}(u(t))\leq C_2\varepsilon^2(1+t)^{\delta^2},
\]
then 
\[
\sum_{j=0}^{+\infty}\sum_{k=j+20}^{+\infty}\left((2^j)^{3+50\delta}(2^k)^{2+2\delta}+(2^j)^{1+40\delta}(2^k)^{4+2\delta}\right)\|\partial^{\leq1}u_{j,k}\|_{L^{2}}^2\leq C\varepsilon^2(1+t)^{\delta^2}.
\]
A straight calculation shows that
\begin{align*}
(2^k)^{2.1+2\delta}(2^j)^{2.9+49.5\delta}=&\left((2^j)^{3+50\delta}(2^k)^{2+2\delta}\right)^{0.95}\left((2^j)^{1+40\delta}(2^k)^{4+2\delta}\right)^{0.05}
\\ \leq &\frac{19}{20}\left((2^j)^{3+50\delta}(2^k)^{2+2\delta}\right)+\frac1{20}\left((2^j)^{1+40\delta}(2^k)^{4+2\delta}\right),
\end{align*}
which implies
\begin{align}\label{cc2}
\sum_{j=0}^{+\infty}\sum_{k=j+20}^{+\infty}(2^k)^{2.1+2\delta}(2^j)^{2.9+49.5\delta}\|\partial^{\leq 1}u_{j,k}\|_{L^{2}(\mathbb R^3)}^2\leq C\varepsilon^2(1+t)^{\delta^2}.
\end{align}

Then there exists $C_0$ depending only on $\widetilde{C}$ such that
\begin{align}\label{C_0_3}
\sum_{j=0}^{+\infty}\sum_{k=j+20}^{+\infty}(2^k)^{2.1+2\delta}(2^j)^{2.9+49.5\delta}N_{3,j,k}(u(0)) \leq \frac{C_0}{1000}\varepsilon^2.
\end{align}
Using \eqref{cc2} yields
\[
N_{3,j,k}(u(t))\leq 5\|\partial^{\leq1} u_{j,k}\|^2_{L^2}\leq 5C\varepsilon^2(2^k)^{-2.1-2\delta}(2^j)^{-2.9-49.5\delta}(1+t)^{\delta^2}.
\]
Similarly to \eqref{low2}, for $t\geq 0$, we have
\begin{align}\label{cc3}
\frac d{dt}\left(\left(N_{3,j,k}(u(t))\right)^{\frac1{\delta}}\right)\leq& 
\frac{C_{low2}}{\delta^7}(5C_2)^{\frac1{\delta}-1}\varepsilon^{\frac2{\delta}+2}(1+t)^{-2+\delta-\delta^2}(2^k)^{-\frac{2.1}{\delta}+0.1+2\delta}(2^j)^{-\frac{2.9}{\delta}-44.6+50\delta}
\nonumber\\&+\frac{1}{\delta}\left(N_{3,j,k}(u(t))\right)^{\frac{1}{\delta}-1}(VIII_{1,1}+terms\ similar\ to\ VIII_{1,1}).
\end{align}

Consequently, if $\delta$ in $(0,0.001]$ is a fixed constant and $\varepsilon$ in $(0,\delta^6]$ is small enough, for any positive time $t$, integrating both sides of \eqref{cc3} over the interval $[0,t]$ and using \eqref{C_0_3} gives
\[
\sum_{j=0}^{+\infty}\sum_{k=j+20}^{+\infty}\Big((2^j)^{3+42\delta}(2^{k})^{2+2\delta}\|\partial^{\leq1}u_{j,k}\|_{L^2}^2\Big)\leq 5\sum_{j=0}^{+\infty}\sum_{k=j+20}^{+\infty}\Big((2^j)^{3+42\delta}(2^{k})^{2+2\delta}N_{3,j,k}(u(t))\Big)\leq \frac{C_0}{32}\varepsilon^2.
\]

\textbf{Step2: let's estimate $\sum_{j=0}^{+\infty}\sum_{k=j+20}^{+\infty}\Big((2^j)^{5+42\delta}(2^{k})^{10\delta}\|\partial^{\leq1}u_{j,k}\|_{L^{2}}^2\Big)$}. Since
\[
M_{high1}(u(t))+M_{high2}(u(t))\leq C_2\varepsilon^2(1+t)^{\delta^2},
\]
then 
\[
\sum_{j=0}^{+\infty}\sum_{k=j+20}^{+\infty}\left((2^j)^{3+50\delta}(2^k)^{2+2\delta}+(2^j)^{5.01+64\delta}(2^k)^{-0.01}\right)\|\partial^{\leq1}u_{j,k}\|_{L^{2}}^2\leq 130C_2\varepsilon^2(1+t)^{\delta^2}.
\]
Similarly to \textbf{Step1}, we get 
\begin{align*}
(2^k)^{0.0905+0.1\delta}(2^j)^{4.9095+63.3\delta}\leq \frac{1}{20}\left((2^j)^{3+50\delta}(2^k)^{2+2\delta}\right)+\frac{19}{20}\left((2^j)^{5.01+64\delta}(2^k)^{-0.01}\right),
\end{align*}
which implies
\begin{align*}
\sum_{j=0}^{+\infty}\sum_{k=j+20}^{+\infty}(2^k)^{0.0905+0.1\delta}(2^j)^{4.9095+63.3\delta}\|\partial^{\leq 1}u_{j,k}\|_{L^{2}(\mathbb R^3)}^2\leq 130C_2\varepsilon^2(1+t)^{\delta^2}.
\end{align*}
Then 
\[
N_{3,j,k}(u(t))\leq 5\|\partial^{\leq1} u_{j,k}\|^2_{L^2}\leq 650C_2\varepsilon^2(2^k)^{-0.0905-0.1\delta}(2^j)^{-4.9095-63.3\delta}(1+t)^{\delta^2},
\]
and there exists $C_0$ depending only on $\widetilde{C}$ such that
\begin{align}\label{C_0_4}
\sum_{j=0}^{+\infty}\sum_{k=j+20}^{+\infty}(2^k)^{0.0905+0.1\delta}(2^j)^{4.9095+63.3\delta}N_{3,j,k}(u(0)) \leq \frac{C_0}{1000}\varepsilon^2.
\end{align}
Similarly to Section \ref{sec_low1}, for any $t>0$, if $\delta$ in $(0,0.001]$ is a fixed constant and $\varepsilon$ in $(0,\delta^6]$ is small enough, we get
\begin{align*}
\sum_{j=0}^{+\infty}\sum_{k=j+20}^{+\infty}\Big((2^j)^{5+42\delta}(2^{k})^{10\delta}\|\partial^{\leq1}u_{j,k}\|_{L^{2}}^2\Big)\leq \frac{C_0}{32}\varepsilon^2.
\end{align*}

Combining the discussion above shows that
\[
M_{low2}(u(t))\leq \frac{C_0}{16}\varepsilon^2.
\]

\subsubsection{\textbf{Estimates for $M_{low4}(u(t))$}}\label{sec_low4}

In Section \ref{sec_low4}, let  $v=(Zu)_{j,k}$ and
\[
N_{4,j,k}(u(t))=\int e^{\arctan(r-t)}\left(|\partial^{\leq1} (Zu)_{j,k}|^2\right)dx.
\]
Similarly to inequality \eqref{low4},
\begin{align}\label{low5}
\frac d{dt}N_{4,j,k}(u(t))+\int e^{\arctan(r-t)}\frac1{1+(t-r)^2}\left(v^2+\left|(\partial_t+\partial_r)v\right|^2\right)dx\leq IX_1+IX_2+IX_3,
\end{align}
where
\begin{align*}
&IX_1=\sum_{j_1,j_2=0}^{+\infty}\int R_kP_j\left( \mathcal{O}\left(\partial_t(\nabla u_{j_1}\nabla (Zu)_{j_2}\right)+\mathcal{O}\left(\nabla(\partial^{\leq1} u_{j_1}\partial^{\leq1} (Zu)_{j_2}\right)\right)e^{\arctan(r-t)}\partial_tvdx,
\\&IX_2=\sum_{j_1,j_2=0}^{+\infty}\int R_kP_j\mathcal{O}\left(\partial(\partial^{\leq1} u_{j_1}\partial u_{j_2})\right)e^{\arctan(r-t)}\partial_tvdx,
\\&IX_3=\sum_{j_1,j_2=0}^{+\infty}\int R_kP_j\mathcal{O}\left(\partial^{\leq1}u_{j_1}\partial^{\leq1}(Z^{\leq1}u)_{j_2}\right)e^{\arctan(r-t)}\partial_tvdx.
\end{align*}
Using the estimate $\|R_kP_jf\|_{L^2}\leq C\|P_jf\|_{L^2}$, we reduce the discussion of terms $IX_i$ to that of terms $VII_i$, $i=1,2,3$, respectively.
Thus we achieve 
\begin{align}\label{low4_per}
\frac{d}{dt}N_{4,j,k}(u(t))\leq &C\left(\frac{\varepsilon^4}{\delta^6}(1+t)^{-2+2\delta^2}(2^j)^4+\frac{\varepsilon^{\frac{10}3}}{\delta^2}(1+t)^{-\frac43+2\delta^2}(2^j)^{2}\right)
\nonumber\\&+IX_{1,1,1}+other\ terms\ similar\ to\ IX_{1,1,1},
\end{align}
where
\begin{align*}
IX_{1,1,1}=\mathcal{O}(1)\sum_{j_1,j_2=0}^{+\infty}\frac d{dt}\int \left(R_kP_j\left(\partial\partial_t(\nabla u_{j_1}\nabla(Zu)_{j_2})\right)\right)e^{\arctan(r-t)}\partial_tvdx.
\end{align*}

Since
\[
M_{high3}(u(t))+M_{high4}(u(t))\leq C_2\varepsilon^2(1+t)^{\delta^2},
\]
then 
\[
\sum_{j=0}^{+\infty}\sum_{k=j+20}^{+\infty}\left((2^j)^{1+40\delta}(2^k)^{2+2\delta}+(2^j)^{3.01+54\delta}(2^k)^{-0.01}\right)\|\partial^{\leq1}(Zu)_{j,k}\|_{L^{2}}^2\leq 130C_2\varepsilon^2(1+t)^{\delta^2}.
\]
A direct calculation yields
\begin{align*}
(2^k)^{0.0905+0.1\delta}(2^j)^{2.9095+ 53.3\delta}\leq \frac{19}{20}\left((2^j)^{3.01+54\delta}(2^{k})^{-0.01}\right)+\frac1{20}\left((2^j)^{1+40\delta}(2^k)^{2+2\delta}\right),
\end{align*}
which implies
\begin{align}\label{cc4}
\sum_{j=0}^{+\infty}\sum_{k=j+20}^{+\infty}(2^k)^{0.0905+0.1\delta}(2^j)^{2.9095+ 53.3\delta}\|\partial^{\leq 1}(Zu)_{j,k}\|_{L^{2}(\mathbb R^3)}^2\leq 130C_2\varepsilon^2(1+t)^{\delta^2}.
\end{align}
Then there exists $C_0$ depending only on $\widetilde{C}$ such that
\begin{align}\label{C_0_5}
\sum_{j=0}^{+\infty}\sum_{k=j+20}^{+\infty}(2^k)^{0.0905+0.1\delta}(2^j)^{2.9095+ 53.3\delta}N_{4,j,k}(u(0)) \leq \frac{C_0}{1000}\varepsilon^2.
\end{align}
Using \eqref{cc4} yields
\[
N_{4,j,k}(u(t))\leq 5\|\partial^{\leq1} (Zu)_{j,k}\|^2_{L^2}\leq 650C_2(2^k)^{-0.0905-0.1\delta}(2^j)^{-2.9095-53.3\delta}(1+t)^{\delta^2}.
\]
Similarly to Section \ref{section1}, for any $t>0$, if $\delta$ in $(0,0.001]$ is a fixed constant and $\varepsilon$ in $(0,\delta^6]$ is small enough, we get
\begin{align*}
M_{low4}(u(t))\leq 5\sum_{j=0}^{+\infty}\sum_{k=j+20}^{+\infty}\Big((2^j)^{3+30\delta}(2^{k})^{12\delta}N_{4,j,k}(u(t))\Big)\leq \frac{C_0}{32}\varepsilon^2.
\end{align*}

Combining the discussion above in Section \ref{sec_low}, we get
\[
M_{low}(u(t))\leq C_0\varepsilon^2.
\]

\subsection{$L^{\infty}$ estimates}\label{sec_infi}
In this section, we give the proof of inequalities \eqref{small1} and \eqref{small2}. For any solution to equation \eqref{NLG1}, since $M_{low}(u(t))\leq C_0\varepsilon^2$, using Lemma \ref{sob1} gives that
\begin{align*}
&(1+t)\|P_0\partial^{\leq2}u\|_{L^{\infty}}\leq (1+t)\sum_{k=0}^{+\infty}\|R_kP_0\partial^{\leq2}u\|_{L^{\infty}}\leq \frac{C\varepsilon}{\delta}+C\|P_0\partial_t\partial^{\leq2}u\|_{\dot{H}^{-1}},
\\&(1+t)\|\partial^{\leq2}u_j\|_{L^{\infty}}\leq \frac{C\varepsilon}{\delta}(2^j)^{-\delta},\ for\ j\geq1.
\end{align*}
To prove inequality \eqref{small1}, by using Lemma \ref{sob_lemma}, we get
\[
(1+t)^2\|\partial^{\leq2}u\|^2_{L^{\infty}}\leq C\|\partial_t\partial^{\leq2}u\|^2_{\dot{H}^{-1}(\mathbb R^3)}+\frac{C\varepsilon^2}{\delta^3}.
\]
Then to complete the proof of inequalities \eqref{small1} and \eqref{small2}, it suffices to estimate 
\[
\|\partial_tu\|_{\dot{H}^{-1}(\mathbb R^3)}+\ \|\partial^2_tu\|_{\dot{H}^{-1}(\mathbb R^3)}+\ \|\partial_t^3u\|_{\dot{H}^{-1}(\mathbb R^3)}\leq C\varepsilon/\delta.
\]
By using $L^{\frac65}\hookrightarrow \dot{H}^{-1}$, $\|fg\|_{L^{\frac65}}\leq \|f\|_{L^2}\|g\|_{L^3}$ and $\dot{H}^{\frac12}\hookrightarrow L^3$ in $\mathbb R^3$, we obtain
\begin{align}\label{inf1}
\|\partial^{\leq1}_t\partial^2_tu\|_{\dot{H}^{-1}}\leq \|\partial^{\leq1}_t(\Box u+u)\|_{\dot{H}^{-1}}+\|\Delta \partial^{\leq1}_tu\|_{\dot{H}^{-1}}+\|\partial^{\leq1}_tu\|_{\dot{H}^{-1}}
\leq C\varepsilon+\|\partial^{\leq1}_tP_0u\|_{\dot{H}^{-1}}.
\end{align}
In Section \ref{sec_infi}, let 
\[
   \widehat{v}=F(|\nabla|^{-1}P_0u)=\frac1{|\xi|}\widehat{P_0u}.
\]
Then we only need to estimate
\[
\|\partial_t^{\leq1}P_0u\|_{\dot{H}^{-1}}=\|\partial_t^{\leq1}v\|_{L^2}\leq C\varepsilon/\delta.
\]
By calculation, the function $v$ satisfies the following equation
\[
\Box v+v=|\nabla|^{-1}P_0(C^{i\alpha\gamma}\partial_{\gamma}u\partial_i\partial_{\alpha}u+C^{i\alpha}u\partial_i\partial_{\alpha}u+H(u)).
\]
We now carry out the ghost energy estimate for the equation above. Using \eqref{ghost1} yields that
\[
\frac d{dt}\int e^{\arctan(r-t)}|\partial^{\leq1}v|dx+\int e^{\arctan(r-t)}\frac1{1+(t-r)^2}\left(v^2+\left|(\partial_t+\partial_r)v\right|^2\right)dx\leq B_1+B_2+B_3,
\]
where
\begin{align*}
    &B_1:=\int_{\mathbb R^3}e^{\arctan(r-t)}\partial_tv|\nabla|^{-1}P_0(C^{i\alpha\gamma}\partial_{\gamma}u\partial_i\partial_{\alpha}u)dx,
   \\&B_2:=\int_{\mathbb R^3}e^{\arctan(r-t)}\partial_tv|\nabla|^{-1}P_0(H(u))dx,
   \\&B_3:=\int_{\mathbb R^3}e^{\arctan(r-t)}\partial_tv|\nabla|^{-1}P_0(C^{i\alpha}u\partial_i\partial_{\alpha}u)dx.
\end{align*}
\subsubsection{\textbf{Estimates for $B_1$}}
We first consider terms $B_1$. By integration by parts with respect to $t$, 
\begin{align*}
B_1\leq B_{1,1}+B_{1,2}+B_{1,3},
\end{align*}
where
\begin{align*}
&B_{1,1}=\frac{d}{dt}\int_{\mathbb R^3}e^{\arctan(r-t)}v|\nabla|^{-1}P_0(C^{i\alpha\gamma}\partial_{\gamma}u\partial_i\partial_{\alpha}u)dx,
\\&B_{1,2}=C\int_{\mathbb R^3}e^{\arctan(r-t)}\frac1{1+|t-r|^2}|v|\left||\nabla|^{-1}P_0(C^{i\alpha\gamma}\partial_{\gamma}u\partial_i\partial_{\alpha}u)\right|dx,
\\&B_{1,3}=-\int_{\mathbb R^3}e^{\arctan(r-t)}v|\nabla|^{-1}P_0\partial_t(C^{i\alpha\gamma}\partial_{\gamma}u\partial_i\partial_{\alpha}u)dx.
\end{align*}
If we integrate $B_{1,1}$ over $[0,t]$, terms $B_{1,1}$ can be easily controlled. Thus we omit the details. 

We now turn to $B_{1,2}$. Then
\[
    B_{1,2}\leq \frac1{32}\int e^{\arctan(r-t)}\frac1{1+(t-r)^2}v^2dx+B_{1,2,1},
\]
where
\[
B_{1,2,1}=C\left\||\nabla|^{-1}(C^{i\alpha\gamma}\partial_{\gamma}u\partial_i\partial_{\alpha}u)\right\|^2_{L^2}
\]
First, we carefully analyze the properties of the nonlinear term $\partial_{\gamma}u\partial_i\partial_{\alpha}u$ as follows.
\begin{itemize}
\item \textbf{Case 1}: If $\alpha=\gamma=0$, the identity $2\partial_{t}u\partial_i\partial_{t}u=\partial_i\left((\partial_t u)^2\right)$ holds.
\item \textbf{Case 2}: For the case $\alpha=j,\gamma=k$,  the identity $\partial_ku\partial_i\partial_ju=\mathcal{O}(1)\nabla(\nabla u\nabla u)$ holds.
\item \textbf{Case 3}: If $\alpha=0,\gamma=j$, using the identity $(1+t)\partial_j=\partial_j+L_j-x^j\partial_t$ gives
\[
\partial_ju\partial_t\partial_i u=\frac1{1+t}\left(\mathcal{O}(Zu\partial\nabla u)+\mathcal{O}\left(\nabla(x(\partial_tu)^2)\right)+\mathcal{O}(1)(\partial_tu)^2\right).
\]
\item \textbf{Case 4}: If $\alpha=j,\gamma=0$, then $\partial_tu\partial_i\partial_ju=\partial_i(\partial_tu\partial_ju)-\partial_ju\partial_t\partial_iu$, which is reduced to \textbf{Case 3}.
\end{itemize}
Thus we get
\begin{align}\label{inf2}
C^{i\alpha\gamma}\partial_{\gamma}u\partial_i\partial_{\alpha}u=\mathcal{O}(1)\nabla(\partial u\partial u)+\frac1{1+t}\left(\mathcal{O}(Zu\partial\nabla u)+\mathcal{O}\left(\nabla(x(\partial_tu)^2)\right)+\mathcal{O}(1)(\partial_tu)^2\right)
\end{align}
Using a similar fashion of discussion on \eqref{inf1} gives
\begin{align*}
B_{1,2,1}&\leq C\|\partial u\partial u\|_{L^2}^2+C(1+t)^{-2}\left(\||\nabla|^{-1}(Zu\nabla\partial u)\|^2_{L^2}+\||x|(\partial_{t}u)^2\|^2_{L^2}+\||\nabla|^{-1}(\partial_tu)^2\|^2_{L^2}\right)
\\&\leq C\varepsilon^2(1+t)^{-2}\left(\frac{\varepsilon^2}{\delta^3}+\|\partial_t^{\leq1}v\|^2_{L^2}\right)+\frac{C\varepsilon^4}{\delta^4}(1+t)^{-2}.
\end{align*}
Thus terms $B_{1,2,1}$ are acceptable.  

We next move on to $B_{1,3}$. Using \eqref{inf2} shows that
\[
B_{1,3}\leq B_{1,3,1}+B_{1,3,2}+B_{1,3,3},
\]
where
\begin{align*}
&B_{1,3,1}=\mathcal{O}(1)\int_{\mathbb R^3}e^{\arctan(r-t)}v|\nabla|^{-1}\nabla P_0\partial_t(\partial u\partial u)dx,
\\&B_{1,3,2}=\mathcal{O}(1)\frac1{1+t}\int_{\mathbb R^3}e^{\arctan(r-t)}v\left(|\nabla|^{-1}\nabla P_0\left(x\partial_t((\partial_t u)^2)\right)\right)dx,
\\&B_{1,3,3}=C(1+t)^{-1}\|v\|_{L^2}\left(\|\partial_tP_0(Zu\partial\nabla u)\|_{L^{\frac65}}+\|P_0(\partial u\partial^2u)\|_{L^{\frac65}}\right).
\end{align*}
Using the estimate $\|f\|_{L^3}\leq (\|f\|_{L^2})^{2/3}(\|f\|_{L^{\infty}})^{1/3}$ yields
\begin{align*}
\|\partial_tP_0(Zu\partial\nabla u)\|_{L^{\frac65}}\leq &C\sum_{j_1=0}^{+\infty}\sum_{j_2=j_1-2}^{j_1+2}\left(\|\partial_t(Zu)_{j_1}\|_{L^2}\|\partial\nabla u_{j_2}\|_{L^3}+\|(Zu)_{j_1}\|_{L^2}\|\partial_t\partial\nabla u_{j_2}\|_{L^3}\right)
\\ \leq& C\varepsilon^{\frac53}(1+t)^{-\frac13}\left(\frac{C\varepsilon}{\delta}+\|\partial_t^{\leq1}v\|_{L^2}\right)^{\frac13}.
\end{align*}
Similar discussion holds for $\|P_0(\partial u\partial^2u)\|_{L^{\frac65}}$. Thus terms $B_{1,3,3}$ are fine. For terms $B_{1,3,1}$, by using Lemma \ref{product_KG} with $w_1=\partial u_{j_1},w_2=\partial u_{j_2}$, we only need to estimate
\begin{align*}
B_{1,3,1}\leq C\sum_{j_1=0}^{+\infty}\sum_{j_2=j_1-2}^{j_1+2}B_{1,3,1,1,j_1,j_2}+good\ terms,
\end{align*}
where by "good terms", we mean terms that are similar to $B_{1,1},B_{1,2}$ or the integral of fourth-order nonlinear term, and 
\[
B_{1,3,1,1,j_1,j_2}=\left|\int e^{\arctan(r-t)}v\partial_t\left(|\nabla|^{-1}\nabla P_0\left(Q_{j_1,j_2}\right)\right)dx\right|,
\]
with 
\[
Q_{j_1,j_2}=\partial_t\partial u_{j_1}\partial_t\partial u_{j_2}
-\nabla\partial u_{j_1}\cdot\nabla\partial u_{j_2}+\partial u_{j_1}\partial u_{j_2}. 
\]
Using the identity
\[
(1+t+r)(1+|t-r|)|\partial_tf|\leq (1+t+r)|\partial_tf|+|t-r||\partial_tf|+|(t^2-r^2)\partial_tf|
\]
and identity \eqref{cc5} yields 
\begin{align}\label{combin_qq2}
&B_{1,3,1,1,j_1,j_2}
\nonumber\\ \leq &\frac1{32}\int e^{\arctan(r-t)}\frac1{1+(t-r)^2}v^2dx+
\frac{C}{1+t}\|v\|_{L^2}\left\|\partial_t(\partial^{\leq1}\partial u_{j_1}\partial^{\leq1}\partial u_{j_2})\right\|_{L^2}
\nonumber\\&+C\left\|\partial_t(\partial^{\leq1}\partial u_{j_1}\partial^{\leq1}\partial u_{j_2})\right\|_{L^2}^2+C\left\|S|\nabla|^{-1}\nabla P_0(Q_{j_1,j_2})\right\|_{L^2}^2+C\left\|L|\nabla|^{-1}\nabla P_0(Q_{j_1,j_2})\right\|_{L^2}^2.
\end{align}
Carrying out Fourier transform, we get
\begin{align}\label{inf3}
\|S(|\nabla|^{-1}\nabla f)-|\nabla|^{-1}\nabla Sf\|_{L^2}\leq C\|f\|_{L^2},\qquad \|L(|\nabla|^{-1}\nabla f)-|\nabla|^{-1}\nabla Lf\|_{L^2}\leq C\|\partial_tf\|_{\dot{H}^{-1}}.
\end{align}
Therefore, the last two terms on the right side of \eqref{combin_qq2} can be estimated as follows
\begin{align*}
&\left\|S|\nabla|^{-1}\nabla P_0(Q_{j_1,j_2})\right\|_{L^2}^2\leq C\|S(Q_{j_1,j_2})\|_{L^2}^2+C\|Q_{j_1,j_2}\|_{L^2}^2,
\\&\left\|L|\nabla|^{-1}\nabla P_0(Q_{j_1,j_2})\right\|_{L^2}^2\leq C\|L(Q_{j_1,j_2})\|_{L^2}^2+C\|\partial_tQ_{j_1,j_2}\|_{\dot{H}^{-1}}^2
\end{align*}
By using \eqref{kg_pro1}, terms $\|S(Q_{j_1,j_2})\|_{L^2}^2$ are acceptable. Employing
\[
(1+t)\partial_tf=\partial_tf+Sf-\partial_i(x^if)+3f,
\]
and using \eqref{kg_pro1} yields that terms $\|\partial_tQ_{j_1,j_2}\|_{\dot{H}^{-1}}^2$ are acceptable. Thus terms $B_{1,3,1}$ are fine. Terms $B_{1,3,2}$ are similar to $B_{1,3,1}$.  Thus terms $B_{1,3}$ are admissible and then terms $B_1$ are acceptable.

\subsubsection{\textbf{Estimates for $B_2$}}
We now consider terms $B_2$. We first observe that $u\partial_{\alpha}u=\frac12\partial_{\alpha}(u^2)$ in $H(u)$. For the case $\alpha\neq0$, these terms enjoy better properties than $B_1$. So it suffices to focus on the case $\alpha=0$. By integration by parts with respect to $t$, we only need to estimate the following terms: 
\begin{align*}
&B_{2,1}=\mathcal{O}(1)\frac{d}{dt}\int_{\mathbb R^3}e^{\arctan(r-t)}v|\nabla|^{-1}P_0\partial_t(u^2)dx,
\\&B_{2,2}=C\int_{\mathbb R^3}\frac1{1+|t-r|^2}|v|\left||\nabla|^{-1}P_0(u\partial_tu)\right|dx,
\\&B_{2,3}=-\mathcal{O}(1)\int_{\mathbb R^3}e^{\arctan(r-t)}v|\nabla|^{-1}P_0\partial_t^2(u^2)dx.
\end{align*}
Terms $B_{2,1}$ can be readily handled and we omit the details. We now move on to $B_{2,2}$. Applying Lemma \ref{Hardy2} with $s=\frac12-\varepsilon_0$, where $\varepsilon_0$ is a small positive constant to be chosen later, we obtain
\begin{align*}
\int \frac{1}{1+(t-r)^2}||\nabla|^{-1}P_0(u\partial_tu)|^2dx\leq C\|u\partial_tu\|^2_{\dot{H}^{-\frac12-\varepsilon_0}(\mathbb R^3)}.
\end{align*}
Since $\dot{H}^{s}\hookrightarrow L^q$ with $\frac 1q=\frac12-\frac sn$ in $\mathbb R^3$, then $L^p\hookrightarrow \dot{H}^{-s}$, with $\frac1p-\frac sn=\frac12$. Let $s=0.5+\varepsilon_0$. Thus $p=3/(2+\varepsilon_0)$ and
\begin{align*}
&\quad\int \frac{1}{1+(t-r)^2}\left||\nabla|^{-1}P_0(u\partial_tu)\right|^2dx
\\&\leq C\left(\int |u\partial_tu|^{\frac3{2+\varepsilon_0}}dx\right)^{\frac{4+2\varepsilon_0}3}
\\&\leq C\left(\|u\|_{L^{\infty}}^{\frac{1-\varepsilon_0}{2+\varepsilon_0}}\|\partial_tu\|_{L^{\infty}}^{\frac{1-\varepsilon_0}{2+\varepsilon_0}}\int |u\partial_tu|dx\right)^{\frac{4+2\varepsilon_0}3}
\\&\leq C(1+t)^{-\frac{4(1-\varepsilon_0)}{3}}\left((1+t)^2\|u\|_{L^{\infty}}\|\partial_tu\|_{L^{\infty}}\right)^{\frac{2(1-\varepsilon_0)}{3}}\left(\|u\|_{L^2}\|\partial_tu\|_{L^2}\right)^{\frac{4+2\varepsilon_0}3}.
\end{align*}
Now we fix $\varepsilon_0=0.1$, so that the terms $B_{2,2}$ are admissible. It remains to estimate $B_{2,3}$. Applying Lemma \ref{product_KG}, 
\[
-3u^2=(\Box+1)(u^2)-2u(\Box u+u)-2((\partial_tu)^2-|\nabla u|^2+u^2).
\]
For the part $(\Box+1)(u^2)$, integration by parts shows that it can be treated similarly to $B_{2,1},B_{2,2}$. For the part $u(\Box u+u)$, by H\"older inequality and  Sobolev embedding theorem $L^{\frac65}\hookrightarrow \dot{H}^{-1}$, we have
\begin{align*}
\left|\int |\nabla|^{-1}P_0\partial_t^2(u(\Box u+u))e^{\arctan(r-t)}vdx\right| \leq C\|P_0\partial_t^2(u\partial^{\leq1} u\partial^{\leq2} u)\|_{L^{\frac65}}\|v\|_{L^2},
\end{align*}
for which we can obtain $(1+t)^{-\frac43}$ decay. Hence all these terms are acceptable. For the part $(\partial_tu)^2-|\nabla u|^2+u^2$, integrating by parts in time $t$ reduces the problem to estimating the following terms
\begin{align*}
\mathcal{O}(1)\int \nabla^{-1}P_0\partial_t\left((\partial_tu)^2-|\nabla u|^2+u^2\right)e^{\arctan(r-t)}\partial_tvdx.
\end{align*}
Since $\partial_t=\frac{1}{1+t}(\partial_t+S-x^i\partial_i)$, the part $\frac{1}{1+t}x^i\partial_i$ can be treated similarly to $B_{1,3,2}$ via integration by parts in time. It remains to estimate
\begin{align*}
B_{2,3,1}:=&\frac{C}{1+t}\left|\int |\nabla|^{-1}P_0S\left((\partial_tu)^2-|\nabla u|^2+u^2\right)e^{\arctan(r-t)}\partial_tvdx\right|
\\ \leq &\frac{C}{1+t}\left\|S\left((\partial_tu)^2-|\nabla u|^2+u^2\right)\right\|_{L^{\frac65}}\|\partial_tv\|_{L^2}.
\end{align*}
By using inequality \eqref{kg_pro1} (Let $w_1=w_2=u$.) and $|Sf|\leq (t+r)|\partial f|$, they are acceptable. Thus the part $u\partial_{\alpha}u$ in $H(u)$ is fine.

For terms $\partial_{\alpha}u\partial_{\beta}u$ in $H(u)$, by making a little modification of the proof of lemma \ref{product_KG}, we have
\begin{align*}
&\quad\partial_tP_0(\partial_{\alpha}u\partial_{\beta}u)
\\&=\sum_{j_1=0}^{+\infty}\sum_{j_2=j_1-2}^{j_1+2}\mathcal{O}(1)\partial_tP_0(\partial_t\partial_{\alpha}u_{j_1}\partial_t\partial_{\beta}u_{j_2}-\nabla\partial_{\alpha}u_{j_1}\cdot\nabla\partial_{\beta}u_{j_2}
-\partial_{\alpha}u_{j_1}\partial_{\beta}u_{j_2})+good\ terms
\\&=\frac1{1+t}\sum_{j_1=0}^{+\infty}\sum_{j_2=j_1-2}^{j_1+2}P_0\left(\mathcal{O}(1)(S-x^i\partial_i)(\partial_t\partial_{\alpha}u_{j_1}\partial_t\partial_{\beta}u_{j_2}
-\nabla\partial_{\alpha}u_{j_1}\cdot\nabla\partial_{\beta}u_{j_2}-\partial_{\alpha}u_{j_1}\partial_{\beta}u_{j_2})\right)
\\&\quad+good\ terms.
\end{align*}
For the part involving $S$, we have
\[
\frac1{1+t}P_0\left(S(\partial_t\partial_{\alpha}u_{j_1}\partial_t\partial_{\beta}u_{j_2}
-\nabla\partial_{\alpha}u_{j_1}\cdot\nabla\partial_{\beta}u_{j_2}+\partial_{\alpha}u_{j_1}\partial_{\beta}u_{j_2})\right)-\frac2{1+t}P_0\left(S(\partial_{\alpha}u_{j_1}\partial_{\beta}u_{j_2})\right),
\]
where the second term can be dealt with by using lemma \ref{product_KG} again. Now we focus on
\begin{align*}
\frac{x^i}{1+t}\left(\partial_i(\partial_t\partial_{\alpha}u_{j_1}\partial_t\partial_{\beta}u_{j_2}
-\nabla\partial_{\alpha}u_{j_1}\cdot\nabla\partial_{\beta}u_{j_2}-\partial_{\alpha}u_{j_1}\partial_{\beta}u_{j_2})\right).
\end{align*}
A straight calculation shows that 
\begin{align*}
&\quad\partial_t\partial_{\alpha}u_{j_1}\partial_t\partial_{\beta}u_{j_2}
-\nabla\partial_{\alpha}u_{j_1}\cdot\nabla\partial_{\beta}u_{j_2}-\partial_{\alpha}u_{j_1}\partial_{\beta}u_{j_2}
\\&=\partial_{\beta}(\partial_t\partial_{\alpha}u_{j_1}\partial_tu_{j_2})-\partial_t(\partial_{\alpha}\partial_{\beta}u_{j_1}\partial_tu_{j_2})
-\partial_{\beta}(\partial_j\partial_{\alpha}u_{j_1}\partial_ju_{j_2})+\partial_j(\partial_{\alpha}\partial_{\beta}u_{j_1}\partial_ju_{j_2})-\partial_{\beta}(\partial_{\alpha}u_{j_1}u_{j_2})
\\&\quad+\partial_{\alpha}\partial_{\beta}u_{j_1}(\Box u_{j_2}+u_{j_2}),
\end{align*}
where the divergence part can be dealt with by using a similar fashion of discussion for terms $B_{1,3,2}$ and the remainder term forms cubic nonlinear terms. Therefore the part $\partial_{\alpha}u\partial_{\beta}u$ in $H(u)$ is fine. 

For the term $u^2$ in $H(u)$, by using a similar fashion of lemma \ref{product_KG}, 
\[
\partial_t(u^2)=good \ terms +\partial_t((\partial_tu)^2-|\nabla u|^2)=good \ terms +\mathcal{O}(1)\partial_t(\partial_{\alpha}u\partial_{\beta}u),
\]
it's reduced to the above case. Therefore terms $B_2$ are fine.

\subsubsection{\textbf{Estimates for $B_3$}}
Applying the identity $u\partial_{\alpha}\partial_{i}u=\partial_{i}(u\partial_{\alpha}u)-\partial_{i}u\partial_{\alpha}u$ to the $B_3$ terms, in conjunction with the foregoing discussion, shows that terms $B_3$ are well-behaved.

The proofs of \eqref{small1} and \eqref{small2} are now complete, which finishes the proof of Theorem \ref{main_theorem}.

\section*{\textbf{Acknowledgement}}
Wei Xu is partially supported by the Doctoral Initiation Fund of Nanchang Hangkong University (No. EA202207232) and NSFC (No. 12661042 and 12461043). Yi Zhou is partially supported by NSFC (No. 12571231 and 12171097).

\textbf {Statements and Declarations}
\begin{itemize}
\item \textbf{Competing Interests:} The authors state that there is no conflict of interest.
    
\item \textbf{Data availability statements:} No data was used for the research described in this article.
\end{itemize}


\begin{appendix}
\section{\textbf{Auxiliary results}}\label{appendix}
\subsection{Standard Littlewood-Paley operator on $\mathbb R^3$}
Applying the Fourier transform, we achieve the commutation property between the operators $P_k$ and the vector fields.
\begin{lemma}\label{LP_commu}
For any function $f(t,x)$, we get
\[
[\partial,P_k]=0, \ [\Omega,P_k]=0,\ [L_j,P_k]f=\widetilde{P}_{j,k}\partial_tf,
\]
where $\widehat{\widetilde{P}_{j,k}f}=i\partial_{\xi_j}\left(\chi_k(\xi)\right)\widehat{f}$.
\end{lemma}

Employing the Fourier transform and using the support of $\chi_i(x),\overline{\chi}_i(x)$ yields that the lemma stated below holds.
\begin{lemma}\label{LP2}
For any functions $f,g$ and any positive integers $j,l$ with $|j-l|\geq 3$, we get
\begin{align}
     \label{van1}&P_jP_lf\equiv0,
     \\ \label{van2} &P_l(P_{\leq j-3}fP_jg)\equiv0.
\end{align}
In addition, whenever $\max\{j,k,l\}\geq med\{j,k,l\}+3$, the product projection vanishes: 
\begin{align}\label{van3}
     P_l(P_jfP_kg)=0.
\end{align}
\end{lemma}

We rely on the lemma below to resolve the derivative loss issue arising in our estimates. 
\begin{lemma}\label{para-product2}
For any scalar functions $f,g$ and any non-negative integer $l$, the commutator satisfies
\begin{align*}
[P_l,f] g:=P_l(f g)-fP_lg=D_1+D_2,
\end{align*}
where the terms $D_1,D_2$ are defined by
\begin{align*}
&D_1=\sum_{j=0}^{+\infty}P_l\left(\left(P_{\leq j+2} g\right)P_jf\right)-\sum_{j=0}^{+\infty}\left(P_{\leq j+2}P_l g\right)P_jf,
\\&D_2=2^{-l}\sum_{j=l-2}^{l+2}L(P_{\leq j-3}f,P_j g).
\end{align*} 
Here the bilinear operator $L(\tilde{f},\tilde{g})$ is given by
\[
L(\tilde{f},\tilde{g})=-\int_0^1\int_{\mathbb R^3}\check{\chi}(2^ly)(2^l)^3(2^ly)\cdot\left((\nabla \tilde{f})(x-sy)\right)\tilde{g}(x-y)dyds,
\]
where $\check{\chi}$ denotes the inverse Fourier transform of $\chi(x)$. 
\end{lemma}
\begin{proof}
By calculation, we get the following Bony para-product decomposition
\begin{align*}
&fg=\sum_{j=0}^{+\infty}P_{\leq j-3}fP_jg+\sum_{j=0}^{+\infty}\left(P_{\leq j+2} g\right)P_jf,
\\&fP_lg=\sum_{j=0}^{+\infty}P_{\leq j-3}fP_jP_l g+\sum_{j=0}^{+\infty}\left(P_{\leq j+2}P_lg\right)P_jf.
\end{align*}
Lemma \ref{LP2} allows us to only consider
\[
   P_l\left(\sum_{j=l-2}^{l+2}P_j gP_{\leq j-3}f\right)-\sum_{j=l-2}^{l+2}\left(P_lP_j g\right)P_{\leq j-3}f.
\]
Let $\tilde{f}=P_{\leq j-3}f,\ \tilde{g}=P_j g$. Then 
\begin{align}\label{comm1}
P_l(\tilde{f}\tilde{g})-\tilde{f}P_l\tilde{g}&=\int_{\mathbb R^3}\check{\chi_l}(y)\tilde{f}(x-y)\tilde{g}(x-y)dy-\tilde{f}(x)\int_{\mathbb R^3}\check{\chi_l}(y)\tilde{g}(x-y)dy
\nonumber\\&=\int_{\mathbb R^3}\check{\chi}(2^ly)(2^l)^3\left(\tilde{f}(x-y)-\tilde{f}(x)\right)\tilde{g}(x-y)dy
\nonumber\\&=\int_{\mathbb R^3}\check{\chi}(2^ly)(2^l)^3\int_0^1\frac{d}{ds}\left(\tilde{f}(x-sy)\right)\tilde{g}(x-y)dsdy
\nonumber\\&=-\int_0^1\int_{\mathbb R^3}\check{\chi}(2^ly)(2^l)^3\left(y\cdot(\nabla \tilde{f})(x-sy)\right)\tilde{g}(x-y)dyds
\nonumber\\&=-2^{-l}\int_0^1\int_{\mathbb R^3}\check{\chi}(2^ly)(2^l)^3\left((2^ly)\cdot(\nabla \tilde{f})(x-sy)\right)\tilde{g}(x-y)dyds.
\end{align}
This concludes the proof.
\end{proof}

\subsection{Littlewood-Paley decomposition on the sphere}\label{sphere}

We proceed to present an alternative equivalent definition for the operator $R_k$. To this end, we first recall a fundamental property of spherical harmonics known as the addition theorem. 
\begin{theorem}[Addition Theorem]\label{thm:addition}
    Let \(\{\varphi_{l,m}:-l\leq m\leq l\}\) be an orthonormal basis of spherical harmonics of degree $l$ on the sphere $S^2$, satisfying the orthogonality relation
    \[
        \int_{\mathbb{S}^{2}} \varphi_{l,m}(\theta,\phi) \overline{\varphi_{l^{'},m^{'}}(\theta,\phi)} \,d_{S^{2}} =  \delta_{ll^{'}}\delta_{mm^{'}}, \quad -l \leq m \leq l,\ -l^{'} \leq m^{'} \leq l^{'}.
    \]
    Then the identity
    \[
        \sum_{m=l}^{-l} \varphi_{l,m}(\theta,\phi) \overline{\varphi_{l,m}}(\theta^{'},\phi^{'})
        = \frac{2l+1}{4\pi} \, Q_{l}(\cos\gamma),
    \]
    holds, where $Q_l$ stands for the Legendre polynomial of order $l$: $Q_0(z)=1$ and for each integer $l\geq1$,
    \[
      Q_l(z)=\frac{1}{2^ll!}\frac{d^l}{dz^l}\left((z^2-1)^l\right),\ |z|\leq 1.
    \] 
    The angle $\gamma$ between the two spherical directions satisfies 
    \[
    \cos\gamma=\cos\theta\cos\theta^{'}+\sin\theta\sin\theta^{'}\cos(\phi-\phi^{'}).
    \]
\end{theorem}

\begin{proof}
A complete proof is provided in \cite[Theorem 2.9]{Atkinson}, so we omit the details.
\end{proof}

 Using Theorem \ref{thm:addition} yields that for all $k\geq 1$,
\begin{align}\label{ALP_def}
   (R_kf)(x)=\sum_{l=2^{k-1}}^{2^{k}-1}\frac{2l+1}{4\pi}\int_{S^2}f(|x|\xi)Q_l\left(\langle \xi,\frac{x}{|x|}\rangle\right)d_{S^2(\xi)}.
\end{align}

In polar coordinates \eqref{polar}, applying the angular partial derivatives $\partial_{\theta},\partial_{\phi}$ yields the explicit expressions for $\widetilde{\Omega}_{ij}$. Consequently, the norm $\|\Omega R_kf\|^2_{L^2(S^2)}$ thereby transforms into an integral over the variables $\theta,\phi$. Finally, integrating by parts, we obtain
\begin{align*}
\int_{S^2}|\Omega R_kf|^2d_{S^2}=-\int_{S^2}\left(\overline{R_kf}(\Delta_{S^2}R_kf)\right)d_{S^2}.
\end{align*}
Using the definition \eqref{Rk_def} of $R_k$ yields the following lemma.
\begin{lemma}\label{equal_omiga}
For $k\geq1$ and any scalar function $f$ defined on the sphere $S^2$, we have
\begin{align*}
\|\Omega R_kf\|^2_{L^2(S^2)}=\sum_{l=2^{k-1}}^{2^{k}-1}\sum_{m=-l}^l\mu_l|\langle f,\overline{\varphi_{l,m}}\rangle|^2.
\end{align*}
\end{lemma}

The combination of the definition \eqref{Rk_def} of $R_k$, the Cauchy–Schwarz inequality, the identity
\[
\sum_{m=-l}^l|\varphi_{l,m}|^2=\frac{2l+1}{4\pi}Q_l(1)=\frac{2l+1}{4\pi},
\]
and Lemma \ref{equal_omiga} gives the following Sobolev embedding inequality on the unit sphere \(S^2\).
\begin{lemma}\label{ALP1}
For $k\geq1$, we have
\[
\|R_kf\|_{L^{\infty}(S^2)}\leq C\|\Omega R_kf\|_{L^2(S^2)}.
\]
\end{lemma}

We now give the commutation relations involving $R_k$, $P_j$, the vector fields and Fourier transform.
\begin{lemma}\label{ALP_commu}
The operator $R_k$ commutes with $\partial_t,\partial_r,\Omega,P_j$ and the Fourier transform $F$; that's
\[
[\Omega,R_k]=[\partial_r,R_k]=[\partial_t,R_k]=[R_k,P_j]=[R_k,F]=0.
\]
\end{lemma}
\begin{proof}
See \cite[Proposition 3.1]{Guo}.
\end{proof}

Next let's present the standard selection rule for the triple product of spherical harmonics.
\begin{lemma}\label{three}
For $\varphi_{l,m}$ defined previously, the integral
\[
\int_{S^2}\varphi_{l_1,m_1}\varphi_{l_2,m_2}\overline{\varphi_{l_3,m_3}}d_{S^2}
\]
can be nonzero only if
\[
m_3=m_1+m_2,\ |l_1-l_2|\leq l_3\leq l_1+l_2.
\]
\end{lemma}

\begin{proof}
By using identity (9) given in Section 5.6.2 of \cite{Varshalovich}, we get
\[
\varphi_{l_1,m_1}\varphi_{l_2,m_2}=\sum_{L,M}\sqrt{\frac{(2l_1+1)(2l_2+1)}{4\pi(2L+1)}}C^{L0}_{l_10l_20}C^{LM}_{l_1m_1l_2m_2}\varphi_{L,M}.
\]
In section 8.1.1 of \cite{Varshalovich}, Clebsch-Gordan coefficients $C^{LM}_{l_1m_1l_2m_2}$ vanish unless 
\[
m_1+m_2=M,\ |l_1-l_2|\leq L\leq l_1+l_2.
\]
By using the orthogonality of $\varphi_{l,m}$, we conclude the proof. 
\end{proof}

Since $x/r$ is linear combination of $\varphi_{1,m}$, $m=-1,0,1$, using the definition \eqref{Rk_def} for $R_{k_1}g$ and applying Lemma \ref{three} gives the following lemma, concerning the commutation of $R_k$ and $\nabla$.
\begin{lemma}\label{xw1}
For $k\geq1$ and any scalar function $g$,
\begin{align}\label{one1}
    R_k\left(\frac xrg\right)=\sum_{k_1=k-1}^{k+1}R_k\left(\frac xrR_{k_1}g\right).
\end{align}
\end{lemma}

Using the definition \eqref{Rk_def} of $R_k$ and applying Lemma \ref{three}, we establish a counterpart of Lemma \ref{LP2} in the setting of the Littlewood–Paley decomposition on the sphere.
\begin{lemma}\label{ALP_ortho}
If $j\geq1$, $l\geq1$ and $|j-l|\geq 1$, for any scalar functions $f,g$, we have
\begin{align*}
     &R_jR_lf\equiv0.
\end{align*}
If $\max\{k,k_1,k_2\}\geq med\{k,k_1,k_2\}+2$, it holds that
\[
     R_k(R_{k_1}fR_{k_2}g)=0.
\]
\end{lemma}

The following lemma plays a key role in estimating $M_{high2}(u(t))$ and $M_{high4}(u(t))$.
\begin{lemma}\label{diff1}
Let
\[
\widetilde{Q}_k(z)=\sum_{l=2^{k-1}}^{2^{k}-1}\frac{2l+1}{4\pi}Q_l(z).
\]
Then for any scalar functions $f,g$,
\begin{align*}
\|R_k(fg)-fR_kg\|_{L^2(S^2)}
=&\left\|\int_{S^2}\left(f(t,r\xi)-f(t,r\eta)\right)g(t,r\xi)\widetilde{Q}_k(\langle \xi,\eta\rangle)d_{S^2(\xi)}\right\|_{L^2(S^2(\eta))}
\\ \leq &C2^{-k}\|\Omega f\|_{L^{\infty}(S^2)}\|g\|_{L^2(S^2)}.
\end{align*}
\end{lemma}
\begin{proof}
This lemma follows directly from the proof of Lemma 2.4 in Zhang and Zhou \cite{ZhangZhou}; see also Appendix A.1 in \cite{Guo}.
\end{proof}

\end{appendix}

\end{document}